\documentclass[11pt]{article}
\usepackage{amssymb,amsmath,amsthm,hyperref,graphicx,color,array,fullpage,mathrsfs}
\usepackage[notref,notcite,color,final]{showkeys}
\usepackage{amscd}
\usepackage{mathrsfs}
\usepackage{makeidx}
\usepackage[OT2,T1]{fontenc}
\DeclareSymbolFont{cyrletters}{OT2}{wncyr}{m}{n}
\DeclareMathSymbol{\Sha}{\mathalpha}{cyrletters}{"58}
\makeindex

\newtheorem{theorem}{Theorem}[section]
\newtheorem{lemma}[theorem]{Lemma}
\newtheorem{corollary}[theorem]{Corollary}
\newtheorem{conjecture}[theorem]{Conjecture}
\newtheorem{definition}[theorem]{Definition}

\newtheorem{remark}[theorem]{Remark}
\newtheorem{proposition}[theorem]{Proposition}

\begin{document}
\title{\textsc{Iwasawa theory for $\mathrm{U}(r,s)$, Bloch-Kato conjecture and Functional Equation}}

\author{\textsc{Xin Wan}}
\date{}
\maketitle
\begin{abstract}
In this paper we develop a new method to study Iwasawa theory and Eisenstein families for unitary groups $\mathrm{U}(r,s)$ of general signature over a totally real field $F$. As a consequence we prove that for a motive corresponding to a regular algebraic cuspidal automorphic representation $\pi$ on $\mathrm{U}(r,s)_{/F}$ which is ordinary at $p$, twisted by a Hecke character, if its Selmer group has rank $0$, then the corresponding central $L$-value is nonzero. This generalizes a result of Skinner-Urban in their ICM 2006 report in the special case when $F=\mathbb{Q}$ and the motive is conjugate self-dual. Along the way we also obtain $p$-adic functional equations for the corresponding $p$-adic $L$-functions and $p$-adic families of Klingen Eisenstein series. Our method does not involve computing Fourier-Jacobi coefficients (as opposed to previous work which only work in low rank cases, e.g. $\mathrm{U}(1,1)$, $\mathrm{U}(2,0)$ and $\mathrm{U}(1,0)$) whose automorphic interpretation is unclear in general.
\end{abstract}
\section{Introduction}
One of the central problems in number theory is to study the relation between special values of $L$-functions and arithmetic objects. A classical example is the class number formula for number fields, relating the residue of the Dedekind zeta function of a number field $K$ at $s=1$ to the class number $h_K$ of it. Another famous example is the Birch-Swinnerton-Dyer conjecture about the relation between $L$-functions of elliptic curves and its arithmetic invariants (Shafarevich-Tate group and Mordell-Weil group).

This philosophy had been generalized by Bloch-Kato \cite{BK} to general ``motives'' in an explicit formulation, which we briefly recall. Let $M$ be a motive with a (hypothetical) $L$-function denoted as $L(M,s)$. In particular it has a $p$-adic realization which is a Galois representation of $G_\mathbb{Q}$, potentially semi-stable in the sense of Fontaine. In favorable cases, this converges to a holomorphic function when $\mathrm{Re}(s)>>0$, has analytic continuation to the complex plane and satisfies a functional equation. In practice this is often ensured by relating $M$ to automorphic representations. Suppose the functional equation is
$$L(M,s)=L(M^\vee,1-s),$$
(the center is $s=\frac{1}{2}$, which we suppose is a critical value in the sense of Deligne.) In this paper we only consider central $L$-values, although Bloch-Kato theory is valid for all critical values, since normally only central values can vanish.
On the arithmetic side, Bloch-Kato defined a $p$-adic Selmer group $\mathrm{Sel}_{p^\infty}(M)$ of $M$ using the degree one Galois cohomology of $M$ satisfying certain local Selmer conditions (using Fontaine's $p$-adic Hodge theory). Then the Bloch-Kato conjecture is the following.
\begin{conjecture}
The vanishing order of $L(M^\vee(1),s)$ at $s=\frac{1}{2}$ equals the rank of the Selmer group $\mathrm{Sel}_{p^\infty}(M)$.
\end{conjecture}
In this paper, we consider the motive considered by M. Harris in the book project, which is associated to a cusp form on unitary groups, twisted by a Hecke character. Thus the Langlands correspondence and its properties are well understood. Let $\mathcal{K}$ be a CM field which is quadratic over its maximal totally real subfield $F$. Suppose $p$ is an odd prime splitting completely in $\mathcal{K}$. (This is just for simplicity: we expect no difficulty only assuming $p$ to be unramified in $F$ and split in $\mathcal{K}/F$.) In this paper we consider a cuspidal automorphic representation $\pi$ of $\mathrm{U}(r,s)_{/F}$ which is unramified and ordinary at all primes above $p$ (we make it precise in the text). We focus on motives $M$ associated to $\pi$, twisted by a Hecke character of the corresponding CM field $\mathcal{K}$. We write $n=r+s$ and let $L$ (finite extension of $\mathbb{Q}_p$) be the coefficient field of its $p$-adic Galois representation.
Throughout this paper we assume
\begin{itemize}
\item[(Irred)] The residual Galois representation $\bar{M}$ of $G_\mathcal{K}$ is absolutely irreducible.
\end{itemize}

In Skinner-Urban's 2006 ICM report \cite{SU2}, they proved
\begin{theorem}(Skinner-Urban)
Suppose $F=\mathbb{Q}$, (Irred) and $M$ satisfies $M^c\simeq M^\vee(1)$. Suppose moreover that $0$ and $1$ are not Hodge-Tate weights of $M$. If $L(M^\vee(1),\frac{1}{2})=0$ ,then the Selmer group $\mathrm{Sel}_{p^\infty}(M)$ has positive rank.
\end{theorem}
Skinner-Urban also proved that if moreover the global sign of $M$ is $+1$, then the rank of the corresponding Selmer group is at least two.

In the above theorem, if one assumes in addition that the global sign of $M$ is $-1$, then this result is also proved by Bellaiche-Chenevier in \cite{BC} by a different approach.

One main result of our paper is to remove the assumption $F=\mathbb{Q}$, and $M$ satisfies $M^c\simeq M^\vee(1)$ of the above theorem. We need the following
\begin{itemize}
\item[(QS)] We assume that for each prime $v$ of $F$, if $L_v(\tilde{\pi}, \bar{\tau}^c, -\frac{1}{2})$ has a pole, then $\mathrm{U}(r,s)(F_v)$ is quasi-split.
\end{itemize}

\begin{theorem}\label{main theorem}
Assume (QS), (Irred) and that $0$ and $1$ are not Hodge-Tate weights of $M$. Suppose moreover $\pi$ is unramified and ordinary at all primes above $p$. If $L(M^\vee(1),0)=0$, then the rank of the Selmer group $\mathrm{Sel}_{p^\infty}(M)$ is positive.
\end{theorem}
The assumption (QS) is put to apply theory of Casselman-Shahidi on intertwining operators to prove non-vanishing of certain $p$-adic limit of some Eisenstein family. They only proved results in quasi-split cases, although they are expected to be true in general. It also seems plausible to allow general finite slope cases instead of just the ordinary cases. But this requires a lot more work (e.g. construct finite slope Eisenstein families using the techniques in the recent work of Andretta-Iovita \cite{AI}, and triangulations of Galois representations along local eigenvarieties), which we leave to the future.
\begin{remark}
The assumption on the Hodge-Tate weight is to ensure certain Eisenstein series has classical weight as needed. It seems difficult to get $p$-integral results of the Bloch-Kato conjecture in this generality. Even the  formulation (e.g. defining the appropriate period) seems quite non-trivial.
\end{remark}
\begin{remark}
Our construction uses results in \cite{KMSW} about Arthur conjectures. We need to know that if the base change of $\pi$ to $\mathrm{GL}(n)_{/\mathcal{K}}$ is cuspidal, then $\pi$ appears in the space of cusp forms of $\mathrm{U}(r,s)$ with multiplicity one. We also use the local-global compatibility of this base change map. As explained in the introduction of \emph{loc.cit.}, at the moment these depend on ongoing work of Moglin-Waldspurger on the stabilization of trace formulas. But these are certainly provable and will come out in near future.
\end{remark}

\noindent\underline{Comparison to Literature}\\
Skinner-Urban's idea for proof is more conceptual and does not use Iwasawa theory: they construct critical slope Eisenstein series, and deform it in a cuspidal family. The resulting congruences between this Eisenstein series and cusp forms enable them to deduce the lower bound for Selmer groups. However in their construction (the Galois theoretic ``lattice construction'') there is a Selmer group for Dirichlet character $H^1_f(\mathbb{Q}, \chi(1))$ interacting with the elements of the Selmer group for $M$. This interaction makes trouble in proving lower bound for the Selmer group of $M$. This Selmer group has rank $0$ if $F=\mathbb{Q}$ and $\chi$ is the trivial character. However in general it is not the case. For example in the case when $\chi$ is trivial, by Kummer theory $H^1(F, L(1))$ is just the $p$-adic completion of the unit group $\mathcal{O}^\times_F$, which has positive rank if $F\not=\mathbb{Q}$.

In this paper, instead of working with critical slope Eisenstein series, we study \emph{Iwasawa theory} using the \emph{ordinary} Hida family of Klingen Eisenstein series, constructed in \cite{WAN}.
Previous works on Eisenstein congruences for unitary group Iwasawa theory include the work of Skinner-Urban \cite{SU} on $\mathrm{U}(2,2)/\mathbb{Q}$, Hsieh on $\mathrm{U}(2,1)$ \cite{Hsieh CM}, the author on $\mathrm{U}(2,2)/F$ \cite{WAN1} and $\mathrm{U}(3,1)$ \cite{WAN2} over $\mathbb{Q}$. These results had important consequences on BSD conjecture for elliptic curves and its generalizations to modular forms. The key ingredient and main difficulty of all such work, is to study $p$-adic properties of the families of Klingen Eisenstein series. More precisely, to prove the Klingen Eisenstein families are co-prime to the $p$-adic $L$-function to study. In those low rank cases there are various tools from the theory of automorphic forms that we can employ to study the Fourier-Jacobi coefficients, and special value formulas for $L$-functions (e.g Waldspurger formula) for showing the primitivity. Unfortunately it seems one can hardly expect to generalize them to unitary groups of general signature. In fact our strategy of study is \emph{completely different} from previous literature.\\

\noindent\underline{Our Idea}\\
Our main goal is to prove that if we specialize the Klingen Eisenstein family to the arithmetic point $\phi_0$ where the $p$-adic $L$-function takes the central value we study, then it is non-vanishing. This specialization is not of classical weight, and is purely a $p$-adic limit form, which makes it difficult to study $p$-adic properties. (In contrast, the Eisenstein series construct by Skinner-Urban in \cite{SU2} does have classical weight at the point of study). Our idea is to relate the image under certain Maass-Shimura differential operator of this $p$-adic limit form to a classical Eisenstein series using a \emph{$p$-adic functional equation}, and prove the latter is nonzero by explicit computations.\\

\noindent\underline{An example}\\
To illustrate how it works, we first discuss a toy example for Katz $p$-adic $L$-functions in \cite{Katz}. We consider the family of Eisenstein series on $\mathrm{GL}_2$ whose $q$-expansion is given by
$$E_k=\sum^\infty_{n=1}a_{n,k}q^n$$
for $a_{n,k}=\sum_{d|n}d^{k-1}$ if $(n, Np)=1$ and $a_n=0$ otherwise. Clearly these coefficients are interpolated in a $p$-adic family.  Incorporating the Maass-Shimura differential operators $\delta$, we get a $2$-variable family interpolating
$$\delta^jE_k=\sum_n\sum_d n^jd^{k-1}$$
for $j$ and $k$ varying, which we denote as $\mathbf{E}$.

For any integers $a$ and $b$, we have the simple identity of formal $q$-expansions
\begin{equation}\label{toy}
\sum_{n, (n, Np)=1}\sum_{d|n} d^a(\frac{n}{d})^bq^n=\sum_{n, (n, Np)=1}\sum_{d|n} d^b(\frac{n}{d})^aq^n.
\end{equation}
We can express the left hand side as
$$\sum d^{a-b}n^bq^n=\delta^b(\sum d^{a-b}q^n)=\delta^b E_{a-b+1}$$
where $\delta$ is the Maass-Shimura differential operator whose action on $q$ expansion is $q\frac{d}{dq}$.
Similarly the right hand side is
$$\sum d^{b-a}n^aq^n=\delta^a(\sum d^{b-a}q^n)=\delta^a E_{b-a+1}.$$
So (\ref{toy}) becomes
\begin{equation}\label{baby}
\delta^b E_{a-b+1}=\delta^a E_{b-a+1}
\end{equation}
As the CM $L$-values are expressed as certain linear combinations of values of Eisenstein series at CM points, Katz constructed the CM $p$-adic $L$-function from evaluating the family $\mathbf{E}$ at CM points. The interpolation formula is proved to the right side of the central line. Then one uses (\ref{baby}) to extend the interpolation formula to all critical values.

The same idea is also used in the construction of Rankin-Selberg $p$-adic $L$-functions by Hida in \cite{Hida91}.\\

\noindent\underline{Unitary Group Case}\\
In the recent work of Eischen-Harris-Li-Skinner \cite{EHLS}, they constructed $p$-adic $L$-functions from the doubling method of Piatetski-Shapiro and Rallis. The idea is to construct a family of Siegel Eisenstein series on $\mathrm{U}(n,n)$ ($n=r+s$) and then pullback under the embedding
$$\mathrm{U}(r,s)\times\mathrm{U}(s,r)\hookrightarrow \mathrm{U}(n,n).$$
Projecting to the $\pi\boxtimes\pi^\vee$-component on $\mathrm{U}(r,s)\times\mathrm{U}(s,r)$, one gets the desired special $L$-value.
The interpolation formula is proved to the right half of the central critical line.

In \cite[Section 2.3]{Eischen2} Eischen proposed the question that if it is possible to do the same thing as Katz in the setting of $\mathrm{U}(r,s)$. We give an affirmative answer in this paper.
In order to extend the interpolation formula to all critical values, we use again a $p$-adic functional equation on formal $q$-expansions for Siegel Eisenstein series as above. The key ingredient is
\begin{itemize}
\item a functional equation for Whittaker coefficients of degenerate principal series. (Equivalently this is the functional equation for local Siegel series). For example in the above toy example, this boils down to the simple identity
$$\ell^{ak}\sum^a_{i=0} \ell^{-ki}=\sum^a_{i=0}\ell^{ki}.$$
The general case is a deeper result of Kudla-Sweet \cite{KS}.
\item We look at the difference of the Siegel Eisenstein series on the left and right hands sides of the functional equation, and do a computation of $p$-adic Maass-Shimura differential operators (see Section \ref{Diff}) which is extensively used in \cite{EHLS}. Using the pullback formula of doubling method, we find this difference is in the image of some differential operators on the smaller group $1\times\mathrm{U}(s,r)$, which is killed by Hida's ordinary projector on it. Thus we get the functional equation for the $p$-adic $L$-function.  Replace the doubling method by Shimura's pullback formula, similarly we get also the $p$-adic functional equation for $p$-adic Klingen Eisenstein series.
\end{itemize}

In subsection \ref{Inter} we define the notion of non-arithmetic point $\phi$ and the corresponding ``dual'' arithmetic Eisenstein datum $\tilde{\mathbf{D}}^{(1)}_\phi$, $\tilde{\mathbf{D}}^{(2)}_\phi$ and an associated integer $j_\phi$. These $\phi$'s are not ``arithmetic'' in the following sense. For constructing $p$-adic $L$-functions, these points are on the left side of the center of the critical strip, while in \cite{EHLS}, the arithmetic points do not include these points. In the construction of Klingen Eisenstein family, these points do not correspond to classical weight. However the $\tilde{\mathbf{D}}^{(1)}_\phi$ and $\tilde{\mathbf{D}}^{(2)}_\phi$ are indeed arithmetic Eisenstein datum as in \cite{EHLS} or equivalently, correspond to classical Klingen Eisenstein series. We define Siegel Eisenstein sections $f_{\mathrm{sieg}}$ on $\mathrm{GU}(n+1,n+1)$ and $f'_{\mathrm{sieg}}$ on $\mathrm{GU}(n,n)$ respectively. We also define the Siegel Eisenstein sections $f_{\mathrm{sieg}}^{\mathrm{fteq}}$ on $\mathrm{GU}(n+1,n+1)$ and $f_{\mathrm{sieg}}^{\mathrm{fteq},\prime}$ on $\mathrm{GU}(n,n)$ for the right side of the functional equations. Throughout we fix a finite set $\Sigma$ of places of $F$ containing all places where $\pi$ or $\tau$ or $\mathcal{K}$ is ramified.

The Theorems on functional equations we prove are the following (proved in Section \ref{Feq}).
\begin{theorem}\label{Theorem 1.4}
For each non-arithmetic point $\phi$ and the corresponding arithmetic Eisenstein datum $\tilde{\mathbf{D}}^{(1)}_\phi$ We have
$$\mathcal{L}^\Sigma_{\mathbf{D}_\phi}=\mathcal{L}^\Sigma_{\tilde{\mathbf{D}}_\phi^{(1)},f^{\mathrm{fteq,\prime}}_{\mathrm{sieg}}}.$$
\end{theorem}

\begin{theorem}\label{Theorem 1.5}
For each non-arithmetic point and the corresponding arithmetic Eisenstein datum $\tilde{\mathbf{D}}^{(2)}_\phi$. We have
$$\delta^{1+j_\phi}_{r+1,s+1}E_{\mathrm{Kling},\mathbf{D}_{\phi},f_{\mathrm{sieg}}}=E_{\mathrm{Kling}, \tilde{\mathbf{D}}^{(2)}_\phi,f^{\mathrm{fteq}}_{\mathrm{sieg}}}.$$
The $\delta_{r+1,s+1}$ is the $p$-adic differential operator defined in Definition \ref{Definition 5.4}
\end{theorem}
\begin{remark}
The formula for the right hand side of Theorem \ref{Theorem 1.4} can be easily deduced from Proposition \ref{formula intertwining} and the proof of Proposition \ref{proposition 6.17}. We omit the precise formulas.  This extends the interpolation formula of Eischen-Harris-Li-Skinner \cite{EHLS} to the left side as well. Note that due to the existence of the differential operator in Theorem \ref{Theorem 1.5}, the right hand side is not moving in a $p$-adic analytic family.
\end{remark}
Now we explain how this helps us with proving cases of the Bloch-Kato conjecture. Look at the ordinary family $\mathbf{E}_{\mathrm{Kling}}$ of Klingen Eisenstein series constructed in \cite{WAN}, whose constant terms are divisible by the $p$-adic $L$-functions of the unitary groups. Consider the arithmetic point $\phi_0$ where this $p$-adic $L$-function takes the central critical value (which we assume to be $0$). All we need to show is the $\phi_0(\mathrm{E}_{\mathrm{Kling}})$ is nonzero. However this specialization is not in a classical weight, and is purely a $p$-adic limit, which makes it difficult to study the non-vanishing. However we can apply the $p$-adic functional equation above on it: the left side is the image of $\phi_0(\mathrm{E}_{\mathrm{Kling}})$ under certain Maass-Shimura differential operator on $\mathrm{U}(r+1,s+1)$, which makes its weight in the classical range. The right side turns out to be a \emph{classical} Klingen Eisenstein series, which we have lot of tools from automorphic form theory to compute. So our goal now is to compute this Klingen Eisenstein series on the right side and prove its non-vanishing. Here for convenience of the reader we summarize the difficulties to achieve this and our ideas to solve them.
\begin{itemize}
\item Some local pullback sections for $f^{\mathrm{fteq}}$ at bad primes are difficult to compute. We use a trick of comparing global functional equations for Siegel and Klingen Eisenstein series. Such trick is used by Skinner-Urban in \cite{SU} to compute ordinary sections at $p$-adic places. We use it here to reduce the calculation at bad primes to that of good primes (see Section \ref{Nv}).
\item In order to apply the functional equation of Kudla-Sweet we need to ensure that at one prime the local Fourier coefficient of the Siegel Eisenstein series is identically $0$ (as a function of $z$). This is because the Siegel Eisenstein measure we use to construct the ordinary Klingen Eisenstein series has only non-degenerate Fourier expansion, while it is not clearly the case for the other side of the functional equation outside the absolutely convergent range. For this purpose we pick an auxiliary prime $v$ split in $\mathcal{K}$, such that the Eisenstein datum is unramified. We choose $v$ so that $\pi_v$ has pairwise distinct Satake parameters (we prove the existence) using compatible system of Galois representations). For this $v$ we need to construct a Siegel section whose degenerate Fourier coefficients are all zero, and the pullback Klingen Eisenstein section is computable.

    \ \ Such pullback section is difficult to compute directly -- having nice description for the Fourier coefficients would result in complicated description of the Siegel section itself, and thus complicated pullback sections (uncertainty principle). Our method to solve the problem is partially borrowed from the beautiful idea of Eischen-Harris-Li-Skinner \cite[Section 4.3]{EHLS} when they do the \emph{$p$-adic} computations (which is the technical core of \cite{EHLS}). It uses the Godement-Jacquet functional equation to relate pullback sections of Siegel-Weil sections whose Schwartz functions are related under Fourier transform. Our situation is more complicated however, since we are working with Klingen Eisenstein series compared to the $p$-adic $L$-function case of \emph{loc.cit.} (see Section \ref{Auxi}).
\item In the case when the local $L$-factors for bad primes at $z=-\frac{1}{2}$ do not have poles, the required non-vanishing result is directly seen from computations. However if they do have poles, then the corresponding intertwining operator at $z=-\frac{1}{2}$ are expected to have poles, and the situation is more complicated. We apply deep theory developed by Casselman and Shahidi \cite{CS} on analytic properties of intertwining operators and reducibility of standard modules, to prove that these expected poles do exist, which imply the non-vanishing of the pullback section (see Section \ref{Nv}).
\end{itemize}

\noindent This paper is organized as follows: in section 2 we fix the set up and give the detailed formulation. In section 3 we develop the Hida theory for general $\mathrm{U}(r,s)$ needed for our argument. In section 4 we summarize our construction of Siegel and Klingen Eisenstein families. In section 5 we carry out the representation theory computations for differential operators. In section 6 we interpolate everything in families, deduce our results for $p$-adic functional equations, and prove the required non-vanishing results for the Klingen Eisenstein family. In section 7 we prove the lower bound of the Selmer group rank.\\

\noindent\emph{Acknowledgement} We would like to Z. Liu, C-P. Mok, F. Shahidi, S.G. Shin, and E. Urban for helpful communications.

\section{Set Up and Formulation}
Let $d:=[F:\mathbb{Q}]$. We take a CM type of $\mathcal{K}$ denoted as $\Sigma_\infty$ (thus $\Sigma_\infty\sqcup\Sigma^c_\infty$ are all embeddings $\mathcal{K}\rightarrow\mathbb{C}$ where $\Sigma_\infty^c=\{\tau\circ c,\tau\in\Sigma_\infty\}$). Fixing throughout an isomorphism $\iota: \mathbb{C}\simeq \mathbb{C}_p$, we can associate from $\Sigma$ a set of $p$-adic places, which we still denote as $\Sigma$ and call it a $p$-adic CM type. Consider $\mathcal{O}_{\mathcal{K},p}\simeq \oplus_{v\in\Sigma\cup\Sigma^c}\mathcal{O}_{\mathcal{K},v}$. We define idempotents $e^+=e_\Sigma$ and $e^-=e_{\Sigma^c}$ to be the projections to $v\in\Sigma$ and $v\in\Sigma^c$ parts, respectively.

For any Hecke character $\chi$ of $\mathcal{K}^\times\backslash\mathbb{A}^\times_\mathcal{K}$, we write $\chi'$ for the restriction of $\chi$ to $\mathbb{A}^\times_F$. We write $\chi^c(x):=\chi(x^c)$ where $c$ is the complex conjugation and $\bar{\chi}(x)=\overline{\chi(x)}$. We write $\chi_\mathcal{K}$ for the quadratic Hecke character corresponding to the extension $\mathcal{K}/F$.

We define:
\begin{equation}\label{hermitian space}
\theta_{r,s}=\begin{pmatrix}&&1_s\\&\zeta&\\-1_s&&\end{pmatrix}
\end{equation}
\index{$\theta_{r,s}$} where $\zeta$ is a fixed diagonal matrix such that $i^{-1}\zeta$ is totally positive. Let $\mathrm{GU}(r,s)$ and $\mathrm{U}(r,s)$ be the corresponding unitary similitude group and unitary group of signature $(r,s)$ respectively (see \cite[Section 2.2]{WAN}).

As in \cite{WAN} in this paper we write $a$ and $b$ such that $r=a+b$ and $s=b$.
\begin{definition}
A weight $\underline{k}$ \index{$\underline{k}$} is defined to be an $(r+s)$-tuple
$$\underline{k}=(a_{1,v},\cdots, a_{r,v}; b_{1,v},\cdots, b_{s,v})_{v\in\Sigma_\infty}\in\Sigma^{r+s}$$ with $a_{1,v}\geq \cdots \geq a_{r,v}\geq -b_{1,v}\geq \cdots -b_{s,v}$. We often omit the subscript $v$ when writing the weights for a given Archimedean place $v$.
\end{definition}
In this paper we also allow the case that the $a_{i,v}$ and $b_{j,v}$ are all half integers for all $i$, $j$ and $v$. It means that if we twist the $\pi$ by a Hecke character of $\mathcal{K}^\times\backslash\mathbb{A}^\times_\mathcal{K}$ of infinity type $(-\frac{1}{2},\frac{1}{2})$ at all Archimedean places, the resulting representation is of integer weight defined above. The reason for introducing this is we sometimes study special $L$-values at $z$ which is a half integer, such that the corresponding Galois representation has integral Hodge-Tate weights. We will introduce a scalar $\kappa$ later in Section \ref{Diff} when discussing Siegel Eisenstein series. We require that if $\kappa$ is odd, then all the $a_{i,v}$ and $b_{j,v}$ are half integers; If $\kappa$ is even, then these $a_{i,v}$ and $b_{j,v}$ are integers. This is to ensure that the $a'_i$ and $b'_i$ in Section \ref{Diff} are integers, and that the Hodge-Tate weights of the motive $M$ are integers.

We refer to \cite[Section 3.1]{Hsieh CM} for the definition of the algebraic representation $V_{\underline{k}}$ (denoted $L_{\underline{k}}$ there) of $H$ with the action denoted by $\rho_{\underline{k}}$ (note the different index for weight) and define a model $V^{\underline{k}}$ of the representation $H$ with the highest weight $\underline{k}$ as follows. The underlying space of $V^{\underline{k}}$ is $V_{\underline{k}}$ and the group action is defined by
$$\rho^{\underline{k}}(h)=\rho_{\underline{k}}({}^t\!h^{-1}),h\in H.$$
Let $n=r+s$. Suppose $\pi$ is an irreducible cuspidal automorphic representation with algebraic weight
$\underline{k}$.
Then by work of Harris-Taylor, Shin, Morel, etc, there is Galois representation
$$\rho_\pi: G_K\rightarrow \mathrm{GL}_n(L)=\mathrm{GL}(V)$$
associated to the base change of $\pi$ to $\mathcal{K}$. More precisely, by the identification
$$L(\mathrm{BC}(\pi),\frac{1}{2})=L(\rho_\pi,0)$$
normalized by the geometric Frobenius.

Now suppose $\pi$ is unramified and ordinary at all primes above $p$. The notion of being ordinary is defined using the Satake parameters at $p$-adic places and the weight $\underline{k}$, which basically says that the eigenvalues of $U_p$ operators are $p$-adic units. We discuss this in next section in details. Let $v=v_0\bar{v}_0$ be a place above $p$ with $v_0$ in our $p$-adic CM type. Then $\rho_\pi$ satisfies
$$\rho_\pi|_{G_{{v}_0}}\simeq \begin{pmatrix}\xi_{1,v_0}\epsilon^{-\kappa'_{n,v_0}}&*&*\\0&\cdots&*\\0&0&\xi_{n,v_0}\epsilon^{-\kappa'_{1,v_0}}\end{pmatrix},$$
and
$$\rho_\pi|_{G_{\bar{v}_0}}\simeq \begin{pmatrix}\xi_{1,\bar{v}_0}\epsilon^{-\kappa'_{n,\bar{v}_0}}&*&*\\0&\cdots&*\\0&0&\xi_{1,\bar{v}_0}
\epsilon^{-\kappa'_{1,\bar{v}_0}}
\end{pmatrix}.$$
Here $\xi$'s are unramified characters. The Hodge-Tate weights $\kappa_{i,v_0}$'s are defined as follows. Let $\kappa'_{i,v_0}=\frac{n}{2}+s-i+b_{1+s-i}$ for $1\leq i\leq s$ and $\kappa'_{s+i,v_0}=-a_{r+1-i}+s+r-i+\frac{n}{2}$ for $1\leq i \leq r$. We also let $\kappa_j=-\kappa'_{r+s+1-j}$ for every $1\leq j\leq r+s$. They depends on $\underline{k}$ and satisfy
$\kappa_{1,v_0}>\kappa_{2,v_0}>\cdots>\kappa_{n,v_0}$. One similarly has the sequence of decreasing Hodge-Tate weights $\kappa_{1,\bar{v}_0}>\kappa_{2,\bar{v}_0}>\cdots>\kappa_{n,\bar{v}_0}$.
It is well known that there is an $\mathcal{O}_L$-lattice $T$ of $V$ stable under $G_K$.

Recall we made the following assumption:
\begin{itemize}
\item[(Irred)] There is a Galois stable lattice $T$ such that the resulting residual Galois representation $\bar{\rho}_\pi$ is absolutely irreducible.
\end{itemize}
Under this assumption, the Galois stable lattice $T$ is unique up to scalar.

Let $\chi$ be a Hecke character of $K^\times\backslash\mathbb{A}^\times_K$ such that the corresponding Galois character has Hodge-Tate weight $(k_{v_0},k_{\bar{v}_0})_{v_0}$. We assume $2k_{v_0}$ and all the $2\kappa_{v_0,i}$ have the same parity, and $2k_{\bar{v}_0}$ and all the $2\kappa_{\bar{v}_0,i}$ have the same parity.

Suppose $L(\tilde{\rho_\pi}\otimes\chi,1)$ corresponds to critical value of $L$-function (following Deligne). Then there is some $i$ such that
$$\kappa_{i+1,v_0}\leq k_{v_0}<\kappa_{i,v_0},\ \kappa_{n-i+1,\bar{v}_0}\leq k_{\bar{v}_0} <\kappa_{n-i,\bar{v}_0}.$$
In this paper we assume $i=r$. The reason is that these critical values are realized via doubling method of $$\mathrm{U}(r,s)\times\mathrm{U}(s,r)\hookrightarrow \mathrm{U}(r+s,r+s).$$
This is used by \cite{EHLS} to construct the corresponding $p$-adic $L$-functions.

Now we turn to the arithmetic side. The following definition of Selmer group is due to Greenberg \cite{Gr94}.
Fix a finite set of primes $\Sigma$ including all bad primes and primes above $p$. Let $\mathcal{K}_\infty$ be the extension over $\mathcal{K}$ which is the composition of the cyclotomic $\mathbb{Z}_p$-extension and the anti-cyclotomic extension whose Galois group is isomorphic $\mathbb{Z}^d_p$. So $\Gamma_\mathcal{K}:=\mathrm{Gal}(\mathcal{K}_\infty/\mathcal{K})\simeq \mathbb{Z}^{d+1}_p$. Write $\Gamma^+_\mathcal{K}$ for the subgroup of $\Gamma_\mathcal{K}$ such that the complex conjugation acts by $+1$. Then $\Gamma^+_\mathcal{K}\simeq \mathbb{Z}_p$. Let $\Lambda_\mathcal{K}:=\mathcal{O}_L[[\Gamma_\mathcal{K}]]$. We define the Selmer group of $\rho_\pi\otimes\chi^{-1}$ over $\mathcal{K}_n$ between $\mathcal{K}$ and $\mathcal{K}_\infty$:
\begin{align*}
&\mathrm{Sel}(\mathcal{K}_n, V/T\otimes\chi^{-1}):=\mathrm{Ker}\{H^1(\mathcal{K}^\Sigma_n, V/T\otimes\chi^{-1})
\rightarrow \prod_{v\in\Sigma}\frac{H^1(\mathcal{K}_{n,v}, V/T\otimes\chi^{-1})}{H^1_f(\mathcal{K}_{n,v}, V/T\otimes\chi^{-1})}\},&
\end{align*}
where the $H^1_f$ are defined as follows.
\begin{itemize}
\item For primes $v\nmid p$, we define
\begin{align*}
&H^1_f(\mathcal{K}_{n,v}, V\otimes\chi^{-1}):=\mathrm{ker}\{H^1(\mathcal{K}_{n,v}, V\otimes\chi^{-1})\rightarrow H^1(I_{n,v}, V\otimes\chi^{-1})\},&\end{align*}
and $H^1_f(\mathcal{K}_{n,v}, V/T\otimes\chi^{-1})$ is defined to be the image of $H^1_f(\mathcal{K}_{n,v}, V\otimes\chi^{-1})$.
\item For primes above $p$, recall the local Galois representation $T$ is upper-triangular. There is a co-torsion free rank $r$ submodule $T^+_{v_0}\subseteq T$ corresponding to the upper $r$ rows at $v_0$ which is stable under $G_{v_0}$. Similarly there is a rank $s$ co-torsion free submodule $T^+_{\bar{v}_0}\subseteq T$ corresponding to the upper $s$ rows at $\bar{v}_0$. We define $H^1_f(\mathcal{K}_{n,v_0}, V/T\otimes\chi^{-1})$ as the image of $H^1(\mathcal{K}_{n,v_0}, V^+/T^+\otimes\chi^{-1})$, and similarly for $\bar{v}_0$.
\end{itemize}
We define
$$\mathrm{Sel}(\mathcal{K}_\infty, V/T\otimes\chi^{-1})=\varinjlim_{\mathcal{K}_n}\mathrm{Sel}(\mathcal{K}_n, V/T\otimes\chi^{-1}),$$
and
$X_{\pi,\chi,\mathcal{K}}$ being its Pontryagin dual. This is a finitely generated module over $\mathcal{O}_L[[\Gamma_\mathcal{K}]]$.

For a Hida family $\mathbf{f}$ containing an ordinary vector in $\pi$ as specialization with coefficient ring $\mathbb{I}$ a Noetherian normal domain, we can still construct the corresponding family of Galois representation $\rho_\mathbf{f}$, thanks to the assumption that the residual representation $\bar{\rho}_\pi$ is absolutely irreducible. We can similarly define its dual Selmer module $X_{\mathbf{f},\chi,\mathcal{K}}$. This is a finitely generated module over $\mathbb{I}[[\Gamma_\mathcal{K}]]$. We also define $\Sigma$-imprimitive versions $X^\Sigma_{\pi,\chi,\mathcal{K}}$ and $X^\Sigma_{\mathbf{f},\chi,\mathcal{K}}$ of them.

\begin{conjecture} (Bloch-Kato)\\
The vanishing order of $L(\tilde{\rho}_\pi\otimes\chi, s)$ at $s=1$ is equal to the rank of the Selmer group $\mathrm{Sel}(\mathcal{K}, \rho_\pi\otimes\chi^{-1})$.
\end{conjecture}

\section{Hida Theory for $\mathrm{U}(r,s)$}
\subsection{Notations and Conventions}
We are going to fix some basis of the various Hermitian spaces. We let $$y^1,...,y^s,w^1,...,w^{r-s},x^1,...,x^s$$
be the standard basis of the Hermitian space $V$ such that the Hermitian form is given by \ref{hermitian space}. Let $W$ be the span over $\mathcal{K}$ of $w^1,...,w^{r-s}$. Let $X^\vee=\mathcal{O}_\mathcal{K}x^1\oplus ...\oplus \mathcal{O}_{\mathcal{K}}x^s$ and $Y=\mathcal{O}_\mathcal{K}y^1\oplus...\oplus \mathcal{O}_\mathcal{K}y^s$. Let $L$ be an $\mathcal{O}_\mathcal{K}$-maximal lattice such that $L_p:=L\otimes_\mathbb{Z}\mathbb{Z}_p=\sum_{i=1}^{r-s}(\mathcal{O}_\mathcal{K}\otimes_\mathbb{Z}\mathbb{Z}_p)w^i$. We define a $\mathcal{O}_\mathcal{K}$-lattice $M$ of $V$ by
$$M:=Y\oplus L\oplus X^\vee.$$
Let $M_p=M\otimes_\mathbb{Z}\mathbb{Z}_p$. A pair of sublattice $\mathrm{Pol}_p=\{N^{-1}, N^0\}$ of $M_p$ is called an ordered polarization of $M_p$ if $N^{-1}$ and $N^0$ are maximal isotropic direct summands in $M_p$ and they are dual to each other with respect to the Hermitian pairing. Moreover we require that for each $v=ww^c$, $w\in \Sigma$,
$\mathrm{rank}N_w^{-1}=\mathrm{rank}N^0_{w^c}=r$ and $\mathrm{rank}N_{w^c}^{-1}=\mathrm{rank}N_w^0=s$. The standard polarization of $M_p$ is given by:
$M^{-1}_v=Y_w\oplus L_w\oplus Y_{w^c}$ and $M^0_v=X_{w^c}\oplus L_{w^c}\oplus X_{w}$.\\

\noindent\underline{Shimura Varieties}\\
Fix a neat open compact subgroup $K$ of $\mathrm{GU}(r,s)(\mathbb{A}_f)$ whose $p$-component is $\mathrm{GU}(r,s)(\mathbb{Z}_p)$, we refer to \cite[Section 2]{Hsieh CM} for the definitions and arithmetic models of Shimura varieties over the reflex field $E$ which we denote as $S_G(K)$. It parameterizes isomorphism classes of the quadruples $(A,\lambda, \iota, \bar{\eta}^{(\Box)})_{/S}$ where $\Box$ is a finite set of primes, $(A,\lambda)$ is a polarized abelian variety over some base ring $S$, $\lambda$ is an orbit (see \cite[Definition 2.1]{Hsieh CM}) of prime to $\Box$ polarizations of $A$, $\iota$ is an embedding of $\mathcal{O}_\mathcal{K}$ into the endomorphism ring of $A$ and $\bar{\eta}^{(\Box)}$ is some prime to $\Box$ level structure of $A$. To each point $(\tau,g)\in X^+\times G(\mathbb{A}_{F,f})$ we attach the quadruple as follows:
\begin{itemize}
\item The abelian variety $\mathcal{A}_g(\tau):= V\otimes_\mathbb{Q}\mathbb{R}/M_{[g]} (M_{[g]}:=H_1(\mathcal{A}_g(\tau),\hat{\mathbb{Z}}^p))$.
\item The polarization of $\mathcal{A}$ is given by the pullback of $-\langle,\rangle_{r,s}$ on $\mathbb{C}^{r,s}$ to $V\otimes_\mathbb{Q}\mathbb{R}$ via $p(\tau)$.
\item The complex multiplication $\iota$ is the $\mathcal{O}_\mathcal{K}$-action induced by the action on $V$.
\item The prime to $p$ level structure: $\eta_g^{(p)}: M\otimes\hat{\mathbb{Z}}^p\simeq M_{[g]}$ is defined by $\eta_g^{(p)}(x)=g *x$ for $x\in M$.
\end{itemize}

\noindent Now we recall briefly the notion of Igusa schemes over $\mathcal{O}_{v_0}$ (the localization of the integer ring of the reflex field at the $p$-adic place $v_0$ determined by $\iota_p:\mathbb{C}\simeq \mathbb{C}_p$) in \cite[Section 2]{Hsieh CM}. Recall $M$ is the standard lattice of $V$ and $M_p=M\otimes_{\mathbb{Z}}\mathbb{Z}_p$. Let $\mathrm{Pol}_p=\{N^{-1}, N^0\}$ be a polarization of $M_p$.
The Igusa variety $I_G(K^n)$ of level $p^n$ is the scheme representing the usual quadruple for Shimura variety together with a
$$j:\mu_{p^n}\otimes_\mathbb{Z}N^0\hookrightarrow A[p^n]$$
where $A$ is the abelian variety in the quadruple. Note that the existence of $j$ implies that if $p$ is nilpotent in the base ring then $A$ must be ordinary. For any integer $m>0$ let $\mathcal{O}_m:=\mathcal{O}_L/p^m$.
\\

\noindent\underline{Igusa Schemes over $\bar{S}_G(K)$}:\\
To define $p$-adic automorphic forms one needs Igusa Schemes over $\bar{S}_G(K)$. We fix such a toroidal compactification and refer to \cite[2.7.6]{Hsieh CM} for the construction. We still denote it as $I_G(K^n)$. Then over $\mathcal{O}_m$ the $I_G(K^n)$ is a Galois covering of the ordinary locus of the Shimura variety with Galois group $\prod_{v|p}\mathrm{GL}_r(\mathcal{O}_{F,v}/p^n)\times\mathrm{GL}_s(\mathcal{O}_{F,v}/p^n)$. If we write $g_p=\begin{pmatrix}A&B\\C&D\end{pmatrix}$ for the $p$-component of $g$, then we define
$$K^n=\{g\in K|g_p\equiv \begin{pmatrix}1_r&*\\0&1_s\end{pmatrix}\mathrm{mod}p^n\},$$
$$K^n_1=\{g\in K|A\in N_r(\mathbb{Z}_p)\ \mathrm{mod}p^n, D\in N_s^-(\mathbb{Z}_p)\mathrm{mod} p^n, C=0\},$$
$$K^n_0=\{g\in K|A\in B_r(\mathbb{Z}_p)\ \mathrm{mod}p^n, D\in B_s^-(\mathbb{Z}_p)\mathrm{mod} p^n, C=0\}.$$
Here the $N_r$ is the unipotent radical of the upper triangular Borel group $B_r$ of $\mathrm{GL}_r$ and $N^-_r$ is the opposite unipotent group of it, and similarly for $N_s$ and $B_s$. We write $I_G(K_0^n)=I_G(K^n)^{K_0^n}$ and $I_G(K_1^n)=I_G(K^n)^{K_1^n}$ over $\mathcal{O}_m$.\\

\noindent\underline{Igusa Schemes for Unitary Groups}\\
We refer to \cite[2.5]{Hsieh CM} for the notion of Igusa Schemes for the unitary groups $\mathrm{U}(r,s)$ (not the similitude group). It parameterizes quintuples $(A,\lambda,\iota,\bar{\eta}^{(p)},j)_{/S}$ similar to the Igusa Schemes for unitary similitude groups but requiring $\lambda$ to be a prime to $p$-polarization of $A$ (instead of an orbit). In order to use the pullback formula algebraically we need a map of Igusa schemes given by:
$$i([(A_1,\lambda_1,\iota_1,\eta_1^pK_1,j_1)],[(A_2,\lambda_2,\iota_2,\eta_2^pK_2,j_2)])=[(A_1\times A_2,\lambda_1\times\lambda_2,\iota_1,\iota_2,(\eta_1^p\times \eta_2^p)K_3,j_1\times j_2)].$$

We discuss the complex uniformization. Recall the following Hermitian symmetric domains for $\mathrm{U}(r,s)$
$$ X^+=X_{r,s}=\{\tau=\begin{pmatrix}x\\y \end{pmatrix}|x\in M_s(\mathbb{C}),y\in M_{(r-s)\times s}(\mathbb{C}),i(x^*-x)>iy^*\zeta^{-1}y\}.$$

For $z=\begin{pmatrix}x\\y\end{pmatrix}$ on it, let $B(z)=\begin{pmatrix}x^*&y^*&x\\0&-\zeta&y\\I_s&0&I_s\end{pmatrix}$. We write the complex vector space
$\mathbb{C}^{r,s}=\mathbb{C}(\Sigma^c)^s\oplus\mathbb{C}(\Sigma^c)^{r-s}\oplus \mathbb{C}(\Sigma)^s$, regarded as row vectors. We define a morphism
$$c_{r,s}: (u_1, u_2, u_3)c_{r,s}=(\bar{u}_1, \bar{u}_2, u_3).$$
Define the $\mathbb{R}$-linear map $p(z)$ by $p(z)v=vB(z)c_{r,s}$. Define the lattice $M_{[g]}(z)=p(z)M_{[g]}$. The Abelian variety at the point $(z,g)$ is defined by $\mathbb{C}^{r,s}/p(z)M_{[g]}$, and the complex multiplication is induced by the action of $V_{r,s}$ via $p(z)$.
We similarly define
$$c_{s,r}: (u_1, u_2, u_3)c_{r,s}=(u_1, u_2, \bar{u}_3).$$
For the moduli problem for $\mathrm{U}(s,r)$, we use the $p'(z)v=vB(z)c_{s,r}$, and define the Abelian variety and complex multiplication similarly, with $c_{r,s}$ replaced by $c_{s,r}$.

We discuss the pullback of Hermitian spaces. Let $z=\begin{pmatrix}x\\y\end{pmatrix}$ and $w=\begin{pmatrix}u\\v\end{pmatrix}$ be points in the symmetric domains of $\mathrm{U}(r+1,s+1)$ and $\mathrm{U}(r,s)$ respectively. As in \cite[6.10, 6.11]{Shimura97}, we define
$$R=\begin{pmatrix}1_{s+1}&&&&&\\&\frac{1}{2}1_{r-s}&&&-\frac{1}{2}1_{r-s}&\\&&&-1_s&&
\\&&1_{s+1}&&&\\&-\zeta^{-1}&&&-\zeta^{-1}&\\&&&&&1_s\end{pmatrix}$$
and
$$L=\begin{pmatrix}1_{r+1}&&&\\&&&1_{s+1}\\&&1_r&\\&1_s&&\end{pmatrix}.$$
Then by \cite[(6.11.3)]{Shimura97}, if $Z=\iota(z,w)$, then
$$R.\mathrm{diag}[B(z), B(w)]=B(Z)\mathrm{diag}[\overline{M(w)}, N(z)]L^{-1}.$$
From this one seems that $\mathrm{diag}[\overline{M(w)}, N(z)]L^{-1}$ induces isomorphism
$$M_{g,h}(Z)\simeq M_g(z)\oplus M_h(w).$$

With the above formulas, similar to \cite[Section 2.6]{Hsieh CM}, we know that taking the change of polarization into consideration
\begin{equation}\label{Upsilon}
i([z,g],[w,h])=[\iota(z,w),(g,h)\Upsilon]
\end{equation}\label{Upsilon}
where $\Upsilon\in \mathrm{U}(n+1,n+1)(F_p)$ is defined such that for each $v|p$ such that $v=ww^c$ where $w$ is in our $p$-adic CM type $\Sigma_p$, $\Upsilon_w=S_w^{-1}$ (the $S_w$ is the image of $S$ defined in (\ref{(1)}) in $\mathbb{Q}_p$).\\

\noindent\underline{$p$-adic Cusp Labels}\\
For those $v|p$, we define $\Gamma_{0,v}(p^n)\subset G(\mathcal{O}_{F,v})$ consisting of block matrices (with respect to $r+s$) $\begin{pmatrix}a&b\\c&d\end{pmatrix}$ with $c\equiv 0$ modulo $p^n$ under the standard basis. As in \cite{Rosso}, we define $\mathscr{C}_M$ to be the set of cotorsion-free isotropic submodules of $M$ with an action of $G(\mathcal{O}_F)$. Let $\Gamma=K\cap \mathrm{GU}(r,s)(\mathcal{O}_{F})$. For simplicity we assume $\Gamma$ is of the form of the principal congruence subgroup $\Gamma(N)$ of level $N$. The quotient of it by $\Gamma$ is called the set of cusp labels.

In this paper we are mainly interested in cusp labels of corank $1$. Write $\mathscr{C}_{M,1}$ for the set of cusp labels of codimension $1$. As in \emph{loc.cit.} we define the set of ``ordinary cusp labels'' $\mathscr{C}_{M,p^n,1}$ to be the orbit of the $1$-dimensional space spanned by $x^1$, under the action of $\Gamma\cap\prod_v\Gamma_v(p^n)$. This can be viewed as the set of cusp labels on the Igusa variety. Then there are natural surjective maps
$$\mathfrak{p}_{\mathscr{C},n}: \mathscr{C}_{M,p^n,1}/\Gamma\cap\prod_v\Gamma_{0,v}(p^n)\rightarrow \mathscr{C}_{M,1}/\Gamma.$$
We consider cusp labels of level $K^n_0$ at $p$-adic places.
Then for a given $V\in \mathscr{C}_{M,1}$, we have
$$\mathfrak{p}^{-1}_{\mathscr{C},n}(V)\simeq P_{r-1,1}(\mathbb{Z}_p/p^n\mathbb{Z}_p)\times P_{s-1,1}(\mathbb{Z}_p/p^n\mathbb{Z}_p)\backslash\mathrm{GL}_r(\mathbb{Z}_p/p^n\mathbb{Z}_p)\times\mathrm{GL}_s(\mathbb{Z}_p/p^n)
/B_r(\mathbb{Z}_p/p^n\mathbb{Z}_p)\times B_s(\mathbb{Z}_p/p^n\mathbb{Z}_p).$$
We define the $\mathcal{I}^0_{m,n}$ ($\mathcal{I}^1_{m,n}$) to be the ideal sheaf of the Shimura varieties or Igusa varieties of functions vanishing at the boundary components (boundary components of co-rank at least two), respectively.

Let $\Gamma_V$ be the intersection of $\Gamma$ with the stabilizer of $V$.
We let $P^\circ_{n,V}(\mathbb{Z}/p^n)$ be the image of
$$\Gamma_{V}\cap\Gamma_0(p^n)\rightarrow \mathrm{GL}_r(\mathcal{O}_{F,p})\times\mathrm{GL}_s(\mathcal{O}_{F,p}),$$
which at each $v|p$, is given by
$$g=\prod_v\begin{pmatrix}A_v&B_v\\ &D_v\end{pmatrix}\in\Gamma_{0,v}(p^n)\mapsto A_v\times D_v\ \mathrm{mod}\ p^n.$$
Note that since $g\in\Gamma_V$, we have $P^\circ_{n,V}(\mathbb{Z}/p^n)$ consists of matrices $$\begin{pmatrix}\mathrm{im}(\mathrm{GL}_1(\mathcal{O}^\times_\mathcal{K}))&*\\0&\mathrm{GL}_{r-1}(\mathcal{O}_{F,p})\end{pmatrix}
\times \begin{pmatrix}\mathrm{im}(\mathrm{GL}_1(\mathcal{O}^\times_\mathcal{K}))&0\\ *&\mathrm{GL}_{s-1}(\mathcal{O}_{F,p})\end{pmatrix}$$
such that upper left entries of the two matrices are conjugate inverse to each other.

\subsection{Hida Theory}
In this section we develop Hida theory for $\mathrm{U}(r,s)$ using the framework of \cite{Rosso}. The advantage is two fold. First it makes the definition of Hida's ordinary projector $e^{\mathrm{ord}}$ on non-cuspidal families more clear. Second, it uses only scalar valued forms, which avoids some geometric complications. The main difference here is we do not have Fourier-expansions for unitary groups of general signature (as opposed to $\mathrm{Gsp}(2n)$ of \emph{loc.cit.}), thus we need different arguments to prove certain compatibility of $U_p$-operators with respect to restricting to boundary maps. We will be brief for standard results of Hida theory and refer to \emph{loc.cit.} for details.

Let $H=\mathrm{GL}_r\times \mathrm{GL}_s$ and $T$ be the diagonal torus. Write $\mathbf{H}=H(\mathcal{O}_{F,p})$.
Let $R$ be a $p$-adic $\mathbb{Z}_p$-algebra and let $R_m:=R/p^m$. Let $T_{n,m}:=I_G(K^n)_{/R_m}$. Define:
$$V_{n,m}=H^0(T_{n,m},\mathcal{O}_{T_{n,m}}),$$
$$V_{\underline{k}}(K_\bullet^n,R_m)=H^0({T_{n,m}}_{/R_m},\omega_{\underline{k}})^{K_\bullet^n}.$$
Let $V_{\infty,m}=\varinjlim_nV_{n,m}$ and $V_{\infty,\infty}=\varprojlim_mV_{\infty,m}$. Define $V_p(G,K):=V_{\infty,\infty}^N$ ($N=N_r(\mathcal{O}_{F,p})\times N^-_s(\mathcal{O}_{F,p})\subset\mathbf{H}$) the space of $p$-adic modular forms.
We define $V_{n,m}^0$, etc, to be the cuspidal part of the corresponding spaces.\\

As in \cite[3.4, 3.5]{Hsieh CM} for $n\geq m$ we have
\begin{equation}\label{trivialization}
H^0(T_{m,n},\omega_{\underline{k}})\simeq V_{m,n}\otimes V_{\underline{k}}.
\end{equation}
\begin{definition}
Let $f$ be a $p$-adic automorphic form of weight $V_{\underline{k}}$, and let $v^*$ be a vector in $V^{\underline{k}}$. Then using (\ref{trivialization}) we can define the $v^*$-entry of $f$ to be the $p$-adic automorphic form $\langle f, v^*\rangle$ of trivial weight.
\end{definition}

\noindent\underline{Weight Space}\\
\noindent We let $\Lambda_{r,s}=\Lambda$ be the completed group algebra $\mathbb{Z}_p[[T(1+p\mathbb{Z}_p)]]$. This is a formal power series ring with $r+s$ variables. There is an action of ${T}(\mathbb{Z}_p)$ given by the action on the $j:\mu_{p^n}\otimes_\mathbb{Z} N^0\hookrightarrow  A[p^n]$. (see \cite[3.4]{Hsieh CM}) This gives the space of $p$-adic modular forms a structure of $\Lambda$-algebra. A $\bar{\mathbb{Q}}_p$-point $\phi$ of $\mathrm{Spec}\Lambda$ is called arithmetic if it is determined by a character $[\underline{k}]\cdot[\zeta]$ of $T(1+p\mathbb{Z}_p)$ where $\underline{k}$ is a weight and $\zeta=(\zeta_1,\cdots,\zeta_r;\zeta_1,\cdots,\zeta_s)$ for $\zeta_i\in \mu_{p^\infty}$.
\begin{proposition}
$$0\rightarrow \pi_{\mathcal{I},*}\mathcal{I}^0_{m,n}\rightarrow \pi_{\mathcal{I},*}\mathcal{I}^1_{m,n}\rightarrow\oplus_{\tilde{V}, \mathrm{corank}\tilde{V}=1}\iota^*_{\tilde{V},*}\mathcal{I}^{0}_{\mathrm{U}(r-1,s-1),m,n}\rightarrow 0$$
where the $\iota$'s are closed embeddings of the boundary components into the compactified Igusa variety.
\end{proposition}
This follows from that the minimally compactified Igusa varieties are affine. See \cite[Proposition 1.6.1]{Rosso} for details.

Let $Z_V$ be the co-rank one boundary component corresponding to the space $V\subset M$ and $Z^{\mathrm{ord}}_V$ be the ordinary locus. We define a subscheme $\mathcal{I}^\flat_{Z^{\mathrm{ord}}_V,m,n}\subseteq\mathcal{I}_{Z^{\mathrm{ord}}_V,m,n}$ to be the subset of $\mathfrak{p}^{-1}_{\mathscr{C},n}(V)$ corresponding to the double coset including the element
$$\begin{pmatrix}0&1\\1_{r-1}&0\end{pmatrix}\times\begin{pmatrix}0&1\\1_{s-1}&0\end{pmatrix}.$$
We also define the space $V^{1,\flat}_{m,n}$ to be the subspace of $V^1_{m,n}$ whose restriction to $\mathcal{I}_{Z^{\mathrm{ord}}_V,m,n}$ vanishes outside $\mathcal{I}^\flat_{Z^{\mathrm{ord}}_V,m,n}$.
We write this double set as $\mathfrak{p}^{-1,\flat}_{\mathscr{C},n}(\tilde{V})$. We need some further description of this coset:
$$\mathfrak{p}^{-1,\flat}_{\mathscr{C},n}(\tilde{V})=\begin{pmatrix}&1\\1_{r-1}&\end{pmatrix}
\begin{pmatrix}1_{r-1}&\\&\mathrm{GL}_1(\mathcal{O}_{F,p})/\mathrm{im}(\mathrm{GL}_1(\mathcal{O}_F))\end{pmatrix}\ \mathrm{mod}\ p^n.$$
So this is isomorphic to $\mathrm{GL}_1(\mathcal{O}_{F,p})/\mathrm{im}(\mathrm{GL}_1(\mathcal{O}_F))$.
It is expected from the Leopoldt conjecture that the $\mathrm{GL}_1(\mathcal{O}_{F,p})/\mathrm{im}(\mathrm{GL}_1(\mathcal{O}_F))$ should be rank one. This means in order to get Hida control theorem for non-cuspidal families, we should work with a smaller weight space where some weight (in fact $a_r+b_1$) is parallel.
\begin{definition}
We define the parallel weight space $\mathcal{W}^{\mathrm{par}}$ to parameterize characters $$\chi=(\chi_1,\cdots,\chi_r;\chi_{r+1},\cdots,\chi_{r+s})$$ of $T(\mathcal{O}_{F,p})$, such that the $(\chi_{r+1}/\chi_r)_v$'s for all $v|p$ are the same characters of $\mathbb{Z}^\times_p$. Clearly it is trivial on $\mathrm{im}(\mathrm{GL}_1(\mathcal{O}_F))$.
\end{definition}
From now on we write superscript $\mathrm{par}$ for the subspace of forms whose nebentypus correspond to points in $\mathcal{W}^{\mathrm{par}}$.
\begin{proposition}\label{exa}
We have the following fundamental exact sequence
$$0\rightarrow V^{0,\mathrm{par}}_{m,n}\rightarrow V^{1,\flat,\mathrm{par}}_{m,n}\rightarrow \oplus_{V\in\mathscr{C}_M/\Gamma,\mathrm{rk}V=1}\mathbb{Z}_p[[T_{\mathrm{U}(r-1,s-1)}\otimes \mathbb{Z}^\times_p]]\otimes_{\mathbb{Z}_p
[[T_{\mathrm{U}(r-1,s-1)}]]} V^0_{V,m,n}\rightarrow 0.$$
\end{proposition}
The proof is the same as \cite[Proposition 1.7.1]{Rosso}.

Let $v$ be a $p$-adic place of $F$ splitting as $w\bar{w}$ in $\mathcal{K}$. We first give a description of some power of $U_{p,i}$-operators associated to $\begin{pmatrix}p1_{i}&\\&1_{n-i}\end{pmatrix}$. We refer to \cite[Section 1.9]{Rosso} for details, and \cite[Section 3.7]{Hsieh CM} for the case of unitary groups. We fix an integer $b$ throughout this paper, such that the following is possible. We require that there is an element $\mathsf{k}_p\in\mathcal{O}_\mathcal{K}$ whose divisor is $\varpi^b_w$ and is congruent to $1$ modulo $N$. We also require that there is an element $\mathsf{k}'_p\in\mathcal{O}_\mathcal{K}$, whose divisor is $\varpi^b_w\varpi^{-b}_w$, and is congruent to $1$ modulo $N$. We define $\gamma_i$ as $\mathrm{diag}(\mathsf{k}_p,\cdots,\mathsf{k}_p,1,\cdots, 1, \mathsf{k}^{-c}_p,\cdots,\mathsf{k}^{-c}_p)$ if $i\leq s$, as $\mathrm{diag}(\mathsf{k}_p,\cdots,\mathsf{k}_p, \mathsf{k}'_p, \cdots, \mathsf{k}'_p, 1, \dots, 1, \mathsf{k}^{-c}_p, \cdots, \mathsf{k}^{-c}_p)$ for $s\leq i\leq r$. We make similar definition for other cases. We use these to express some power of the $U_{p,i}$ operators in (\ref{(50)}) below.

For $i\leq r+1$, we define set $\mathfrak{Y}'_{i}$ as the set of matrices $\begin{pmatrix}1_i&Nx\\&1_{n-i}\end{pmatrix}$ with $x$ running over $$M_{i\times(n-i)}(\mathbb{Z}_p/p^b\mathbb{Z}_p).$$ Define $\mathfrak{Y}_{i,v}$ as a set of unipotent elements in $\mathrm{U}(r,s)(\mathcal{O}_\mathcal{K})$ which are congruent to identity modulo $N$, congruent to identity modulo $p^n$ at all $p$-adic places outside $v$, and at the place $v$ are representatives of $\mathfrak{Y}'_i$. It clearly exists. For $i\leq s+1$, then $U^b_{p,i}$ is given by the following
\begin{equation}\label{(50)}
U^b_{p,i}f=\mu_{r,s}(\alpha_i)^{-1}\sum_{y\in\mathfrak{Y}_{i,v}} f|(y\gamma_i)^{-1}.
\end{equation}
For other $i$ we have similar definitions. (Note that we only consider $f$'s of trivial weight).

For given $m$ and $n$, for any $g\in\mathrm{GL}_r(\mathbb{Z}_p)\times\mathrm{GL}_s(\mathbb{Z}_p)$, we define $i_{\mathrm{gl},v}(g)$ to be an element in $\mathrm{U}_{r,s}(\mathcal{O}_F)$ which is congruent to identity modulo $N$, is congruent to $p^n$ at all $p$-adic places outside $v$, and such that $e^+i_{\mathrm{gl},v}(g)$ is congruent to $g$ modulo $p^n$ at the place $v$.

The $\mathcal{I}^\flat_{Z^{\mathrm{ord}}_V,m,n}$ can also be defined as the relative positions between the filtration of the $p$-divisible group $\mathcal{A}[p^\infty]^{\circ}$ determined by the universal family and the one defined via the semi-Abelian variety from the Mumford construction as in \cite[Section 1.7]{Rosso}. More precisely, we consider the standard basis $(x^*_{1,+},\cdots, x^*_{r,+}; x^*_{1,-}, x^*_{s,-})$ for the maximal anisotropic subspace $N^0_v\simeq \mathbb{Z}_p^{r+s}$ of $M_v$. Then the filtration from the universal family over the Igusa variety is given by
$$0\subset \mathbb{Z}_px^*_{1,+}\otimes\mu_{p^\infty}\subset\cdots\subset \mathbb{Z}_px^*_{1,+}\otimes\mu_{p^\infty}+\cdots+\mathbb{Z}_px^*_{r,+}\otimes\mu_{p^\infty},$$
$$0\subset \mathbb{Z}_px^*_{1,-}\otimes\mu_{p^\infty}\subset\cdots\subset \mathbb{Z}_px^*_{1,-}\otimes\mu_{p^\infty}+\cdots+\mathbb{Z}_px^*_{s,-}\otimes\mu_{p^\infty}.$$
We have the following Lemma.
\begin{lemma}
A $\tilde{V}$ belongs to $\mathcal{I}^\flat_{Z^{\mathrm{ord}}_V,m,n}$ if and only if $e^+\tilde{V}$ does not contain a primitive vector in $\mathbb{Z}_px^*_{1,+}+\cdots+\mathbb{Z}_px^*_{r-1,+}+p\mathbb{Z}_px^*_{r,+}$ (by primitive vector we mean a vector in $\mathbb{Z}_px^*_{1,+}+\cdots+\mathbb{Z}_px^*_{r-1,+}+\mathbb{Z}_px^*_{r,+}$ which is not divisible by $p$ in this space), and $e^-\tilde{V}$ does not contain a primitive vector in $\mathbb{Z}_px^*_{1,-}+\cdots+\mathbb{Z}_px^*_{s-1,-}+p\mathbb{Z}_px^*_{s,-}$.
\end{lemma}
The proof is the same as \cite[Proposition 1.8.2]{Rosso}.
\begin{proposition}\label{support}
Let $V^{\flat,+}_{m,n}$ be the subspace of $V^1_{m,n}$ vanishing at boundary components $\tilde{V}$ such that $e^+\tilde{V}$ contains a primitive vector in $\mathbb{Z}_px^*_{1,+}+\cdots+\mathbb{Z}_px^*_{r-1,+}+p\mathbb{Z}_px^*_{r,+}$. We similarly define $V^{\flat,-}_{m,n}$.
If $a\geq n\geq m$, then
$$U_{p,r-1}^{2ab}V^1_{m,n}\subseteq V^{\flat,+}_{m,n},$$
$$U_{p,r+s-1}^{2ab}V^1_{m,n}\subseteq V^{\flat,-}_{m,n}.$$
\end{proposition}
\begin{proof}
The proof is an analogue of \cite[Proposition 1.9.4]{Rosso}. Without loss of generality we prove the first inclusion. Suppose $\tilde{V}$ is a one dimensional space over $\mathcal{O}_\mathcal{K}$ generated by a vector $v$ such that $e^+v$ is a primitive vector in $\mathbb{Z}_px^*_{1,+}+\cdots + \mathbb{Z}_px^*_{r-1,+}+p\mathbb{Z}_px^*_{r,+}$. Write
$$X^*_r=\mathbb{Z}_px^*_{1,+}+\cdots+\mathbb{Z}_px^*_{r,+}.$$
Then it is easy to check that
$$\mathbb{Q}\begin{pmatrix}p^{ab} I_{r-1}& Nx\\ &1\end{pmatrix}^{-1}e^+\tilde{V}\cap X^*_r\subset \mathbb{Z}_px^*_{1,+}+\cdots + \mathbb{Z}_px^*_{r-1,+}+p^{ab}\mathbb{Z}_px^*_{r,+}.$$
So it is enough to show that
$$\Phi_{\tilde{V}}((U_{p,r})^{ab}f)=0$$
for each $f\in V^1_{m,n}$ and $\tilde{V}$ generated by a vector $v$ with $e^+v$ a primitive vector in $\mathbb{Z}_px^*_{1,+}+\cdots + \mathbb{Z}_px^*_{r-1,+}+p^{ab}\mathbb{Z}_px^*_{r,+}$. Suppose $e^+v$ is ${}^t\!(b_1,\cdots, b_{r-1}, b_r)$ with $p^a|b_r$ and $p\nmid b_j$ for some $j$. Write $x_1=\begin{pmatrix}b_1\\ \cdots\\ b_{r-1}\end{pmatrix}$.

We note the following fact: suppose $P'$ is a parabolic subgroup of $\mathrm{U}(r,s)$ conjugate to $P$ stabilizing $\tilde{V}$. Then for any $g\in N_{P'}(F)$, we have
\begin{equation}\label{fac}\Phi_{\tilde{V}}(f|g)=\Phi_{\tilde{V}}(f)
\end{equation}
as a form on $\mathrm{U}(r-1,s-1)$.
It is easy to see that any $i_{\mathrm{gl}}(\begin{pmatrix}1_{r-1}& \begin{matrix}Nx_1\\0\end{matrix}\\ &1\end{pmatrix}^{-1},1)$ stabilizes $\tilde{V}$. Then we have the follow expression for the $U_{p,r-1}^{ab}$ action:
$$\mu_{r,s}(\alpha_{r-1})^{-1}\sum_{x,Z}\Phi_{\tilde{V}}(f|\gamma^{-1}_{r-1}|i_{\mathrm{gl},v}
(\begin{pmatrix}1_r&\begin{matrix}Z\\0\end{matrix}\\&1_s
\end{pmatrix})^{-1}|i_{\mathrm{gl},v}({\begin{pmatrix}1&Nx\\&1\end{pmatrix}}^{-1},1))$$
where $Z$ runs over matrices in $M_{(r-1)\times s}(\mathbb{Z}/p^{ab}\mathbb{Z})$, and $x$ runs over matrices in $M_{(r-1)\times 1}(\mathbb{Z}/p^{ab}\mathbb{Z})$. Now we can write
$$M_{(r-1)\times 1}(\mathbb{Z}/p^{ab}\mathbb{Z})=(\mathbb{Z}/p^{ab}\mathbb{Z}){}^t\!(b_1,\cdots,b_{r-1})\oplus C$$
for some subgroup $C$ of $M_{(r-1)\times 1}(\mathbb{Z}/p^{ab}\mathbb{Z})$.
The above expression is
$$p^{ab}U^{ab}_{p,r-1}\sum_{x\in C}\Phi_{\tilde{V}}(f|i_{\mathrm{gl},v}({\begin{pmatrix}1&Nx\\&1\end{pmatrix}}^{-1},1).$$
Therefore $\Phi_{\tilde{V}}((U_{p,r-1})^{ab}f)$ is a multiple of $p^{ab}$, thus is $0$ since $a\geq m$.
\end{proof}
\begin{proposition}\label{stab}
The space $V^{1,\flat}_{m,n}$ is stable under the $U_p$ operators.
\end{proposition}
\begin{proof}
This is similar to \cite[Proposition 1.9.2]{Rosso}. Recall
$$X^*_r=\mathbb{Z}_px^*_{1,+}+\cdots+\mathbb{Z}_px^*_{r,+}.$$
If $v$ is a primitive vector in $\mathbb{Z}_px^*_{1,+}+\cdots + \mathbb{Z}_px^*_{r-1,+}+p\mathbb{Z}_px^*_{r,+}$, then
$$\mathbb{Q}\begin{pmatrix}pI_i& Nx\\ & I_{r-i}\end{pmatrix}^{-1}v\cap X^*_r\subset \mathbb{Z}_px^*_{1,+}+\cdots + \mathbb{Z}_px^*_{r-1,+}+p\mathbb{Z}_px^*_{r,+}.$$
Thus if $f\in V^{1,\flat}_{m,n}$, then $U_{p,i}f$ also has $0$ restriction to $\tilde{V}$'s such that $e^+\tilde{V}$ is generated by a primitive vector in $\mathbb{Z}_px^*_{1,+}+\cdots + \mathbb{Z}_px^*_{r-1,+}+p\mathbb{Z}_px^*_{r,+}$. We have a similar conclusion for the $s$-part as the $r$-part. This implies the proposition.
\end{proof}
\begin{proposition}\label{compa}
If $f\in V^{1,\flat}_{m,n}$, then for $\tilde{V}\in \mathcal{I}^\flat_{Z_v^{\mathrm{ord}},m,n}$, we have
$$\Phi_{\tilde{V}}(U_{p,i}f)=U_{p,i}'\Phi_{\tilde{V}}(f)$$
where the $\prime$ in $U'_{p,i}$ means the corresponding Hecke operator on $\mathrm{U}(r-1,s-1)$.
\end{proposition}
\begin{proof}
This is similar to \cite[Proposition 1.9.3]{Rosso}. Let the $\mathbb{Z}_p$-entry matrices $\begin{pmatrix}A_+&B_+\\C_+&D_+\end{pmatrix}\in \mathrm{GL}_{r}(\mathbb{Z}_p)$ and $\begin{pmatrix}A_-&B_-\\C_-&D_-\end{pmatrix}\in\mathrm{GL}_{s}(\mathbb{Z}_p)$ be such that $$\tilde{V}^+=\begin{pmatrix}A_+&B_+\\C_+&D_+\end{pmatrix}^{-1}\begin{pmatrix}0_{r-1}\\1\end{pmatrix}$$
and
$$\tilde{V}^-=\begin{pmatrix}A_-&B_-\\C_-&D_-\end{pmatrix}^{-1}\begin{pmatrix}1\\0_{s-1}\end{pmatrix}.$$
We prove the case for $i< r$ and other cases are similar. In this case $\mu_{r,s}(\alpha_i)=p^{ibs}$. It is easy to see that we can take $C_+$ is $0$ modulo $p^n$,
and that in terms of block matrices with respect to $i+(r-1-i)$,
$$A_+\equiv\begin{pmatrix}A_1&0\\0&A_2\end{pmatrix} (\mathrm{mod}\ p^n).$$
For $x\in M_{i\times (r-1-i)}(\mathbb{Z}_p/p^b\mathbb{Z}_p)$,  define $x_A:=A^{-1}_1xA_2$, and
$$y(x)=N^{-1}A^{-1}_1\begin{pmatrix}-I_{r-1}&Nx\end{pmatrix}B.$$
We check that
$$(\begin{pmatrix}p^{b}I_i&Nx&0\\0&I_{r-1-i}&0\\0&0&1\end{pmatrix}^{-1}\begin{pmatrix}A_+&B_+\\C_+&D_+\end{pmatrix}
\begin{pmatrix}p^{b}I_i&Nx_A&Ny(x)\\0&I_{r-1-i}&0\\0&0&1\end{pmatrix})^{-1}
\begin{pmatrix}A_+&B_+\\C_+&D_+\end{pmatrix}\in\Gamma(N)\cap\Gamma_1(p^s).$$
We first check that (noting the special form of the matrix $\begin{pmatrix}A_+&B_+\\C_+&D_+\end{pmatrix}$)
\begin{align*}
&\sum_Zf|{i_{\mathrm{gl},v}(\begin{pmatrix}A_+&B_+\\C_+&D_+\end{pmatrix}, \begin{pmatrix}A_-&B_-\\C_-&D_-\end{pmatrix})^{-1}}|{\gamma^{-1}_ii_{\mathrm{gl},v}
(\begin{pmatrix}1_i&Nx&\\&1_{r-i-1}&\\&&1
\end{pmatrix},1)^{-1}}|i_{\mathrm{gl},v}(\begin{pmatrix}1_r&\begin{matrix}Z\\0\end{matrix}\\&1_s\end{pmatrix})^{-1}&\\
&=
\sum_Zf|{\gamma^{-1}_ii_{\mathrm{gl,v}}(\begin{pmatrix}1_i&Nx_A&Ny(x)\\&1_{r-i-1}&\\&&1
\end{pmatrix},1)^{-1}}|i_{\mathrm{gl},v}(\begin{pmatrix}1_r&\begin{matrix}Z\\0\end{matrix}\\&1_s\end{pmatrix})^{-1}|{i_{\mathrm{gl},v}(\begin{pmatrix}A_+&B_+\\C_+&D_+\end{pmatrix}, \begin{pmatrix}A_-&B_-\\C_-&D_-\end{pmatrix})^{-1}},&\end{align*}
where $Z$ runs over $i\times s$ matrices with entries in $\mathbb{Z}/p^b\mathbb{Z}$.
Moreover since
$$\begin{pmatrix}p^b1_i&Nx_A&Ny\\&1_{r-i-1}&\\&&1
\end{pmatrix}^{-1}\begin{pmatrix}A&B\\&1\end{pmatrix}^{-1}\begin{pmatrix}0_{r-1} \\ 1\end{pmatrix}=\begin{pmatrix}
-\begin{pmatrix}p^{-b}Ny\\ 0\end{pmatrix}-\begin{pmatrix}p^{-b}I_i&-p^{-b}Nx_A\\0&1_{r-i-1}\end{pmatrix}A^{-1}B \\ 1 \end{pmatrix},$$
we see it contains no primitive vector in $\mathbb{Z}_px^*_1+\cdots\mathbb{Z}_p^*x_{r-1}+p\mathbb{Z}_px^*_r$
only when
$$Ny\equiv \begin{pmatrix}I_{r-1}&-Nx_A\end{pmatrix}A^{-1}_+B_+\ (\mathrm{mod}\ p^b),$$
which means $y\equiv y(x)\ (\mathrm{mod}\ p)$.
So
$$\Phi_{\tilde{V}_{\mathrm{std}}}(f|{\gamma^{-1}_ii_{\mathrm{gl,v}}(\begin{pmatrix}1&Nx_A&Ny\\&1_{r-i-1}&\\&&1
\end{pmatrix},1)^{-1}}|{i_{\mathrm{gl},v}(\begin{pmatrix}A_+&B_+\\C_+&D_+\end{pmatrix}, \begin{pmatrix}A_-&B_-\\C_-&D_-\end{pmatrix})^{-1}})|i_{\mathrm{gl},v}(\begin{pmatrix}1_r&\begin{matrix}Z\\0
\end{matrix}\\&1_s\end{pmatrix})^{-1}$$
can be nonzero only when $y\equiv y(x)(\mathrm{mod}\ p^b)$.
So we have
\begin{align*}
&\Phi_{\tilde{V}}(U_{p,i}f)&&=\Phi_{\tilde{V}_{\mathrm{std}}}((U_{p,i}f)|{i_{\mathrm{gl},v}(\begin{pmatrix}A_+&B^+\\C_+&D_+\end{pmatrix}, \begin{pmatrix}A_-&B_-\\C_-&D_-\end{pmatrix})^{-1}})&\\
&&&=\frac{1}{p^{ibs}}\Phi_{\tilde{V}_{\mathrm{std}}}(\sum_Z\sum_{x,y}f|{\gamma^{-1}_ii_{\mathrm{gl,v}}(\begin{pmatrix}1&Nx_A&Ny\\&1_{r-i-1}&0\\&&1
\end{pmatrix},1)^{-1}}&\\
&&&|i_{\mathrm{gl},v}(\begin{pmatrix}1_r&\begin{matrix}Z\\0\end{matrix}\\&1_s\end{pmatrix})^{-1})|{i_{\mathrm{gl},v}(\begin{pmatrix}A_+&B^+\\C_+&D_+\end{pmatrix}, \begin{pmatrix}A_-&B_-\\C_-&D_-\end{pmatrix})^{-1}}&\\
&&&=\frac{1}{p^{ibs}}\Phi_{\tilde{V}_{\mathrm{std}}}(\sum_Z\sum_xf|{\gamma^{-1}_ii_{\mathrm{gl,v}}(\begin{pmatrix}1&Nx_A&Ny(x)\\&1_{r-i-1}&0\\&&1
\end{pmatrix},1)^{-1}}&\\
&&&|{i_{\mathrm{gl},v}(\begin{pmatrix}A_+&B_+\\C_+&D_+\end{pmatrix}, \begin{pmatrix}A_-&B_-\\C_-&D_-\end{pmatrix})^{-1}}
|i_{\mathrm{gl},v}(\begin{pmatrix}1_r&\begin{matrix}Z\\0\end{matrix}\\&1_s\end{pmatrix})^{-1})&\\
&&&=\frac{1}{p^{ibs}}\Phi_{\tilde{V}_{\mathrm{std}}}(\sum_Z\sum_x f|{i_{\mathrm{gl},v}(\begin{pmatrix}A_+&B_+\\C_+&D_+\end{pmatrix}, \begin{pmatrix}A_-&B_-\\C_-&D_-\end{pmatrix})^{-1}}|\gamma^{-1}_i&\\
&&&|i_{\mathrm{gl},v}
(\begin{pmatrix}1&Nx&\\&1_{r-i}&\\&&1
\end{pmatrix},1)^{-1})|i_{\mathrm{gl},v}(\begin{pmatrix}1_r&\begin{matrix}Z\\0\end{matrix}\\&1_s\end{pmatrix})^{-1})&
\\
&&&=\frac{1}{p^{ib(s-1)}}\sum_{Z'}(\Phi_{\tilde{V}_{\mathrm{std}}}(f|i_{\mathrm{gl},v}(\begin{pmatrix}A_+&B_+\\C_+&D_+\end{pmatrix}, \begin{pmatrix}A_-&B_-\\C_-&D_-\end{pmatrix})^{-1}))
|\gamma^{-1}_i|i_{\mathrm{gl},v}(\begin{pmatrix}1_{r-1}&\begin{matrix}Z'\\0\end{matrix}\\&1_{s-1}\end{pmatrix})^{-1})&\\
&&&=U'_{p,i}(\Phi_{\tilde{V}_{\mathrm{std}}}(f|i_{\mathrm{gl},v}(\begin{pmatrix}A_+&B_+\\C_+&D_+\end{pmatrix}, \begin{pmatrix}A_-&B_-\\C_-&D_-\end{pmatrix})^{-1})).&
\end{align*}
Here we used (\ref{fac}). The $Z$ ($Z'$) runs over $i\times s$ ($i\times (s-1)$) matrices with entries in $\mathbb{Z}/p^b\mathbb{Z}$, the $x$ and $y$ run over matrices with entries in $\mathbb{Z}/p^b\mathbb{Z}$ with corresponding sizes. The proposition follows.
\end{proof}

With the above preparations, we can get the following standard results of Hida theory.
\begin{proposition}
We define for $q=0$ or $\flat$, Hida's ordinary idempotent $e^{\mathrm{ord}}$ can be well defined on the space $V^{q,\mathrm{ord},\mathrm{par}}$
$$V^{q,\mathrm{ord},\mathrm{par}}:=\mathrm{Hom}_{\mathbb{Z}_p}(\mathcal{V}^{q,\mathrm{ord},\mathrm{par}},\mathbb{Q}_p/\mathbb{Z}_p).$$
The space $\mathcal{V}^{\flat,\mathrm{ord},\mathrm{par}}$ is free of finite rank over $\mathcal{W}^{\mathrm{par}}$. We define
$$\mathrm{M}^{q,\mathrm{ord},\mathrm{par}}(K,\Lambda^{\mathrm{par}}):=\mathrm{Hom}_{\Lambda^{\mathrm{par}}}(V^{q,\mathrm{ord},\mathrm{par}},
\Lambda^{\mathrm{par}}).$$

Moreover for any arithmetic weight $\underline{k}$ in $\mathcal{W}^{\mathrm{par}}$, we have
$$\mathrm{M}^{q,\mathrm{ord},\mathrm{par}}\otimes_{\Lambda^{\mathrm{par}}}\Lambda^{\mathrm{par}}/P_{\underline{k}}\simeq V^{q,\mathrm{ord},\mathrm{par}}[P_{\underline{k}}].$$
\end{proposition}
This follows from the exact sequence in Proposition \ref{exa} and the corresponding result for $q=0$ proved by Hida \cite{HidaControl}.
The definition of ordinary idempotent is easily deduced from the exact sequences and the corresponding definition for cuspidal spaces as in  \cite[Proposition 1.10.1]{Rosso}. Other parts follow from unraveling the definitions and as in \cite[Proposition 1.10.2]{Rosso}.
We also have the classicality result for cusp forms below. In application we only need this cuspidal case results, which is proved by Hida \cite{HidaControl}.
\begin{proposition}
For any weight with nebentypus $\underline{k}$, there is a number $b_{\underline{k}}>0$ depending on $\underline{k}$, such that for any $b\geq b_{\underline{k}}$, all forms in
$$M^{0,\mathrm{ord}}_{\underline{k}+b(1,\cdots,1,0,\cdots,0)}(K,\mathcal{O}_p)$$
are classical.
\end{proposition}

Combining results in Propositions \ref{exa}, \ref{support}, \ref{stab} and \ref{compa}, we immediately get the following proposition:
\begin{proposition}\label{fes} (fundamental exact sequence)
We have \begin{equation}0\rightarrow e^{\mathrm{ord}}V^{0,\mathrm{par}}_{m,n}\rightarrow e^{\mathrm{ord}}V^{1,\mathrm{par}}_{m,n}\rightarrow \oplus_{V\in\mathscr{C}_M/\Gamma,\mathrm{rk}V=1}\mathbb{Z}_p[[T_{\mathrm{U}(r-1,s-1)}\otimes \mathbb{Z}^\times_p]]\otimes_{\mathbb{Z}_p
[[T_{\mathrm{U}(r-1,s-1)}]]} e^{\mathrm{ord}}V^0_{V,m,n}\rightarrow 0.
\end{equation}
\end{proposition}

\subsection{Algebraic Theory for Fourier-Jacobi Expansions}\label{Subsection FJ}
We suppose $s>0$ in this subsection. Let $X_t^\vee=\mathrm{span}_{\mathcal{O}_\mathcal{K}}\{x^1,\cdots,x^t\}$ and $Y_t=\mathrm{span}_{\mathcal{O}_\mathcal{K}}\{y^1,\cdots,y^t\}$. Let $W_t$ be the skew-Hermitian space $\mathrm{span}_{\mathcal{O}_\mathcal{K}}\{y^{t+1},\cdots, y^s,w_1,\cdots, x^{t+1},\cdots, x^s\}$. Let $G_t^0$ be the unitary similitude group of $W_t$. Let $[g]\in C_t(K)$ and $K_{G_{P_t}}=G_{P_t}(\mathbb{A}_f)\cap gKg^{-1}$ (we suppress the subscript $[g]$ so as not to make the notation too cumbersome). Let $\mathcal{A}_t$ be the universal abelian scheme over the Shimura variety $S_{G_{P_t}}(K_{G_{P_t}})$. Write $g^\vee=kg_i^\vee\gamma$ for $\gamma\in G(F)^+$ and $k\in K$. Define $X_g^\vee=X_t^\vee g_i^\vee\gamma, Y_g=Y_tg_i^\vee\gamma$. Let $X_g=\{y\in (Y_t\otimes_\mathbb{Q}\mathbb{Z})\cdot \gamma|\langle y, X_g^\vee\rangle\in\mathbb{Z}\}.$ Then we have
$$i:Y_g\hookrightarrow X_g.$$
Let $\mathcal{Z}_{[g]}$ be
$$\underline{\mathrm{Hom}}_{\mathcal{O}_\mathcal{K}}(X_g,\mathcal{A}_t^\vee)
\times_{\underline{\mathrm{Hom}}_{\mathcal{O}_\mathcal{K}}(Y_g,\mathcal{A}_t^\vee)}\underline{\mathrm{Hom}}_{
\mathcal{O}_\mathcal{K}}(Y_g,\mathcal{A}_t):=\{(c,c^t)|, c(i(y))=\lambda(c^t(y)), y\in Y_g\}.$$
Here $\underline{\mathrm{Hom}}$'s are the obvious sheaves over the big \'etale site of $S_{G_{P_t}}$, represented by Abelian schemes. Let $\mathbf{c}$ and $\mathbf{c}^\vee$ be the universal morphisms over $\underline{\mathrm{Hom}}_{\mathcal{O}_\mathcal{K}}(X_g,\mathcal{A}_t^\vee)$ and $\underline{\mathrm{Hom}}_{
\mathcal{O}_\mathcal{K}}(Y_g,\mathcal{A}_t)$. Let $N_{P_t}$ be the unipotent radical of $P_t$ and $Z(N_{P_t})$ be its center. Let $H_{[g]}:=Z(N_{P_t}(F))\cap g_iKg_i^{-1}$. Note that if we replace the components of $K$ at $v|p$ by $K_1^n$ then the set $H_{[g]}$ remain unchanged. Let $\Gamma_{[g]}:=\mathrm{GL}_\mathcal{K}(Y_t)\cap g_iKg_i^{-1}$. Let $\mathcal{P}_{\mathcal{A}_t}$ be the Poincar\'e sheaf over $\mathcal{A}_t^\vee\times\mathcal{A}_t/_{\mathcal{Z}_{[g]}}$ and $\mathcal{P}_{\mathcal{A}_t}^\times$ its associated $\mathbb{G}_m$-torsor. Let $S_{[g]}:=\mathrm{Hom}(H_{[g]},\mathbb{Z})$. For any $h\in S_{[g]}$ let $c(h)$ be the tautological map $\mathcal{Z}_{[g]}\rightarrow \mathcal{A}_t^\vee\times \mathcal{A}_t$ and $\mathcal{L}(h):=c(h)^*\mathcal{P}_{\mathcal{A}_t}^\times$ its associated $\mathbb{G}_m$ torsor over $\mathcal{Z}_{[g]}$.\\

\noindent It is well-known (see e.g. \cite[Chapter 7]{LAN}) that the minimal compactification $S^*_G(K)$ is the disjoint union of boundary components corresponding to $t$'s for all $1\leq t\leq s$. Let $\mathcal{O}_{\mathbb{C}_p}$ be the valuation ring for $\mathbb{C}_p$. The following proposition is proved in \cite[Proposition 7.2.3.16]{LAN}. Let $[g]\in C_t(K)$ and $\bar{x}$ is a $\mathcal{O}_{\mathbb{C}_p}$-point of the $t$-stratum of $S^*_G(K)(1/E)$ corresponding to $[g]$.
\begin{proposition}
Let $[g]$ and $\bar{x}$ be as above. We write the subscript $\bar{x}$ to mean formal completion along $\bar{x}$. Let $\pi$ be the map $\bar{S}_G(K)\rightarrow S^*_G(K)$. Then $\pi_*(\mathcal{O}_{\bar{S}_G(K)})_{\bar{x}}$ is isomorphic to
$$\{\sum_{h\in S_{[g]}^+} H^0(\mathcal{Z}_{[g]},\mathcal{L}(h))_{\bar{x}}q^h\}^{\Gamma_{[g]}}.$$
Here $S_{[g]}^+$ means the totally non-negative elements in $S_{[g]}$. The $q^h$ is just regarded as a formal symbol and $\Gamma_{[g]}$ acts on the set by a certain formula which we omit.
\end{proposition}
For each $[g]\in C_t(K)$ we fix a $\bar{x}$ corresponding to it as above.
Now we consider the diagram
\[\begin{CD}
T_{n,m}@>\pi_{n,m}>> T_{n,m}^*\\
@VVV                   @VVV\\
\bar{S}_G(K)[1/E]_{\mathcal{O}_m}@>\pi>>S_G^*(K)[1/E]_{\mathcal{O}_m}
\end{CD}\]
where $T_{n,m}\rightarrow T_{n,m}^*\rightarrow S_G^*(K)[1/E]_{\mathcal{O}_m}$ is the Stein factorization. By \cite[Corollary 6.2.2.8]{LAN1} $T_{n,m}^*$ is finite \'etale over $S_G^*(K)[1/E]_{\mathcal{O}_m}$. Taking a preimage of $\bar{x}$ in $T_{n,m}^*$ which we still denote as $\bar{x}$. (For doing this we have to extend the field of definition to include the maximal unramified extension of $L$). Then the formal completion of the structure sheaf of $T_{n,m}^*$ and $S_G^*(K)[1/E]_{\mathcal{O}_m}$ at $\bar{x}$ are isomorphic. So for any $p$-adic automorphic form $f\in\varprojlim_m\varinjlim_n H^0(T_{n,m},\mathcal{O}_{n,m})$ (with trivial coefficients) we have a Fourier-Jacobi coefficient
\begin{equation}\label{DefineFJ}
FJ(f)\in\{\prod_{h\in S_{[g]}^+}\varprojlim_m\varinjlim_n H^0(\mathcal{Z}_{[g]},\mathcal{L}(h))_{\bar{x}}\cdot q^h\}_{[g]}
\end{equation}
by considering $f$ as a global section of $\pi_{n,m}^*(\mathcal{O}_{T_{n,m}})=\mathcal{O}_{T_{n,m}^*}$ and pullback at $\bar{x}$'s. Note that if $t=s=1$ then there is no need to choose the $\bar{x}$'s and pullback since the Shimura varieties for $G_t$ is $0$-dimensional (see \cite[(2.18)]{Hsieh CM}). In application when we construct families of Klingen Eisenstein series in terms off Fourier-Jacobi coefficients, we will take $t=1$ and define
\begin{equation}\label{(5)}
\mathcal{R}_{[g],\infty}:=\prod_{h\in S_{[g]}^+}\varprojlim_m\varinjlim_n H^0(\mathcal{Z}_{[g]},\mathcal{L}(h))_{\bar{x}}\cdot q^h.
\end{equation}
We remark that the map $FJ$ is injective on the space of forms with prescribed nebentypus at $p$. This can be seen using the discussion of \cite{SU} right before Section 6.2 of \emph{loc.cit} (which in turn uses result of Hida about the irreducibility of Igusa towers for the group $\mathrm{SU}(r,s)\subset \mathrm{U}(r,s)$ (kernel of the determinant)). Note also that since the geometric fibers of the minimal compactification are normal, their irreducible componenents are also connected components. In particular to see this injectivity we need the fact that there is a bijection between the irreducible components of the generic and special fiber of $S^*_G(K)$ (see \cite[Subsection 6.4.1]{LAN}). Since the signature is $(r,s)$ for $r\geq s>0$, so by our definition there exists cusp labels in $C_t(K)$ for each $1\leq t \leq s$. Moreover, such cusp label intersects with each connected component by \cite[Theorem A.2.2]{WAN}. Since $p$ splits completely in $\mathcal{K}$ the cusps of minimal genus must be in the ordinary locus.

Now we consider the Fourier-Jacobi coefficient at cusp labels $C_t$ for $t=s$, and define the Fourier-Jacobi expansions for $\Lambda$-adic families. The reason for taking $t=s$ is, when taking the $p$-part of the level group smaller and smaller, one checks that the $p$-part of the level group of the $\theta$-part of the Fourier-Jacobi coefficient, namely for the $H^0(\mathcal{Z}_{[g]},\mathcal{L}(h))_{\bar{x}}$ is unchanged. So taking any functional $\theta$ on finite dimensional vector space $H^0(\mathcal{Z}_{[g]},\mathcal{L}(h))_{\bar{x}}$, it makes sense to define the Fourier-Jacobi coeffcient $\mathrm{FJ}_{h,\theta,g}$ of a $\Lambda$-adic family $F$, taking values in the space of $p$-adic automorphic forms on the definite unitary group $\mathrm{U}_{[g]}(r,0)$, which we denote as $\hat{A}_{[g]}^\infty$.

\begin{definition}
Let $A$ be a finite torsion free $\Lambda$-algebra. Let $\mathcal{N}_{\mathrm{ord}}(K,A)$ be the set of formal Fourier-Jacobi expansions:
$$F=\{\sum_{\beta\in\mathscr{S}_{[g]}}a(\beta, F)q^\beta, a(\beta,F)\in (A\hat{\otimes}\hat{A}_{[g]}^\infty)^\Lambda \otimes H^0(\mathcal{Z}_{[g]}^\circ,\mathcal{L}(\beta))\}_{g\in X(K)}$$
such that for a Zariski dense set $\mathcal{X}_F\subseteq \mathcal{X}_\rho$ of points $\phi\in\mathrm{Spec}A$ where the induced point in $\mathrm{Spec}\Lambda$ is some arithmetic weight $\underline{k}_\zeta$, the specialization $F_\phi$ of $F$ is the highest weight vector of the Fourier-Jacobi expansion of an ordinary modular form with tame level $K^{(p)}$, weight $\underline{k}$ and nebentype at $p$ given by $[\underline{k}][\underline{\zeta}]\omega^{-[\underline{k}]}$ as a character of $K_0(p)$. Here the superscript $\Lambda$ in $(A\hat{\otimes}\hat{A}_{[g]}^\infty)^\Lambda$ means that the $\Lambda$-action as a nebentypus character is compatible with the $\Lambda$-algebra structure of $R$
\end{definition}
Then we have the following
\begin{proposition}
$$\mathcal{M}_{\mathrm{ord}}(K,A)=\mathcal{N}_{\mathrm{ord}}(K,A).$$
\end{proposition}

\section{Eisenstein Series and Fourier-Jacobi Coefficients}\label{Section Eisenstein Series}
The materials of this section are straightforward generalizations of parts of \cite[Section 9 and 11]{SU} and we use the same notations as \emph{loc.cit}; so everything in this section should eventually be the same as \cite{SU} when specializing to the group $\mathrm{GU}(2,2)_{/\mathbb{Q}}$.
\subsection{Klingen Eisenstein Series}\label{K E S}
Let $\mathfrak{g}\mathfrak{u}(\mathbb{R})$ be the Lie algebra of $\mathrm{GU}(r,s)(\mathbb{R})$. Let $\delta$ be a character of the Klingen parabolic subgroup $P$ such that $\delta^{a+2b+1}=\delta_P$ (the modulus character of $P$).
\subsubsection{Archimedean Picture}\label{Archimedean Setup}

\noindent Let $v$ be an infinite place of $F$ so that $F_v\simeq \mathbb{R}$. Let $\boldsymbol{i}'$ and $\boldsymbol{i}$ be the points on the Hermitian symmetric domain for $\mathrm{GU}(r,s)$ and $\mathrm{GU}(r+1,s+1)$ which are $\begin{pmatrix}i1_s\\0\end{pmatrix}$ and $\begin{pmatrix}i1_{s+1}\\0\end{pmatrix}$ respectively (here $0$ means the $(r-s)\times s$ or $(r-s)\times (s+1)$ matrix $0$). Let $\mathrm{GU}(r,s)(\mathbb{R})^+$ be the subgroup of $\mathrm{GU}(r,s)(\mathbb{R})$ whose similitude factor is positive. Let $K_\infty^+$ and $K_\infty^{+,\prime}$ be the compact subgroups of $\mathrm{U}(r+1,s+1)(\mathbb{R})$ and $\mathrm{U}(r,s)(\mathbb{R})$ stabilizing $\boldsymbol{i}$ or $\boldsymbol{i}'$ and let $K_\infty$  ($K_{\infty}'$) \index{$K_{\infty}'$} be the groups generated by $K_\infty^+$ ($K_\infty^{+,\prime}$)\index{$K_\infty^{+,'}$} and $\mathrm{diag}(1_{r+s+1},-1_{s+1})$ (resp. $\mathrm{diag}(1_{r+s},-1_s)$).

Now let $(\pi,H)$ be a unitary tempered Hilbert representation of $\mathrm{GU}(r,s)(\mathbb{R})$ with $H_\infty$ the space of smooth vectors. We define a representation of $P(\mathbb{R})$ on $H_\infty$ as follows: for $p=mn,n\in N_P(\mathbb{R}),m=m(g,a)\in M_P(\mathbb{R})$ with $a\in \mathbb{C}^\times,g\in \mathrm{GU}(r+1,s+1)(\mathbb{R}),$ put
$$\rho(p)v:=\tau(a)\pi(g)v,v\in H_\infty.$$
We define a representation by smooth induction $I(H_\infty):=\mathrm{Ind}_{P(\mathbb{R})}^{\mathrm{GU}(r+1,s+1)(\mathbb{R})}\rho$ and denote $I(\rho)$ as the space of $K_\infty$-finite vectors in $I(H_\infty)$. For $f\in I(\rho)$ we also define for each $z\in \mathbb{C}$ a function
$$f_z(g):=\delta(m)^{(a+2b+1)/2+z}\rho(m)f(k),g=mk\in P(\mathbb{R})K_\infty,$$
\index{$\delta(m)$}
and an action of $\mathrm{GU}(r+1,s+1)(\mathbb{R})$ on it by
$$(\sigma(\rho,z)(g)f)(k):=f_z(kg).$$

\noindent Let $(\pi^\vee,V)$ \index{$\pi^\vee$} be the irreducible $(\mathfrak{gu}(\mathbb{R}),K_\infty')$-module given by $\pi^\vee(x)=\pi(\eta^{-1}x\eta)$ for $\eta=\begin{pmatrix}&&1_b\\&1_a&\\-1_b&&\end{pmatrix}$ and
$x$ in $\mathfrak{gu}(\mathbb{R})$ or $K_\infty'$ (this does not mean the contragradient representation!). Denote $\rho^\vee,I(\rho^\vee), I^\vee(H_\infty)$ \index{$\rho^\vee$} and $\sigma(\rho^\vee,z),I(\rho^\vee))$ the representations and spaces defined as above but with $\pi,\tau$ replaced by $\pi^\vee\otimes(\tau\circ \det ),\bar{\tau}^c$. We are going to define an intertwining operator. Let $w=\begin{pmatrix}&&1_{b+1}\\&1_a&\\-1_{b+1}&&\end{pmatrix}$. For any $z\in\mathbb{C}$, $f\in I(H_\infty)$ and $k\in K_\infty$ consider the integral:
\begin{equation}
A(\rho,z,f)(k):=\int_{N_P(\mathbb{R})}f_z(wnk)dn.
\end{equation}
This is absolutely convergent when $\mathrm{Re}(z)>\frac{a+2b+1}{2}$ and $A(\rho,z,-)\in \mathrm{Hom}_\mathbf{C}(I(H_\infty),I^\vee(H_\infty))$ intertwines the actions of $\sigma(\rho,z)$ and $\sigma(\rho^\vee,-z)$.

Suppose $\pi$ is the holomorphic discrete series representation associated to the (scalar) weight $(0,...,0;\kappa,...,\kappa)$, then it is well known that there is a unique (up to scalar) vector $v\in\pi$ such that $k\cdot v=\det \mu(k,i)^{-\kappa}$ (here $\mu$ means the second component of the automorphic factor $J$ instead of the similitude character) for any $k\in K_{\infty}^{+,\prime}$. Then by Frobenius reciprocity law there is a unique (up to scalar) vector $\tilde v\in I(\rho)$ such that $k\cdot \tilde v=\det \mu(k,i)^{-\kappa}\tilde{v}$ for any $k\in K_\infty^+$. We fix $v$ and multiply $\tilde v$ by a constant so that $\tilde v(1)=v$. In $\pi^\vee$, $\pi(w)v$ has the action of $K_\infty^+$ given by multiplying by $\det \mu(k,i)^{-\kappa}$. We define $w'\in \mathrm{U}(r+1,s+1)$ by $w'=\begin{pmatrix}1_b&&&&\\&&&&1\\&&1_a&&\\&&&1_b&\\&-1&&&\end{pmatrix}$. There is a unique vector $\tilde v^\vee\in I(\rho^\vee)$ such that the action of $K_\infty^+$ is given by $\det \mu(k,i)^{-\kappa}$ and $\tilde v^\vee(w')=\pi(w)v$. Then by uniqueness there is a constant $c(\rho,z)$ such that $A(\rho,z,\tilde v)=c(\rho,z)\tilde v^\vee$. \index{$A(\rho,z,-)$}
\begin{definition}\label{3.1.1}
We define $F_\kappa\in I(\rho)$ to be the $\tilde{v}$ as above.
\end{definition}
\index{$F_\kappa$}
\subsubsection{Prime to $p$ Picture}
\noindent Our discussion here follows \cite[9.1.2]{SU}. Let $(\pi,V)$ be an irreducible, admissible representation of
$\mathrm{GU}(r,s)(F_v)$ which is unitary and tempered. Let $\psi$ and $\tau$ be
unitary characters of $\mathcal{K}_v^\times$ such that
$\psi$ is the central character for $\pi$. We define a representation $\rho$ of $P(F_v)$ as follows. For $p=mn,n\in N_P(F_v)$,
$m=m(g,a)\in M_P(F_v), a\in K_v^\times,g\in \mathrm{GU}(F_v)$ let
$$\rho(p)v:=\tau(a)\pi(g)v,v\in V.$$
Let $I(\rho)$ be the representation defined by admissible induction:
$I(\rho)=\mathrm{Ind}_{P(F_v)}^{\mathrm{GU}(r+1,s+1)(F_v)}\rho$.
As in the Archimedean case, for each $f\in I(\rho)$ and each
$z\in \mathbb{C}$ we define a function $f_z$ on $\mathrm{GU}(r+1,s+1)(F_v)$ by
$$f_z(g):=\delta(m)^{(a+2b+1)/2+z}\rho(m)f(k),g=mk\in P(F_v)K_v$$
and a representation $\sigma(\rho,z)$ of $\mathrm{GU}(r+1,s+1)(F_v)$ on $I(\rho)$
by
$$(\sigma(\rho,z)(g)f)(k):=f_z(kg).$$

\noindent Let $(\pi^\vee, V)$ be given by $\pi^\vee(g)=\pi(\eta^{-1}g\eta).$
This representation is also tempered and unitary. We denote by
$\rho^\vee,I(\rho^\vee)$, and $(\sigma(\rho^\vee,z),I(\rho^\vee))$
the representations and spaces defined as above but with $\pi$
and $\tau$ replaced by $\pi^\vee\otimes(\tau\circ
\det)$, and $\bar \tau^c$, respectively. \\

\noindent For $f\in I(\rho),k\in
K_v$, and $z\in \mathbb{C}$ consider the integral
\begin{equation}
A(\rho,z,v)(k):=\int_{N_P(F_v)}f_z(wnk)dn.
\end{equation}
As a consequence of our hypotheses on $\pi$ this integral converges
absolutely and uniformly for $z$ and $k$ in compact subsets of
$\{z: \mathrm{Re}(z)>(a+2b+1)/2\}\times K_v$. Moreover, for such $z$,
$A(\rho,z,f)\in I(\rho^\vee)$ and the operator $A(\rho,z,-)\in
\mathrm{Hom}_\mathbb{C}(I(\rho),I(\rho^\vee))$ intertwines the actions of
$\sigma(\rho,z)$ and $\sigma(\rho^\vee,-z).$\\\\
For any open subgroup $U\subseteq K_v$ let $I(\rho)^U\subseteq
I(\rho)$ be the finite-dimensional subspace consisting of functions
satisfying $f(ku)=f(k)$ for all $u\in U$. Then the function
$$\{z\in\mathbb{C}:\mathrm{Re}(z)>(a+2b+1)/2\}\rightarrow Hom_\mathbb{C}(I(\rho)^U, I(\rho^\vee)^U), z\mapsto A(\rho,z,-)$$
is holomorphic. This map
has a meromorphic continuation to all of $\mathbb{C}$.\\

\noindent We finally remark that when $\pi$ and $\tau$ are unramified, there is a unique up to scalar unramified vector $F_{\rho_v}\in I(\rho)$. \index{$F_{\rho_v}$}
\subsubsection{Global Picture}
We follow \cite[9.1.4]{SU}. Let $(\pi,V)$ be an irreducible cuspidal tempered automorphic representation of $\mathrm{GU}(r,s)(\mathbb{A}_F)$. It is an admissible $(\mathfrak{gu}(\mathbb{R}),K_\infty')_{v|\infty}\times \mathrm{GU}(r,s)(\mathbb{A}_f)$-module which is a restricted tensor product of local irreducible admissible representations. Let $\psi, \tau:\mathbb{A}_\mathcal{K}^\times \rightarrow \mathbb{C}^\times$ be Hecke characters such that $\psi$ is the central character of $\pi$. Let $\tau=\otimes \tau_w$ and $\psi=\otimes \psi_w$ be their local decompositions, $w$ running over places of $F$. Define a representation of $(P(F_\infty)\cap K_\infty)\times P(\mathbb{A}_{F,f})$ by putting:
 $$\rho(p)v:=\otimes(\rho_w(p_w)v_w),$$
 Let $I(\rho)$ be the restricted product $\otimes I(\rho_w)$'s with respect to the $F_{\rho_w}$'s at those $w$ at which
 $\tau_w,\psi_w,\pi_w$ are unramified. As before, for each $z\in \mathbb{C}$ and $f\in I(\rho)$ we define a function $f_z$ on $\mathrm{GU}(r+1,s+1)(\mathbb{A}_F)$ as
 $$f_z(g):=\otimes f_{w,z}(g_w)$$
 where $f_{w,z}$ are defined as before and an action $\sigma(\rho,z)$ of $(\mathfrak{gu},K_{\infty})\otimes \mathrm{GU}(r+1,s+1)(\mathbb{A}_f)$ on $I(\rho)$ by $\sigma(\rho,z):=\otimes\sigma(\rho_w,z).$
 Similarly we define $\rho^\vee,I(\rho^\vee)$, and $\sigma(\rho^\vee,z)$ but with the corresponding things replaced by their $\vee$'s and we have global versions of the intertwining operators $A(\rho,f,z)$.
\begin{definition}\label{Datum}
Then we call a quadruple $\mathcal{D}=(\pi,\tau, \kappa, \Sigma)$ an Eisenstein datum where $\pi$ is a regular algebraic cuspidal automorphic representation of $\mathrm{U}(r,s)_{/F}$ which is unramfied and ordinary at all places above $p$; the $\tau$ is a finite order Hecke character; $\kappa\geq r+s$ is an integer; $\Sigma$ is a finite set of primes of $F$ containing all the infinite places, primes dividing $p$ and places where $\pi$ or $\tau$ is ramified. We define $z_\kappa=\frac{\kappa-r-s-1}{2}$ and $z'_\kappa=\frac{\kappa-r-s}{2}$.
\end{definition}
\subsubsection{Klingen-Type Eisenstein Series on G}
 We follow \cite[9.1.5]{SU} in this subsubsection. Let $\pi,\psi,$ and $\tau$ be as above. For $f\in I(\rho),z\in
 \mathbb{C}$, there are maps from
$I(\rho)$ and $I(\rho^\vee)$ to spaces of automorphic forms on $P(\mathbb{A}_F)$ given by
$$f\mapsto (g\mapsto f_z(g)(1)).$$
In the following we often write $f_z$ for the automorphic form on $P(\mathbb{A}_F)$ given by this recipe.

If $g\in \mathrm{GU}(r+1,s+1)(\mathbb{A}_F)$ it is well known that

\begin{equation}
E(f,z,g):=\sum_{\gamma\in P(F)\setminus G(F)} f_z(\gamma g)
\end{equation}
converges absolutely and uniformly for $(z,g)$ in compact subsets of $\{z\in\mathbb{C}: \mathrm{Re}(z)>\frac{a+2b+1}{2}\}\times
\mathrm{GU}(r+1,s+1)(\mathbb{A}_F)$. Therefore we get some automorphic forms which are called Klingen Eisenstein series.

\begin{definition}\index{$E_P$} \index{$E_R$}
For any parabolic subgroup $R$ of $\mathrm{GU}(r+1,s+1)$ and an automorphic form $\varphi$ we define $\varphi_R$ to be the constant term of $\varphi$ along $R$ defined by
$$\varphi_R(g)=\int_{n\in N_R(F)\backslash N_R(\mathbb{A}_F)}\varphi(ng)dn.$$
\end{definition}
The following lemma is well-known (see \cite[Lemma 9.2]{SU}).
\begin{lemma}\label{4.1.4}
Let $R$ be a standard $F$-parabolic subgroup of $\mathrm{GU}(r+1,s+1)$ (i.e, $R\supseteq B$ where $B$ is the standard Borel subgroup). Suppose $\mathrm{Re}(z)>\frac{a+2b+1}{2}$.\\
(i) If $R\not= P$ then $E(f,z,g)_R=0$;\\
(ii) $E(f,z,-)_P=f_z+A(\rho,f,z)_{-z}$.
\end{lemma}

As in \cite[Section 9.5]{SU} the Galois representation associated to the Klingen Eisenstein series is the following.
\begin{equation}\label{Klingen Galois}
\rho_\pi\oplus \tau\cdot|\cdot|^{-\frac{\kappa-r-s-2}{2}}\oplus\bar{\tau}^c|\cdot|^{\frac{\kappa-r-s}{2}}.
\end{equation}

\subsection{Siegel Eisenstein Series on $G_n$}

\subsubsection{Local Picture}
Our discussion in this subsection follows \cite[11.1-11.3]{SU} closely. Let $Q=Q_n$ \index{$Q_n$} be the Siegel parabolic subgroup of $\mathrm{GU}_n$ consisting of matrices $\begin{pmatrix}A_q&B_q\\0&D_q\end{pmatrix}$. It consists of matrices whose lower-left $(n\times n)$ block is zero. \\

\noindent For a finite place $v$ of $F$ and a character $\tau$ of $\mathcal{K}_v^\times$ we let $I_n(\tau)$ be the space of smooth
$K_{n,v}$-finite functions (here $K_{n,v}$ means the open compact group $G_n(\mathcal{O}_{F,v})$)\index{$K_{n,v}$} $f: K_{n,v}\rightarrow \mathbb{C}$ such that $f(qk)=\tau(\det D_q)f(k)$ for all $q\in Q_n(F_v)\cap K_{n,v}$ (we write $q$ as block matrix $q=\begin{pmatrix}A_q&B_q\\0&D_q\end{pmatrix}$). For $z\in \mathbb{C}$ and $f\in I(\tau)$ we also define a function $f(z,-):G_n(F_v)\rightarrow \mathbb{C}$ by $$f(z,qk):=\tau(\det D_q))|\det A_q D_q^{-1}|_v^{z+n/2}f(k),$$
$q\in Q_n(F_v)$ and $k\in K_{n,v}.$\\

\noindent For $f\in I_n(\chi),z\in \mathbb{C},$ and $k\in K_{n,v}$, the intertwining integral is defined by:
$$M(z,f)(k):=\bar\tau^n(\mu_n(k))\int_{N_{Q_n}(F_v)}f(z,w_nrk)dr.$$
For $z$ in compact subsets of $\{\mathrm{Re}(z)>n/2\}$ this integral converges absolutely and uniformly, with the convergence being uniform in $k$. In this case it is easy to see that $M(z,f)\in I_n(\bar{\chi}^c).$ A standard fact from the theory of Eisenstein series says that this has a continuation to a meromorphic section on all of $\mathbb{C}$.\\

\noindent Let $\mathcal{U}\subseteq \mathbb{C}$ be an open set. By a meromorphic section of $I_n(\tau)$ on $\mathcal{U}$ we mean a function $\varphi:\mathcal{U}\mapsto I_n(\tau)$ taking values in a finite dimensional subspace $V\subset I_n(\tau)$ and such that
$\varphi:\mathcal{U}\rightarrow V$ is meromorphic.\\

\noindent For Archimedean places there is a similar picture (see \emph{loc.cit}).
\subsubsection{Global Picture}
For an idele class character $\tau=\otimes\tau_v$ of $\mathbb{A}_\mathcal{K}^\times$ we define a space $I_n(\tau)$ to be the restricted tensor product defined using the spherical vectors $f_v^{sph}\in I_n(\tau_v),f_v^{sph}(K_{n,v})=1$, at the finite places $v$ where $\tau_v$ is unramified.\\

\noindent For $f\in I_n(\tau)$ we consider the Eisenstein series
$$E(f;z,g):=\sum_{\gamma\in Q_n(F)\setminus G_n(F)} f(z,\gamma g).$$
This series converges absolutely and uniformly for $(z,g)$ in compact subsets of $\{\mathrm{Re}(z)>n/2\}\times G_n(\mathbb{A}_F)$. The automorphic form defined is called Siegel Eisenstein series. \\

\noindent Let $\varphi:\mathcal{U}\rightarrow I_n(\tau)$ be a meromorphic section, then we put $E(\varphi;z,g)=E(\varphi(z);z,g).$ This is defined at least on the region of absolute convergence and it is well known that it can be meromorphically continued to all $z\in \mathbb{C}$.\\

\noindent Now for $f\in I_n(\tau),z\in \mathbb{C}$, and $k\in \prod_{v\nmid \infty}K_{n,v}\prod_{v|\infty}K_\infty$ there is a similar intertwining integral $M(z,f)(k)$ as above but with the integration being over $N_{Q_n}(\mathbb{A}_F)$. This again converges absolutely and uniformly for $z$ in compact subsets of $\{\mathrm{Re}(z)>n/2\}\times K_n$. Thus $z\mapsto M(z,f)$ defines a holomorphic section $\{Re(z)>n/2\}\rightarrow I_n(\bar{\tau}^c)$. This has a continuation to a meromorphic section on $\mathbb{C}$. For $\mathrm{Re}(z)>n/2$, we have
$$M(z,f)=\otimes_v M(z,f_v),f=\otimes f_v.$$
\index{$M(z;-)$}

\noindent The functional equation for Siegel Eisenstein series is:
$$E(f,z,g)=\chi^n(\mu(g))E(M(z,f);-z,g)$$
in the sense that both sides can be meromorphically continued to all $z\in\mathbb{C}$ and the equality is understood as of meromorphic functions of $z\in\mathbb{C}$.
\subsubsection{The Pullback Formulas}
We define
\begin{equation}\label{(1)}
S=\begin{pmatrix}1_b&&&&&&&-\frac{1}{2}\cdot 1_b\\&1&&&&&&\\&&1_a&&&&-\frac{\zeta}{2}&\\&&&-1_b&\frac{1}{2}\cdot 1_b&&&\\&&&1_b&\frac{1}{2}\cdot1_b &&&\\&&&&&1&&\\&&-1_a&&&&-\frac{\zeta}{2}
&\\-1_b&&&&&&&-\frac{1}{2}\cdot1_b\end{pmatrix}
\end{equation}
and
\begin{equation}\label{(2)}
S'=\begin{pmatrix}1_b&&&&&-\frac{1}{2}\cdot 1_b\\&1_a&&&-\frac{\zeta}{2}&\\&&-1_b&\frac{1}{2}\cdot 1_b&&\\&&1_b&\frac{1}{2}\cdot1_b&&\\&-1_a&&&-\frac{\zeta}{2}&\\-1_b&&&&&-\frac{1}{2}\cdot1_b\end{pmatrix}.
\end{equation}
We also define
$$S_\zeta=\begin{pmatrix}1_b&&&&&&&\\&1&&&&&&\\&&1_a&&&&-\frac{\zeta}{2}&\\&&&1_b&&&&
\\&&&1_b&&&&\\&&&&&1&&\\&&-1_a&&&&-\frac{\zeta}{2}
&\\&&&&&&&1_b\end{pmatrix},\ \tilde{S}=\begin{pmatrix}1_b&&&&&&&-\frac{1}{2}\cdot 1_b\\&1&&&&&&\\&&1_a&&&&&\\&&&-1_b&\frac{1}{2}\cdot 1_b&&&\\&&&1_b&\frac{1}{2}\cdot1_b &&&\\&&&&&1&&\\&&&&&&1_a
&\\-1_b&&&&&&&-\frac{1}{2}\cdot1_b\end{pmatrix}.$$
Let $\tau$ be a unitary idele class character of $\mathbb{A}_\mathcal{K}^\times$. Given a unitary tempered cuspidal eigenform $\varphi$ on $\mathrm{GU}(r,s)$ which is a pure tensor we formally define the integral
$$F_\varphi(f;z,g):=\int_{\mathrm{U}(r,s)(\mathbb{A}_F)} f(z,S^{-1}\alpha(g,g_1h)S)\bar\tau(\det g_1g)\varphi(g_1h)dg_1,$$
$$f\in I_{r+s+1}(\tau),g\in \mathrm{GU}(r+1,s+1)(\mathbb{A}_F),h\in \mathrm{GU}(r,s)(\mathbb{A}_F),\mu(g)=\mu(h).$$
This is independent of $h$. (We suppress the $\tau$ in the notation for $F_\varphi$ since its choice is implicitly given by $f$). We also formally define
$$F_\varphi'(f;z,g):=\int_{\mathrm{U}(r,s)(\mathbb{A}_F)} f(z,S^{\prime-1}\alpha(g,g_1h)S')\bar\tau(\det g_1g)\varphi(g_1h)dg_1,$$
$$f\in I_{r+s}(\tau),g\in \mathrm{GU}(r,s)(\mathbb{A}_F),h\in \mathrm{GU}(r,s)(\mathbb{A}_F),\mu(g)=\mu(h)$$
The pullback formulas are the identities in the following proposition.
\begin{proposition}
Let $\chi$ be a unitary idele class character of $\mathbb{A}_\mathcal{K}^\times$.\\
(i) If $f\in I_{r+s}(\tau),$ then $ F_\varphi(f;z,g)$ converges absolutely and uniformly for $(z,g)$ in compact sets of $\{\mathrm{Re}(z)>r+s\}\times \mathrm{GU}(r,s)(\mathbb{A}_F)$, and for any $h\in \mathrm{GU}(r,s)(\mathbb{A}_F)$ such that $\mu(h)=\mu(g)$
\begin{equation}
\int_{\mathrm{U}(r,s)(F)\setminus \mathrm{U}(r,s)(\mathbb{A}_F)} E(f;z,S'^{-1}\alpha(g,g_1h)S')\bar{\tau}(\det g_1h)\varphi(g_1h)dg_1=F_\varphi'(f;z,g).
\end{equation}
(ii) If $f\in I_{r+s+1}(\tau)$, then $F_\varphi(f;z,g)$ converges absolutely and uniformly for $(z,g)$ in compact sets of $\{\mathrm{Re}(z)>r+s+1/2\}\times \mathrm{GU}(r+1,s+1)(\mathbb{A}_F)$ such that $\mu(h)=\mu(g)$
\begin{equation}
\begin{split}
\int_{\mathrm{U}(r,s)(F)\setminus \mathrm{U}(r,s)(\mathbb{A}_F)} E(f;z,S^{-1}\alpha(g,g_1h)S)\bar{\tau}&(\det g_1h)\varphi(g_1h)dg_1\\
&=\sum_{\gamma\in P(F)\setminus G(r+1,s+1)(F)} F_\varphi(f;z,\gamma g),
\end{split}
\end{equation}
with the series converging absolutely and uniformly for $(z,g)$ in compact subsets of $\{\mathrm{Re}(z)>r+s+1/2\}\times \mathrm{GU}(r+1,s+1)(\mathbb{A}_F).$

\end{proposition}
See \cite[Proposition 3.5]{WAN}, which summarizes results proved in \cite{Shimura97}.
\subsection{Differential Operators}\label{Diffe}
Let $S/T$ be either the Igusa or Shimura variety, and let $A/S$ be the universal Abelian variety.

Let $\pi:X\rightarrow S$ be a smooth proper morphism of schemes, and let $S$ be a smooth scheme over a scheme $T$. Then the \emph{Gauss-Manin} connection is a map
$$\Delta:H^q_{\mathrm{DR}}(X/S)\rightarrow H^q_{\mathrm{DR}}(X/S).$$
By using the chain rule, we can also define
$$\Delta: \mathrm{Sym}^\bullet(H^{1\pm}_{\mathrm{DR}}(A/S))\rightarrow
\mathrm{Sym}^\bullet(H^{1\pm}_{\mathrm{DR}}(A/S)).$$
Here $\mathrm{Sym}^\bullet$ denotes the symmetric tensored powers. As in \cite{Eischen}, $H^{1\pm}_{\mathrm{DR}}$ denotes the submodules on which $\alpha\in\mathcal{K}$ acts via multiplication by $\alpha$ or $\bar{\alpha}$ respectively.\\

As in \cite[Section 7]{Eischen}, there is an algebraic differential operator
$$D^\rho_{A/S}: H^1_{\mathrm{DR}}(A/S)^\rho\otimes\mathrm{Sym}^\bullet
(H^{1+}(A/S)\otimes H^{1-}(A/S))\rightarrow H^1_{\mathrm{DR}}(A/S)^\rho\otimes\mathrm{Sym}^{\bullet+!}
(H^{1+}(A/S)\otimes H^{1-}(A/S)),$$
which is constructed from the Gauss-Manin connection and the Kodairo-Spencer morphism.\\

\noindent\underline{$C^\infty$ Differential Operators}\\
Over $\mathbb{C}$, there is a canonical splitting
$$H^1_{\mathrm{DR}}(C^\infty)=\underline{\omega}(C^\infty)
\oplus\mathrm{Split}(C^\infty)$$
of the Hodge decomposition corresponding to the holomorphic and anti-holomorphic one-forms. Let $\rho=\rho_-\otimes\rho_+$ be a representation of $\mathrm{GL}_n\times\mathrm{GL}_n$ which is quotient of $\mathrm{Sym}^{d_1}(\rho_{\mathrm{st}})\otimes
\mathrm{Sym}^{d_2}(\rho_{\mathrm{st}})$.

There is a $C^\infty$-differential operator
$$\partial(\rho, C^\infty,d): (\underline{\omega}^-)^{\rho_-}\otimes(\underline{\omega}^+)^{\rho_+}\rightarrow
\partial(\rho, C^\infty,d): (\underline{\omega}^-)^{\rho_-}\otimes(\underline{\omega}^+)^{\rho_+}\otimes (\mathrm{Sym}^{d_1}(\underline{\omega}^+)\otimes\mathrm{Sym}^{d_2}(\underline{\omega}^-),$$
defined as in \cite[Section 8]{Eischen}.

\begin{align*}
&(\underline{\omega}^-)^{\rho_-}\otimes(\underline{\omega}^+))^{\rho_+}\hookrightarrow&& H^1_{\mathrm{DR}}(A/S)^{\mathrm{Ind}^G_Q\rho}\rightarrow H^1_{\mathrm{DR}}(A/S)^{\mathrm{Ind}^G_Q\rho}\otimes
(\mathrm{Sym}^{d_1}(H^{1+}_{\mathrm{DR}}(A/S)\otimes\mathrm{Sym}^{d_2}(H^{1-}_{\mathrm{DR}}(A/S))&\\
&&&(\underline{\omega}^-)^{\rho_-}\otimes(\underline{\omega}^+))^{\rho_+}\otimes(\mathrm{Sym}^{d_1}(\underline{\omega}^-)
\otimes\mathrm{Sym}^{d_2}(\underline{\omega}^+)).&
\end{align*}
\noindent\underline{$p$-adic Differential Operators}\\
Now let $S$ be an Igusa scheme over a $p$-adic ring. As it is over the ordinary locus there is a ``unit root splitting''
$$H^1_{\mathrm{DR}}(A/S)=\underline{\omega}\oplus\underline{U},$$
where $\underline{U}$ is the unit root subspace for Frobenius action (see \cite[Section 9]{Eischen} for details). We can define a $p$-adic differential operator $\partial(\rho, p-\mathrm{adic}, d)$
$$(\underline{\omega}^-)^{\rho_-}\otimes(\underline{\omega}^+))^{\rho_+}\rightarrow (\underline{\omega}^-)^{\rho_-}\otimes(\underline{\omega}^+))^{\rho_+}\otimes(\mathrm{Sym}^{d_1}(\underline{\omega}^-)
\otimes\mathrm{Sym}^{d_2}(\underline{\omega}^+))$$
as for the $C^\infty$ cases, but with the $C^\infty$ splitting replaced by the unit root splitting.
\subsection{Archimedean Computations}
We summarize results in \cite[Section 4.1]{WAN}.
Let $v$ be an Archimedean place of $F$. Let $\kappa>0$ be an integer. Suppose $\tau$ is a unitary character of $\mathbb{C}^\times$ of infinity type $(0,0)$.
\begin{definition}
$$f_{\kappa,n}(z,g)=J_n(g,i1_n)^{-\kappa}\det(g)^{\frac{\kappa}{2}}|J_n(g,i1_n)|^{\kappa-2z-n}.$$
\end{definition}
Now we recall \cite[Lemma 11.4]{SU}. Let $J_n(g,i1_n):=\det (C_gi1_n+D_g)$ \index{$J_n$} for $g=\begin{pmatrix}A_g&B_g\\C_g&D_g\end{pmatrix}$.

\begin{lemma}\label{Fourier-Archimedean}
Suppose $\beta\in S_n(\mathbb{R})$. Then the function $z\rightarrow f_{\kappa,\beta}(z,g)$ has a meromorphic continuation to all of $\mathbb{C}$.
Furthermore, if $\kappa\geq n$ then $f_{\kappa,n,\beta}(z,g)$ is holomorphic at $z_\kappa:=(\kappa-n)/2$ and for $y\in \mathrm{GL}_n(\mathbb{C}),f_{\kappa,n,\beta}(z_\kappa,\mathrm{diag}(y,{}^t\!\bar{y}^{-1}))=0$ if $\det\beta\leq 0$ and if $\det\beta>0$ then
$$f_{\kappa,n,\beta}(z_\kappa,\mathrm{diag}(y,{}^t\!\bar{y}^{-1}))=\frac{(-2)^{-n}(2\pi i)^{n\kappa}(2/\pi)^{n(n-1)/2}}{\prod_{j=0}^{n-1}(\kappa-j-1)!}e_v(i\mathrm{Tr}(\beta y{}^t\!\bar{y}))\det(\beta)^{\kappa-n}\det\bar{y}^\kappa.$$
\end{lemma}

Now we look at some conjugation maps between unitary groups over $\mathbb{R}$. Write $\tau$ for a real symmetric positive definite matrix so that $\tau\tau^*=\frac{\zeta}{2i}$. We define
$$S_1=\begin{pmatrix}1_{b+1}&&&&&\\&1_a&&&-\frac{\zeta}{2}&\\&&1_b&&&\\&&&1_{b+1}
&&\\&-1_a&&&-\frac{\zeta}{2}&\\&&&&&1_b\end{pmatrix},\
S_2=\begin{pmatrix}\frac{1}{\sqrt{2}}1_{b+1}&&&-\frac{i}{\sqrt{2}}1_{b+1}&&\\&\tau^{-1}&&&&
\\&&\frac{1}{\sqrt{2}}1_b&&&-\frac{1}{\sqrt{2}}1_b\\-\frac{1}{\sqrt{2}}1_{b+1}&&&-\frac{i}{\sqrt{2}}&&
\\&&&&\tau^{-1}&\\&&-\frac{1}{\sqrt{2}}1_b&&&-\frac{i}{\sqrt{2}}1_b\end{pmatrix}.$$
Then for any $u\in \mathrm{U}(n+1,n+1)$ (unitary group corresponding to $\begin{pmatrix}&1_{n+1}\\-1_{n+1}&\end{pmatrix}$), the $S_1 u S^{-1}_1$ is in the unitary group $\mathrm{U}_2$ of Hermitian matrix
$$\begin{pmatrix}&&&1_{b+1}&&\\& \zeta &&&&\\&&&&&1_b\\-1_{b+1}&&&&&\\&&&&-\zeta&\\&&-1_b&&&\end{pmatrix}.$$
The $S_2S_1uS^{-1}_1S^{-1}_2$ is in the unitary group $\mathrm{U}_3$ of Hermitian matrix
$$i\begin{pmatrix}1_{b+1}&&&&&\\&1_a&&&&\\&&1_b&&&\\&&&-1_{b+1}&&\\&&&&-1_a&\\&&&&&-1_b\end{pmatrix}.$$

We discuss the pullback formula. We record that
\begin{align*}
&\begin{pmatrix}1_b&&&&&&&\\&1&&&&&&\\&&\frac{1}{2}1_a&&&&-\frac{1}{2}1_a&\\&&&1_b&&&&\\&&&&1_b&&&\\
&&&&&1&&\\&&-\zeta^{-1}&&&&-\zeta^{-1}&\\&&&&&&&1_b\end{pmatrix}
\begin{pmatrix}a_1&a_2&a_3&&b_1&b_2&&\\a_4&a_5&a_6&&b_3&b_4&&\\a_7&a_8&a_9&&b_5&b_6&&\\&&&A&&&B&C
\\c_1&c_2&c_3&&d_1&d_2&&\\c_4&c_5&c_6&&d_3&d_4&&\\&&&D&&&E&F\\&&&G&&&H&J\end{pmatrix}
\begin{pmatrix}1_b&&&&&&&\\&1&&&&&&\\&&1_a&&&&-\frac{\zeta}{2}&\\&&&1_b&&&&\\&&&&1_b&&&\\&&&&&1&&\\&&-1_a&&&&
-\frac{\zeta}{2}&\\&&&&&&&1_b\end{pmatrix}&\\
&=\begin{pmatrix}a_1&a_2&a_3&&b_1&b_2&-\frac{a_3\zeta}{2}&\\a_4&a_5&a_6&&b_3&b_4&-\frac{a_6\zeta}{2}&
\\ \frac{a_7}{2}&\frac{a_8}{2}&\frac{a_9+E}{2}&-\frac{D}{2}&\frac{b_5}{2}&\frac{b_6}{2}&
-\frac{a_9\zeta}{4}+\frac{E\zeta}{4}&-\frac{F}{2}\\&&-B&-A&&&-\frac{B\zeta}{2}&C\\c_1&c_2&c_3&&d_1&d_2&-\frac{c_3\zeta}{2}
&\\c_4&c_5&c_6&&d_3&d_4&-\frac{c_6\zeta}{2}&\\-\zeta^{-1}a_7&-\zeta^{-1}a_8&-\zeta^{-1}a_9+\zeta^{-1}E&-\zeta^{-1}D
&-\zeta^{-1}b_5&-\zeta^{-1}b_6&\frac{\zeta^{-1}(a_9+E)\zeta}{2}&-\zeta^{-1}F\\&&-H&G&&&-\frac{H\zeta}{2}&J\end{pmatrix}.&
\end{align*}
Let $g=\begin{pmatrix}A_g&B_g\\C_g&D_g\end{pmatrix}$ be the last matrix above and $\boldsymbol{i}=\mathrm{diag}(i1_b, i, \frac{\zeta}{2}, i1_b)$, then
$$C_g\boldsymbol{i}+D_g=\begin{pmatrix}c_1i+d_1&c_2i+d_2&0&0\\c_4i+d_3&c_5i+d_4&0&0\\-\zeta^{-1}ia_7-\zeta^{-1}b_5
&-\zeta^{-1}ia_8-\zeta^{-1}b_6&\zeta^{-1}E\zeta&-\zeta^{-1}iD-\zeta^{-1}F\\0&0&-H\zeta&Gi+J\end{pmatrix}.$$
Taking determinant, we get the decomposition for the automorphic factor
$$J(g,\boldsymbol{i})=\det(\begin{pmatrix}c_1&c_2\\c_4&c_5\end{pmatrix}\begin{pmatrix}i\\0\end{pmatrix}+
\begin{pmatrix}d_1&d_2\\d_3&d_4\end{pmatrix})\cdot\det(\begin{pmatrix}\zeta^{-1}E\zeta&\zeta^{-1}iD+\zeta^{-1}F
\\H\zeta&Gi+J\end{pmatrix}).$$
We also record the formula for embedding of Hermitian spaces
For $z=\begin{pmatrix}x\\y\end{pmatrix}$ and $w=\begin{pmatrix}u\\v\end{pmatrix}$, we define
$$\iota(z,w)=\begin{pmatrix}x&0&0\\y&\frac{\zeta}{2}&\\-v^*\zeta^{-1}y&-v^*&-u^*\end{pmatrix}.$$
This is compatible with the embedding
$(g_1, g_2)\mapsto S^{-1}_1\alpha(g_1, g_2)S_1$.
The differential operators can be described in terms of actions of Lie algebras of $\mathrm{U}(n+1,n+1)(\mathbb{R})$ as below. Write $I_{m,n}$ for the diagonal matrix $\mathrm{diag}(1_m,-1_n)$. We identify the complexification of the Lie algebra of $\mathrm{U}_{I_{n+1,n+1}}$ with $\mathrm{gl}_{2n+2}(\mathbb{C})$. Then under the Harish-Chandra decomposition
$$\mathfrak{su}(n+1,n+1)\simeq \mathfrak{k}\oplus p^+\oplus p^-,$$
the $p^+$ corresponds to matrices of the form $\begin{pmatrix}0&*\\ &0\end{pmatrix}$ (block matrices with respect to $((n+1)+(n+1))$), and $p^-$ corresponds to matrices of the form $\begin{pmatrix}0& \\ *&0\end{pmatrix}$ (block matrices with respect to $((n+1)+(n+1))$). To relate this with the $\mathrm{GL}_{n+1}\times\mathrm{GL}_{n+1}$ representation in the definition of algebraic differential operators of $\mathrm{U}(n+1,n+1)$, we regard the space of upper right $(n+1)\times(n+1)$ matrices $\underline{X}=\begin{pmatrix}\underline{X}_1&\underline{X}_2\\ \underline{X}_3&\underline{X}_4\end{pmatrix}$ (the action of $\mathrm{GL}_{n+1}\times\mathrm{GL}_{n+1}$ given by left and right multiplications).

\begin{proposition}
For $f$ a $\rho$-valued nearly holomorphic automorphic form, and $v$ a vector in $\rho^\vee$, we have
$$\langle Df, v\otimes x\rangle=\frac{1}{2}S_1^{-1}S_2^{-1}\begin{pmatrix}0&ix\\0&0\end{pmatrix}S_2S_1\cdot\langle f,v\rangle.$$
\end{proposition}
This is proved as in \cite[Proposition 2.4.1]{Zliu1}.

We define a differential operator. Let $\underline{X}_2:=\begin{pmatrix}A&B\\C&D\end{pmatrix}$ be a block matrix with respect to the partition $a+b$ and let $\underline{X}_3:= E$ be a matrix of size $b\times b$. Let $v_{\underline{k},\kappa}$ be the polynomial
$$(\prod_{i=1}^{r-1}\det(X_2)^{a_i-a_{i+1}}_i)\det(X_2)^{a_r-\kappa}_r\cdot\prod_{j=1}^{s-1}\det(X_3)^{b_{s-j+1}-b_{s-j}}_j
\det(X_3)^{b_1}_s.$$
We use the simple notation $D$ to denote the $C^\infty$ differential operators $\partial$ in Section \ref{Diffe}.

\begin{definition}
With above proposition, we can define an element $\delta_{\underline{k},\kappa}$ in the Lie algebra of $\mathrm{U(n+1,n+1)}$ corresponding to the map from the space of holomorphic weight $\kappa$ forms as
$$F_\kappa\mapsto \langle D^d F_\kappa, v_{\underline{k},\kappa}\rangle$$
for $d=a_1+\cdots+a_r-r\kappa+b_1+\cdots+b_r$.
We also define the Siegel section
$$f_{sieg,\underline{k},\kappa}:=\delta_{\underline{k},\kappa}f_\kappa.$$
We similarly define $\delta'_{\underline{k},\kappa}$ and the Siegel section $f'_{sieg,\underline{k},\kappa}$ on $\mathrm{U}(n,n)$.
\end{definition}
We have the following lemma.
\begin{lemma}\label{ArchiFourier}
Write $\beta$ as block matrices $\begin{pmatrix}A_\beta&B_\beta&C_\beta\\D_\beta&E_\beta&F_\beta\\G_\beta&H_\beta&J_\beta\end{pmatrix}$ with respect to $(r+1+s)\times(s+1+r)$. The $\beta$-th Fourier coefficient of the highest weight vector of the $V_{\underline{k}}$-valued form $(D^dF_\kappa)_{\underline{k}}$, which we denote as $D_{\underline{k}}F_\kappa$, is given by
$$\det(C_{\beta,1})^{a_1-a_2}\det(C_{\beta,2})^{a_2-a_3}\cdots \det(C_{\beta,r})^{a_r-a_{r+1}}\det(G_{\beta,1})^{b_{s+1}-b_s}\cdots\det(G_{\beta,s})^{b_2-b_1}
F_{\kappa,\beta}.$$
Here $C_{\beta,i}$ and $G_{\beta,i}$ are the upper left $i\times i$ minors of $C_\beta$ and $G_\beta$ respectively.
\end{lemma}
The proof is as \cite[Proposition 5.3]{EW}, which uses \cite[Theorem 9.2 (4)]{Eischen}.

We define a Weyl element
\begin{equation}\label{Weylelement}
w'_{r+1,s+1}=\begin{pmatrix}&1_{b+1}&&\\1_a&&&\\&&1_b&\\&&&1\end{pmatrix}.
\end{equation}
We consider the unitary group $\mathrm{U}_{I_{r+1,s+1}}$ corresponding to the Hermitian matrix $\mathrm{diag}(1_{r+1}, 1_{s+1})$. Then the $w'_{r+1,s+1}$ above is in the compact group $\mathrm{U}(r+1)(\mathbb{R})\times \mathrm{U}(s+1)(\mathbb{R}) \hookrightarrow \mathrm{U}_{I_{r+1,s+1}}(\mathbb{R})$.
We can write
$$S^{-1}_2(\alpha(w'_{r+1.s+1}, 1)S_2=w''_{r+1,s+1}\otimes 1\in \mathrm{U}(r+1,s+1)(\mathbb{R})\times 1.$$
We have
\begin{align*}&\alpha(w'_{r+1,s+1}, 1_{r+s})\times\begin{pmatrix}1_{n+1}&\begin{matrix}&&A&B\\&&C&D\\&&&\\E&&&\end{matrix}\\&1_{n+1}\end{pmatrix}& &\times\mathrm{diag}(w'_{r+1,s+1}, 1_{r+s})^{-1}=\begin{pmatrix}1_{n+1}&\begin{matrix}&&C&D\\&&&\\&&A&B\\E&&& \end{matrix}\\&1_{n+1}\end{pmatrix}.&\end{align*}
So we have the following proposition.
\begin{proposition}
We have
$$F_\varphi(f_{sieg,\underline{k},\kappa},w''_{r+1,s+1},z)=F'_\varphi(f'_{sieg,\underline{k},\kappa}, 1,z+\frac{1}{2}).$$
\end{proposition}
We also have the following lemma.
\begin{lemma}\label{Archimedean constant}
Let $\varphi$ be the lowest weight module of the holomorphic discrete series with weight $\underline{k}$. Then there is a nonzero constant $c_{\underline{k},\kappa}'$ such that
$$F'_\varphi(f_{sieg,\underline{k},\kappa}, 1, \frac{\kappa-a-2b}{2})=c'_{\underline{k},\kappa}\varphi.$$
\end{lemma}
\begin{proof}
The only non-trivial statement is about the non-vanishing of $c'_{\underline{k},\kappa}$, which is a well known fact as noted in \cite[Section 4.5]{EHLS}. We also remark that this constant is explicitly computable, thanks to a recent technique developed by Z. Liu \cite{ZLiu}.
\end{proof}

\subsection{Finite Primes, Unramified Case}
We summarize results in \cite[Section 4.2]{WAN}.
\subsubsection{Pullback Integrals}
\begin{lemma}
Suppose $\pi,\psi$ and $\tau$ are unramified and $\varphi\in\pi$ is a new vector. If $\mathrm{Re}(z)>(a+b)/2$ then the pullback integral converges and
$$F_\varphi(f_v^{sph};z,g)=
\frac{L(\tilde\pi,\bar\tau^c,z+1)}{\prod_{i=0}^{a+2b-1}L(2z+a+2b+1-i,\bar\tau'\chi_\mathcal{K}^i)}
F_{\rho,z}(g)$$
where $F_{\rho,z}$ is the spherical section taking value $\varphi$ at the identity and
$$F_\varphi(f_v^{sph};z,g)=
\frac{L(\tilde\pi,\bar\tau^c,z+\frac{1}{2})}{\prod_{i=0}^{a+2b-1}L(2z+a+2b-i,\bar\tau'\chi_\mathcal{K}^i)}
\pi(g)\varphi.$$
\end{lemma}
The local Fourier-coefficient is given below.
\begin{lemma}\label{Fourier-Unramified}
Let $\beta\in S_n(F_v)$ and let $r:=\mathrm{rank}(\beta)$. Then for $y\in \mathrm{GL}_n(\mathcal{K}_v)$,
\begin{eqnarray}\label{urd}
f_{v,\beta}^{sph}(z,diag(y,{}^t\!\bar y^{-1}))=&\tau({\det} y)|{\det} y\bar y|_v^{-z+n/2}D_v^{-n(n-1)/4}\\
&\times \frac{\prod_{i=r}^{n-1} L(2z+i-n+1,\bar\tau'\chi_\mathcal{K}^i)}{\prod_{i=0}^{n-1}L(2z+n-i,\bar\tau'\chi_\mathcal{K}^i)}h_{v,{}^t\!\bar y\beta y}(\bar\tau'(\varpi)q_v^{-2z-n}).
\end{eqnarray}
where $h_{v,{}^t\!\bar y\beta y}\in \mathbb{Z}[X]$ is a monic polynomial depending on $v$ and ${}^t\!\bar y\beta y$ but not on $\tau$. If $\beta\in S_n(\mathcal{O}_{F,v})$ and $\det \beta\in \mathcal{O}_{F,v}^\times$, then we say that $\beta$ is $v$-primitive and in this case $h_{v,\beta}=1$.
\end{lemma}
To study functional equations we need another definition
\begin{definition}
$$f^{\mathrm{fteq}}_{v,z}=\prod_{i=1}^{r+s+1}\frac
{L(z-r-s-1+i,\chi_v,\chi_{\mathcal{K}/F,v})}{L(1-z+r+s+1-i,(\chi_v\chi^i_{\mathcal{K}/F,v})^{-1})}M(f^{\mathrm{sph}}_v,-z)_z,$$
$$f^{\mathrm{fteq},\prime}_{v,z}=\prod_{i=1}^{r+s}\frac
{L(z-r-s+i,\chi_v,\chi_{\mathcal{K}/F,v})}{L(1-z+r+s-i,(\chi_v\chi^i_{\mathcal{K}/F,v})^{-1})}M(f^{\mathrm{sph}}_v,-z)_z.$$
\end{definition}
\subsection{Prime to $p$ Ramified Case}
We summarize the results in \cite[Section 4.3]{WAN}.
\subsubsection{Pullback integrals}\label{Ramified Pullback}
Again let $v$ be a prime of $F$ not dividing $p$. We fix some $x$ and $y$ in $\mathcal{K}$ which are divisible by some high power of $\varpi_v$ (can be made precise from the proof of the following two lemmas). (When we are moving things $p$-adically the $x$ and $y$ are not going to change). We define $f^\dag\in I_{n+1}(\tau)$ to be the Siegel section supported on the cell $Q(F_v)w_{a+2b+1}N_Q(\mathcal{O}_{F,v})$ where $w_{a+2b+1}=\begin{pmatrix}&1_{a+2b+1}\\-1_{a+2b+1}&\end{pmatrix}$ and the value at
$N_Q(\mathcal{O}_{F,v})$ equals $1$.  Similarly we define $f^{\dag,\prime}\in I_n(\tau)$ to be the section supported in $Q(F_v)w_{a+2b}N_Q(\mathcal{O}_{F,v})$ and takes value $1$ on $N_Q(\mathcal{O}_{F,v})$.
\begin{definition}\label{xyv}
$$f_{sieg,v}(g)=f_{x,y,v}(g):=f^\dag(g\tilde{S}_v^{-1}
\tilde{\gamma}_v)\in I_{n+1}(\tau)$$
where $\tilde{\gamma}_v$ is defined to be:
$$
\begin{pmatrix}1_b&&&&&&&\frac{1}{x}1_b\\&1&&&&&&\\&&1_a&&&&\frac{1}{y\bar{y}}1_a&\\&&&1_b&\frac{1}{\bar{x}}1_b
&&&\\&&&&1_b&&&\\&&&&&1&&\\&&&&&&1_a&\\&&&&&&&1_b\end{pmatrix}
$$ and
$$\tilde{S}_v=\begin{pmatrix}1_b&&&&&&&-\frac{1}{2}1_b\\&1&&&&&&\\&&1&&&&&\\&&&-1_b&\frac{1}{2}1_b
&&&\\&&&1_b&\frac{1}{2}1_b&&&\\&&&&&1&&\\&&&&&&1_a&\\-1_b&&&&&&&-\frac{1}{2}1_b\end{pmatrix}.$$
Similarly we define $f_{sieg,v}'(g)=f'_{x,y,v}(g):=f^{\dag,\prime}(g\tilde{S}_v^{-1}\tilde{\gamma}_v')$ for
$$\tilde{S}_v':=\begin{pmatrix}1_b&&&&&-\frac{1}{2}1_b\\&1_a&&&&\\&&-1_b&\frac{1}{2}1_b&&\\&&1_b&\frac{1}{2}1_b
&&\\&&&&1_a&\\-1_b&&&&&-\frac{1}{2}1_b\end{pmatrix}$$
and
$$\tilde{\gamma}_v=\begin{pmatrix}1_b&&&&&\frac{1}{x}1_b\\&1_a&&&\frac{1}{y\bar{y}}1_a&\\&&1_b&
\frac{1}{\bar{x}}1_b&&\\&&&1_b&&\\&&&&1_a&\\&&&&&1_b\end{pmatrix}.$$
\end{definition}
\index{$f_{sieg,v}$}

\begin{lemma}
Let $K_v^{(2)}$ be the subgroup of $G(F_v)$ of the form $\begin{pmatrix}1_b&&&&d\\a&1&f&b&c\\&&1_a&&g\\&&&1_b&e\\&&&&1\end{pmatrix}$ where $e=-{}^t\!\bar{a}$, $b={}^t\!\bar{d}$, $g=-\zeta{}^t\!\bar{f}$, $b\in M(\mathcal{O}_v)$, $c-f\zeta{}^t\!\bar{f}\in\mathcal{O}_{F,v}$, $a\in (x)$, $e\in (\bar x)$, $f\in(y\bar y)$, $g\in (2\zeta y\bar y)$. Then $F_\varphi(f_{v,sieg}; z, g)$ is supported in $PwK_v^{(2)}$ and is invariant under the action of $K_v^{(2)}$.
\end{lemma}
\begin{definition}\label{rfd}
Write $g=\begin{pmatrix}a_5&a_6&a_4\\a_8&a_9&a_7\\a_2&a_3&a_1\end{pmatrix}$. Let $\mathfrak{Y}$ be the set of $g$'s so that the entries of $a_2$ are
integers, the entries of $a_3$ are divisible by $y\bar y$, the entries of $a_1-1$ are divisible by $\bar x$, the entries of $1-a_5$ are divisible by $x$, the entries of $a_6$ are divisible by $\bar{x}y\bar{y}$ (note the typo in \cite[Section 4.3]{WAN}), the entries of $a_4$ are divisible by $x\bar x$, ${1-a_9}=y\bar y\zeta(1+y\bar yN)$ for some $N$ with integral entries, the entries of $a_8$ are divisible by $\frac{\bar{y}y\zeta}{2}$, and the entries of $a_7$ are divisible by
$\bar{y}yx\zeta$.
\end{definition}
\begin{lemma}
Let $\varphi_x=\pi(\mathrm{diag}(\bar x,1,x^{-1})\eta^{-1})\varphi$ where $\varphi$ is invariant under the action of $\mathfrak{Y}$ defined above, then
\begin{itemize}
\item[(i)] $F_{\varphi_x}(f_{sieg,v};z,w)=\tau(y\bar y x)|(y\bar y)^2x\bar x|_v^{-z-\frac{a+2b+1}{2}}\mathrm{Vol}(\mathfrak{Y})\cdot\varphi$.
\item[(ii)] $F_{\varphi_x}'(f_{sieg,v}';z,w)=\tau(y\bar y x)|(y\bar y)^2x\bar x|_v^{-z-\frac{a+2b}{2}}\mathrm{Vol}(\mathfrak{Y})\cdot\varphi$.
\end{itemize}
\end{lemma}
The local Fourier-coefficient is given as below.
\begin{lemma}\label{Fourier-Ramified}
(i) Let $\beta=(\beta_{ij})\in S_{n+1}(F_v)$ then for all $z\in \mathbb{C}$ we have:
$$\tilde{f}_{sieg,v,\beta}(z,1)=\mathrm{Vol}(S_{n+1}(\mathcal{O}_{F,v}))e_v(\mathrm{Tr}_{\mathcal{K}_v/F_v}(\frac{\beta_{a+b+2,1}+...+\beta_{a+2b+1,b}}{x})
+\frac{\beta_{b+2,b+2}+...+\beta_{b+1+a,b+1+a}}{y\bar{y}}).$$
(ii) Let $\beta=(\beta_{ij})\in S_v(F_v)$. Then
$$\tilde{f}_{sieg,v,\beta}'(z,1)=\mathrm{Vol}(S_{n}(\mathcal{O}_{F,v}))e_v(\mathrm{Tr}_{\mathcal{K}_v/F_v}(\frac{\beta_{a+b+1,1}+...+\beta_{a+2b,b}}{x}
)+\frac{\beta_{b+1,b+1}+...+\beta_{b+a,b+a}}{y\bar{y}}).$$
\end{lemma}
As before we make the following definition.
\begin{definition}
$$f^{\mathrm{fteq}}_{v,z}=\prod_{i=1}^{r+s+1} \epsilon(z-r-s-1+i,\chi_v\chi^i_{\mathcal{K}/F,v},\psi_v)^{-1}\frac
{L(z,\chi_v,\chi_{\mathcal{K}/F,v})}{L(1-z,(\chi_v\chi^i_{\mathcal{K}/F,v})^{-1})}M(f_{sieg,v},-z)_z,$$
$$f^{\mathrm{fteq},\prime}_{v,z}=\prod_{i=1}^{r+s} \epsilon(z-r-s+i,\chi_v\chi^i_{\mathcal{K}/F,v},\psi_v)^{-1}\frac
{L(z,\chi_v,\chi_{\mathcal{K}/F,v})}{L(1-z,(\chi_v\chi^i_{\mathcal{K}/F,v})^{-1})}M(f'_{sieg,v},-z)_z.$$
\end{definition}
\subsection{$p$-adic Computations}\label{p adic}
\noindent Let $v|p$ be a prime of $F$ and $\mathcal{K}_v\simeq \mathbb{Q}_p\times \mathbb{Q}_p$. Let $\tau$ be character of $\mathbb{Q}_p^\times\times\mathbb{Q}_p^\times$. Suppose $\tau=(\tau_1,\tau_2^{-1})$ and let $p^{s_i}$ be the conductor of $\tau_i,i=1,2$.
Let $\chi_1,...\chi_a,\chi_{a+1},...\chi_{a+2b}$ be characters of $\mathbb{Q}_p^\times$ such that $\pi_v$ is isomorphic to $\pi(\chi_1,\chi_2,\cdots,\chi_n)$ whose conductors are $p^{t_1},...,p^{t_{a+2b}}$. Suppose the ordering of the $\chi_i$'s corresponds to the ordinary stabilization as discussed before \cite[Definition 4.42]{WAN}. Suppose we are in the:
\begin{definition}\label{generic}
(Generic case of \cite[Definition 4.21]{WAN}):
$$t_1>t_2>...>t_{a+b}>s_1>t_{a+b+1}>...>t_{a+2b}>s_2.$$
Also, let $\xi_i=\chi_i\tau_1^{-1}$ for $1\leq i\leq a+b$,
$\xi_j=\chi_j^{-1}\tau_2$ for $a+b+2\leq j\leq a+2b+1$. Let $\xi_{a+b+1}=1$.
\end{definition}
Let $w_{Borel}$ be the matrix $\begin{pmatrix}&1_{b+1}&&\\1_a&&&\\&&1_b&\\&&&1_{a+2b+1}\end{pmatrix}$ and $w'_{Borel}$ be $\begin{pmatrix}&1_{b}&&\\1_a&&&\\&&1_b&\\&&&1_{a+2b}\end{pmatrix}$.
Let
\begin{equation}\label{c_n}
c_n(\tau',z):=\left\{\begin{array}{ll}\tau'(p^{nt})p^{2ntz-tn(n+1)/2}&t>0\\p^{2nz-n(n+1)/2}&t=0.\end{array}\right.
\end{equation}
Suppose $(p^t)=\mathrm{cond}(\tau')$ for $t\geq 1$ then define $\tilde{f}_t$ to be the section supported in $Q(\mathbb{Q}_p)K_Q(p^t)$ and
$\tilde{f}_t(k)=\tau(\det d_k)$ on $K_Q(p^t)$. (The $K_Q(p^t)$ stands for the subgroup of $\mathrm{GL}_{2n}(\mathbb{Z}_p)$ consisting of elements which are block-wise ($n+n$) upper triangular modulo $p^t$).

We define the Siegel Eisenstein sections $f_{sieg,v}$ as the $f^0(z,g)$ below.
\begin{align*}
&f^0(z,g)=&&\frac{1}{c_{n+1}(\tau_p',-z-\frac{1}{2})\mathfrak{g}(\tau_p')^{n+1}}p^{-\sum_{i=1}^{a+b} it_i-\sum_{i=1}^bit_{a+b+i}}\prod_{i=1}^{a+b}\mathfrak{g}(\xi_i)\xi_i(-1)\prod_{i=1}^b\mathfrak{g}(\xi_{a+b+1+i})\xi_{a+b+1+i}(-1)&\\
&&&\times\sum_{A,B,C,D,E}\prod_{i=1}^a\bar\xi_i(\frac{\det A_{i}}{\det A_{i-1}}p^{t_i})\prod_{i=1}^b\bar\xi_{a+i,a+i}(\frac{\det D_{i}}{\det D_{i-1}}p^{t_{a+i}})
\times\prod_{i=1}^b\bar\xi_{a+b+1+i}(\frac{\det E_{i}}{\det E_{i-1}}p^{t_{a+b+i}})&\\
&&&\times \tilde{f}_t(z,gw_{Borel}^{-1}\begin{pmatrix}\begin{matrix}1_b&&&\\&1&&\\&&1_a&\\&&&1_b\end{matrix}&
\begin{matrix}&&C&D\\&&&\\&&A&B\\E&&&\end{matrix}\\&\begin{matrix}1_b&&&\\&1&&\\&&1_a&\\&&&1_b\end{matrix}\end{pmatrix}w_{Borel}).&
\end{align*}
Here $A_i$ is the $i$-th upper-left minor of $A$, $D_i$ is the $(a+i)$-th upper left minor of $\begin{pmatrix}A&B\\C&D\end{pmatrix}$, $E_i$ is the $i$-th upper-left minor of $E$. We have
$$w_{Borel}^{-1}\begin{pmatrix}\begin{matrix}1_b&&&\\&1&&\\&&1_a&\\&&&1_b\end{matrix}&
\begin{matrix}&&C&D\\&&&\\&&A&B\\E&&&\end{matrix}\\&\begin{matrix}1_b&&&\\&1&&\\&&1_a&\\&&&1_b\end{matrix}\end{pmatrix}w_{Borel}
=\begin{pmatrix}\begin{matrix}1_b&&&\\&1&&\\&&1_a&\\&&&1_b\end{matrix}&
\begin{matrix}&&A&B\\&&C&D\\&&&\\E&&&\end{matrix}\\&\begin{matrix}1_b&&&\\&1&&\\&&1_a&\\&&&1_b\end{matrix}\end{pmatrix}.$$
Note that in the last matrix, the upper-right block is with respect to $(a+b+1+b)\times (b+1+a+b)$.

We also define $f'_{sieg,v}$ by
\begin{align*}
&f^{0\prime}(z,g)=&&\frac{1}{c_n(\tau_p',-z)\mathfrak{g}(\tau_p')^n}p^{-\sum_{i=1}^{a+b} it_i-\sum_{i=1}^bit_{a+b+i}}\prod_{i=1}^{a+b}\mathfrak{g}(\xi_i)\xi_i(-1)\prod_{i=1}^b\mathfrak{g}
(\xi_{a+b+1+i})\xi_{a+b+1+i}(-1)&\\
&&&\times\sum_{A,B,C,D,E}\prod_{i=1}^a\bar\xi_i(\frac{\det A_{i}}{\det A_{i-1}})\prod_{i=1}^b\bar\xi_{a+i,a+i}(\frac{\det D_{i}}{\det D_{i-1}})
\times\prod_{i=1}^b\bar\xi_{a+b+1+i}(\frac{\det E_{i}}{\det E_{i-1}})&\\
&&&\times \tilde{f}_t(z,gw_{Borel}^{\prime-1}\begin{pmatrix}\begin{matrix}1_b&&\\&1_a&\\&&1_b\end{matrix}&\begin{matrix}&C&D\\&A&B\\E&&\end{matrix}
\\&\begin{matrix}1_b&&\\&1_a&\\&&1_b\end{matrix}\end{pmatrix}w_{Borel}').&
\end{align*}

The corresponding pullback section is the nearly ordinary section such that
$F_{\varphi'}(f^0,z,w_{Borel})$ is given by
\begin{eqnarray*}
&\bar{\tau}^c((p^{t_1+...+t_{a+b}},p^{t_{a+b+1}+...+t_{a+2b}}))|p^{t_1+...+t_{a+2b}}|^{-z-\frac{a+2b+1}{2}}\mathrm{Vol}(\tilde K')\\
&\times p^{-\sum it_i-\sum it_{a+b+i}}\prod_{i=r+1}^{r+s}\mathfrak{g}(\chi_i^{-1}\tau_2)\chi_i\tau_2^{-1}(p^{t_i})\prod_{j=1}^r\mathfrak{g}(\chi_j\tau_1^{-1})
\chi_j^{-1}\tau_1(p^{t_j})\epsilon(\pi,\tau^c,z)\varphi.\\
\end{eqnarray*}
Also we have $F_{\varphi'}'(z,\rho(\Upsilon')f^{0'},w_{Borel}')$ is given by
\begin{eqnarray*}
&\bar{\tau}^c((p^{t_1+...+t_{a+b}},p^{t_{a+b+1}+...+t_{a+2b}}))|p^{t_1+...+t_{a+2b}}|^{-z-\frac{a+2b}{2}}\mathrm{Vol}(\tilde K')\\
&\times p^{-\sum it_i-\sum it_{a+b+i}}\prod_{i=r+1}^{r+s}\mathfrak{g}(\chi_i^{-1}\tau_2)\chi_i\tau_2^{-1}(p^{t_i})\prod_{j=1}^r\mathfrak{g}(\chi_j\tau_1^{-1})
\chi_j^{-1}\tau_1(p^{t_j})\epsilon(\pi,\tau^c,z+\frac{1}{2})\varphi.\\
\end{eqnarray*}
We define the Siegel section used for the functional equation.
\begin{definition}
For $n=r+s$ or $r+s+1$, let $\tilde{f}^\dag_v$ be the Siegel section supported in $Q_{n}(\mathbb{Q}_pw_nN_{\mathbb{Q}_n})(\mathbb{Z}_p)$ taking the constant function $1$ on $w_nN_{\mathbb{Q}_n}(\mathbb{Z}_p)$. We define $f^{\mathrm{fteq}}_v$ as the definition of $f_{sieg,v}$ but replacing $\frac{1}{c_{n+1}(\tau_p',-z-\frac{1}{2})\mathfrak{g}(\tau_p')^{n+1}}\tilde{f}_t$ by $\tilde{f}^\dag_v$. We define $f^{\mathrm{fteq},\prime}_v$ similarly.
\end{definition}

We need also to study the pullback section of $f^{\mathrm{fteq}}_v$ at a special element. The following simply lemma enables us to reduce it to the computation of Harris-Eischen-Li-Skinner.
\begin{lemma}
We have
$$F(f^{\mathrm{fteq}}_v, ww_{Borel},z)=F'(f^{\mathrm{fteq},\prime}_v, w,z).$$
\end{lemma}
\begin{proof}
It follows easily from looking at the action of $w_{Borel}$ on the Siegel section $f^{\mathrm{fteq}}_v$.
\end{proof}
Now we record the local Fourier coefficient.
Let $\mathfrak{X}$ be the following subset of $M_{r+s+1}(\mathbb{Q}_p)$: if the block matrix $x=\begin{pmatrix}A_x&*&B_x\\ *&*&*\\C_x&*&D_x\end{pmatrix}$ (with respect to $(s+1+r)\times (r+1+s)$), then:\\
-\ \ $x$ has entries in $\mathbb{Z}_p$;\\
-\ \ $C_x$ has the $i$-th-upper-left minors $C_i$ such that $(\det C_i)\in \mathbb{Z}_p^\times$ for $i=1,2,...,r$;\\
-\ \ and $B_x$ has $i$-upper-left minors $B_i$ so that $(\det B_i)\in\mathbb{Z}_p^\times$ for $i=1,2,...,s$.
We define a function
\begin{equation}
\Phi_\xi(x)=\left\{\begin{array}{ll} 0&x\not \in \mathfrak{X}, \\ \xi_1/\xi_2(\det C_1)...\xi_{r-1}/\xi_{r}(\det C_{r-1})\xi_{r}(C_{r})&\\ \times\xi_{a+b+2}/\xi_{a+b+3}(\det B_1)...\xi_{r+s}/\xi_{r+s+1}(\det B_{s-1})\xi_{r+s+1}(\det B_s).& x\in \mathfrak{X}.\end{array}\right.
\end{equation}
The following is \cite[Lemma 4.46]{WAN}.
\begin{lemma}\label{Fourier-p-adic}
Suppose $|\det \beta|\not=0$ then:\\
(i) If $\beta\not\in S_{a+2b+1}(\mathbb{Z}_p)$ then $f^0_\beta(z,1)=0$;\\
(ii) Let $t:=\mathrm{ord}_p(\mathrm{cond}(\tau').$ If $\beta\in S_{a+2b+1}(\mathbb{Z}_p)$, then:
$$f^0_\beta(z,1)=\bar{\tau}'(\det\beta)|\det\beta|_p^{2z}\Phi_\xi({}^t\!\beta).$$
\end{lemma}
\subsection{Auxiliary Prime}\label{Auxi}
We take an auxiliary prime $v$ which splits as $w\bar{w}$ in $\mathcal{K}/F$ such that our Eisenstein datum is unramified at $v$. We need to choose different sections so that the $\beta$-th local Fourier coefficient at $v$ is identically zero if $\det\beta=0$. This is important for our application of the Kudla-Sweet result to get the $p$-adic functional equation. We also need to have an explicit description of the resulting pullback sections. The key idea is to work with Siegel-Weil sections and try to reduce the computation to simpler cases using Godement-Jacquet functional equation, as in \cite[Section 4.3]{EHLS}. As promised in the introduction, we now choose $v$ so that the $n$ Satake parameters of $\pi_v$ are pairwise distinct. We first prove the following lemma.
\begin{lemma}
There exists a prime $v=w\bar{w}$ of $F$ split in $\mathcal{K}$, such that the $\pi$, $\chi$, and $\mathcal{K}$ are unramified at $v$, and the local Satake parameters for $\pi_v$ are pairwise distinct.
\end{lemma}
\begin{proof}
By our ordinarity assumption of $\pi$, the Satake parameters at $p$ are pairwise distinct. We take a prime $\ell$ outside $p$ and consider the $\ell$-adic Galois representation $\rho_{\pi,\ell}$ attached to $\pi$. There is a prime $v$ such that the images of $\mathrm{Frob}_p$ and $\mathrm{Frob}_{v}$ under $\rho_{\pi,\ell}$ are sufficiently close in the $\ell$-adic topology so that $\rho_{\pi,\ell}(\mathrm{Frob}_v)$ has distinct eigenvalues. This $v$ satisfies our needs.
\end{proof}

Let $\varpi_v$ be an uniformizer at $v$. We first define several Schwartz functions.
\begin{definition}\label{defSch}
For convenience of the presentation in this definition, we use the block matrices for $\mathrm{GL}_{2n+2}$ and $\mathrm{GL}_{n+1}$ with respect to the partition $(1+b+a+b+1+b+a+b)$ and $(1+b+a+b)$ respectively. This means we switch the corresponding rows and columns in the unitary groups.

We use the superscript $(n)$ or $(n+1)$ to denote Schwartz functions on the set of $n\times n$ or $(n+1)\times(n+1)$ matrices. Let $\Phi_1^{(n+1)}$ and $\Phi_1^{(n)}$ be the characteristic function of the set of matrices which are congruent to identity modulo $\varpi_v$ (which we denote as $\Gamma_{n+1}$ and $\Gamma_n$).

We define $\hat\Phi^{(n+1)}_2$ to be the characteristic function of the set of matrices of the form
$\begin{pmatrix}E&D&F\\B&A&C\\H&G&J\end{pmatrix}$ (block matrices with respect to $(1+s+r)\times(1+r+s)$) described below. The $A$, $B$ and $D$ have entries divisible by $\varpi_v$; the $E$, $F$, $H$ and $J$ has entries in $\mathcal{O}_{F,v}$; The $C$ is in $\mathrm{GL}_b(\mathcal{O}_{F,v})$ and is lower triangular modulo $\varpi_v$; the $G$ is in $\mathrm{GL}_b(\mathcal{O}_{F,v})$ and is upper triangular modulo $\varpi_v$.

Let $w_0$ be the identity Weyl element in general linear groups. Define $\hat\Phi^{(n)}_{2,w_0,w_0}$ to be the characteristic function of the set of matrices $\begin{pmatrix}A&B\\ C&D\end{pmatrix}$ (block matrices with respect to $(r+s)\times(s+r)$) such that $A$ has entries divisible by $\varpi_v$, $D$ has entries in $\mathcal{O}_{F,v}$. The matrices in $B$ are in $\mathrm{GL}_n(\mathcal{O}_{F,v})$ which are lower triangular modulo $\varpi_v$. The matrices in $C$ are in $\mathrm{GL}_n(\mathcal{O}_{F,v})$ which are upper triangular modulo $\varpi_v$.

For $1\leq j_1\leq r$, we define $w_{j_1}$ to be the Weyl element in $\mathrm{GL}_{r+1}$ corresponding to the simple switch between the 1st and $1+j_1$-th element. We define the Weyl element $w_{j_2}\in\mathrm{GL}_{s+1}$ for $1\leq j_2\leq s$ similarly.
We define $\hat{\Phi}^{(n)}_{2,w_{j_1},w_{j_2}}$ to be the characteristic function of the set of matrices $\begin{pmatrix}A&B\\C&D\end{pmatrix}$ which we describe below. The $(j_1, j_2)$-th entry of $A$ is in $\mathcal{O}_{F,v}$ while other entries are divisible by $\varpi_v$. The $D$ has entries in $\mathcal{O}_{F,v}$. The $B$ is such that the $j_1$-th row has entries in $\mathcal{O}_{F,v}$; for $j\not=j_1$, $B_{jj}\in\mathcal{O}^\times_{F,v}$; $B_{j,k}$ are divisible by $\varpi_v$ if $k>j$; $B_{j,k}$ are in $\mathcal{O}_{F,v}$ if $k<j$. The $C$ is such that the $j_2$-th column has entries in $\mathcal{O}_{F,v}$; For $j\not=j_2$, $C_{jj}$ has entries in $\mathcal{O}^\times_{F,v}$; the $C_{kj}$ are in $\mathcal{O}_{F,v}$ if $k<j$; the $C_{kj}$ are divisible by $\varpi_v$ if $k>j$. We also define $\hat{\Phi}^{(n),\prime}_{2,w_{j_1},w_{j_2}}$ by requiring the entries in $A$ are in $\mathcal{O}_{F,v}$, and the same requirement as the definition of $\hat{\Phi}^{(n)}_{2,w_{j_1},w_{j_2}}$ on the $B$, $C$ and $D$.

For computational convenience we define another Schwartz function $\hat\Phi^{(n+1),\prime}_2$ by the same definition as $\hat\Phi^{(n+1)}_2$ above except that we only require entries in $A$, $B$ and $D$ to be in $\mathcal{O}_{F,v}$, and same as for $\hat\Phi^{(n+1)}_2$ for other blocks.

We define $\Phi^{(n+1)}_2$, $\Phi^{(n+1),\prime}$, $\Phi^{(n)}_2$, and $\Phi^{(n)}_{2,w_{j_1},w_{j_2}}$ to be the inverse Fourier transform of the Schwartz functions $\hat\Phi^{(n+1)}_2$, $\hat\Phi^{(n+1),\prime}$, $\hat\Phi^{(n)}_2$, and $\hat{\Phi}^{(n)}_{2,w_{j_1},w_{j_2}}$.
\end{definition}

For $\Phi$ a Schwartz function on $M_{a+2b+1,2(a+2b+1)}(F_v)$ defined by
$$\Phi(X,Y):=\Phi_1(X)\Phi_2(Y)$$
where $\Phi_1=\Phi^{(n+1)}_1$ and $\Phi_2=\Phi^{(n+1)}_2$,
and define a Godement section (terminology of Jacquet) by:
\begin{equation}\label{fff}
f^\Phi(g)=\tau_2(\det g)|\det g|_v^{-s+\frac{a+2b+1}{2}}\times\int_{\mathrm{GL}_{a+2b+1}(\mathbb{Q}_v)}\Phi((0,X)g)\tau_1^{-1}\tau_2(\det X)|\det X|_v^{-2s+a+2b+1}d^\times X.
\end{equation}
We can also compute its $\beta$-th Fourier coefficient as
$$f^\Phi_\beta=\mathrm{Vol}(\Gamma)\hat{\Phi}_2({}^t\!\beta).$$
This is \cite[Lemma 1.10]{Eischen}.
\begin{definition}
We define Siegel Eisenstein series $f^{(n+1)}_v$, $f^{(n+1),\prime}_v$, etc by (\ref{fff}) taking the $\Phi_1$ as above and the $\Phi^{(n+1)}_2$, $\Phi^{(n+1),\prime}_2$, etc as the $\Phi_2$. We define the Siegel section at $v$
\begin{equation}\label{(15)}f_{sieg,v}(g):=\sum_E \Phi_3(E)f^{(n+1)}_v(g\begin{pmatrix}1_{n+1}&\begin{matrix}E&0&0\\0&0&0\\0&0&0\end{matrix}\\&1_{n+1}\end{pmatrix})
\end{equation}
where $\Phi_3$ is the inverse Fourier transform of the characteristic function of $\mathcal{O}^\times_{F,v}$. Recall the upper right block matrix is with respect to $(1+s+r)\times (1+r+s)$.
This is the Siegel section we use to construct families of Klingen Eisenstein series.

We define an $\Upsilon\in \mathrm{U}(n+1,n+1)(F_v)$ such that $\Upsilon_w=S_w^{-1}$ (as in (\ref{Upsilon})).
\end{definition}
\noindent\underline{Caution}: Later on when we are computing pullback sections of $f^{(n+1)}$, $f^{(n+1),\prime}$ and $f_{sieg,v}$, we mean the pullback sections of right translations by $\Upsilon_v$ of them.

It is clear that the local Fourier coefficient of $f_{sieg,v,\beta}$ can be nonzero only when $\beta$ is non-degenerate. We reduce the computation of the pullback section of $f_{sieg,v}$ to that of $f^{(n),\prime}_v$ which is relatively easier, by the lemma below.

\begin{lemma}\label{Lemma 4.29}
We have
\begin{equation*}
f^{(n+1)}_v(g)=\sum_{A,B,D} f^{(n+1),\prime}_v(g\begin{pmatrix}1_{n+1}&\begin{matrix}0&D&0\\B&A&0\\0&0&0\end{matrix}\\&1_{n+1}\end{pmatrix})
\end{equation*}
where $A$, $B$ and $D$ run over matrices with entries in $\frac{1}{\varpi_v}\mathcal{O}_{F,v}$ modulo $\mathcal{O}_{F,v}$.
\end{lemma}

We consider
$$S^{-1}\alpha(g,1)=\begin{pmatrix}a_1&a_2&a_3&&b_1&b_2&&\\a_4&a_5&a_6&&b_3&b_4&&\\a_7&a_8&a_9&&b_5&b_6&&\\&&&1_b&&&&
\\c_1&c_2&c_3&1_b&d_1&d_2&&\\c_4&c_5&c_6&&d_3&d_4&&\\a_7&a_8&a_9&&b_5&b_6&1_a&\\a_1&a_2&a_3&&b_1&b_2&&1_b\end{pmatrix}$$
where the block matrix is with respect to $b+1+a+b+b+1+a+b$. An argument as in \cite[Page 196]{SU} implies in order for this matrix to be in the support of $f^{(n+1),\prime}_v$, we must have
$$g\begin{pmatrix}&&&1&\\&&&&1\\&&1&&\\1&&&&\\&1&&&\end{pmatrix}$$
is in
$$B_{b+1,a+b+1}(F_v)\begin{pmatrix}1&\\M_{a+b+1,b+1}(\mathcal{O}_{F,v})&1\end{pmatrix}\begin{pmatrix}
1&&&&\\&1&&&\\&&&&1\\&&1&&\\&&&1&\end{pmatrix}.$$
Thus
$$g\begin{pmatrix}&&1_b&&\\&&&1&\\&&&&1_a\\1_b&&&&\\&1&&&\end{pmatrix}$$
is in $B_{b+1,a+b+1}(F_v)\begin{pmatrix}1&\\M_{a+b+1,b+1}(\mathcal{O}_{F,v})&1\end{pmatrix}$. Moreover the pullback section is right invariant under
$$\begin{pmatrix}&&1_b&&\\&&&1&\\&&&&1_a\\1_b&&&&\\&1&&&\end{pmatrix}
\begin{pmatrix}1&\\M_{a+b+1,b+1}(\mathcal{O}_{F,v})&1\end{pmatrix}
\begin{pmatrix}&&1_b&&\\&&&1&\\&&&&1_a\\1_b&&&&\\&1&&&\end{pmatrix}^{-1}.$$

We have also
\begin{lemma}\label{567}
We define $\Gamma^{\mathrm{Kling}}_0(\varpi_v)\subset \mathrm{GL}_n(\mathcal{O}_{F,v})$ to be the set of matrices $\begin{pmatrix}A&B\\C&D\end{pmatrix}$ (with respect to $(r+1)+(s+1)$ such that $A$ is upper triangular modulo $\varpi_v$, $D$ is lower triangular modulo $\varpi_v$, and $C$ has entries divisible by $\varpi_v$. The pullback section of $f^{(n+1),\prime}_v$ is right invariant under the action of $\Gamma^{\mathrm{Kling}}_0(\varpi_v)$.
\end{lemma}
The proof is a straightforward checking.
\begin{corollary}\label{supportt}
Let $N'$ be the set of matrices $\begin{pmatrix}1_{r+1}&\\S&1_{s+1}\end{pmatrix}$ where $S$ has entries in $\frac{1}{\varpi_v}\mathcal{O}_{F,v}$ such that $S_{1,1}$ is in $\mathcal{O}_{F,v}$. If $F_{\varphi_v}(f^{(n+1)},g,z)\not=0$, then $g\in P(F_v)Q(F_v)N'w_n$, and for any $n'\in N'$,
$$F_{\varphi_v}(f^{(n+1)},g,z)=F_{\varphi_v}(f^{(n+1)},gn',z).$$
\end{corollary}
Thus we only need to compute the values of the pullback section at matrices of the form
$$\begin{pmatrix}g_1&\\&g_2\end{pmatrix}\begin{pmatrix}&&&1_b&\\&&&&1\\1_b&&&&\\&1&&&\\&&1_a&&\end{pmatrix}.$$
Combining Lemma \ref{567} we only need to consider the case when $g_1$ and $g_2$ are Weyl elements, say $w'_1$ and $w'_2$. We have
$$\begin{pmatrix}&&&1_b&\\&&&&1\\1_b&&&&\\&1&&&\\&&1_a&&\end{pmatrix}^{-1}\begin{pmatrix}w'_1&\\&w'_2
\end{pmatrix}\begin{pmatrix}&&&1_b&\\&&&&1\\1_b&&&&\\&1&&&\\&&1_a&&\end{pmatrix}=\begin{pmatrix}w_2'&\\&w'_1\end{pmatrix},$$
where $w'_2$ and $w'_1$ are Weyl elements in $\mathrm{GL}_{b+1+a}$ and $\mathrm{GL}_{b+1}$ respectively. We can write
$$\begin{pmatrix}&&&1_b&\\&&&&1\\1_b&&&&\\&1&&&\\&&1_a&&\end{pmatrix}\begin{pmatrix}w'_2&\\&w'_1\end{pmatrix}
=\begin{pmatrix}&&&1_b&\\&&&&1\\&&1_a&&\\1_b&&&&\\&1&&&\end{pmatrix}\begin{pmatrix}w''_1&\\&w''_2\end{pmatrix}$$
for Weyl elements $w''_1$ and $w''_2$ in $\mathrm{GL}_{b+1+a}$ and $\mathrm{GL}_{b+1}$ respectively. It is also clear that we only need to consider the case when $w''_1$ and $w''_2$ are of the forms $w_{j_1}$ and $w_{j_2}$ in Definition \ref{defSch}.

We let $B\subset \mathrm{GL}_n$ be the Borel subgroup consisting of matrices $\begin{pmatrix}A&B\\0&D\end{pmatrix}$ (block matrices with respect to $r+s$) where $A$ is lower triangular and $D$ is upper triangular. We also write $B_r$ and $B_s$ for the upper triangular Borel subgroup of $\mathrm{GL}_r$ and $\mathrm{GL}_s$ respectively. Let $B_{r,s}\subset \mathrm{GL}_n$ be the parabolic subgroup consisting of matrices $\begin{pmatrix}A&B\\0&D\end{pmatrix}$. We realize the $\pi_v$ as induced representation $$\mathrm{Ind}^{\mathrm{GL}_n(F_v)}_{B(F_v)}\chi_{\alpha_1}\otimes\cdots\otimes\chi_{\alpha_n}.$$
We can also realize $\pi_v$ as the induced representation
$$\mathrm{Ind}^{\mathrm{GL}_n(F_v)}_{B_{r,s}(F_v)}\pi^{\mathrm{up}}\otimes\pi^{\mathrm{low}}$$
where $$\pi^{\mathrm{up}}=\mathrm{Ind}^{\mathrm{GL}_r(F_v)}_{{}^t\!B_r(F_v)}\chi_{\alpha_1}\otimes\cdots\otimes\chi_{\alpha_r},$$
and
$$\pi^{\mathrm{low}}=\mathrm{Ind}^{\mathrm{GL}_s(F_v)}_{B_s(F_v)}\chi_{\alpha_{r+1}}\otimes\cdots\otimes\chi_{\alpha_{r+s}}.$$
\noindent\underline{Test Vectors}\\
We consider the model for the induced representation $\pi=\mathrm{Ind}^{\mathrm{GL}_n(F_v)}_{B(F_v)}\chi_{\alpha_1}\otimes\cdots\otimes\chi_{\alpha_n}$, where $\chi_{\alpha_i}$ is the unramified character of $F_v^\times$ with Satake parameter $\alpha_i$. Consider the $v$-stabilization $u$ whose $U_{v,i}$ eigenvalue is $\alpha_n\cdots\alpha_{n+1-i}$. Consider a vector $\tilde{u}$ in $\tilde{\pi}=\mathrm{Ind}^{\mathrm{GL}_n(F_v)}_{B(F_v)}\chi_{\alpha^{-1}_1}\otimes\cdots\otimes\chi_{\alpha^{-1}_n}$ which is the characteristic function of a set $K'\subset K_0(p)$. Then it is easy to see that the pairing $\langle u, \tilde{u}\rangle=\mathrm{Vol}(K')$, and $\tilde{u}$ pairs all other stabilizations of $\pi$ with $0$.

\begin{definition}\label{Definition 4.32}
Let $\Gamma_{0,n}(\varpi_v, \varpi^2_v)$ be the subgroup of $\mathrm{GL}_n(\mathcal{O}_{F,v})$ consisting of matrices which are congruent to a matrix in $B(\mathcal{O}_{F,v})$ modulo $\varpi_v$, and congruent to a matrix in $B_{r,s}(\mathcal{O}_{F,v})$ modulo $\varpi^2_v$. In practice we define $\tilde{\varphi}_v$ to correspond to the characteristic function of $\Gamma_{0,n}(\varpi_v,\varpi^2_v)$ in the above model of induced representation. We define $\varphi_v$ to correspond to the constant function $1$ in the above model of induced representation. This is a spherical vector.

We define the vector $\tilde{\varphi}^{\mathrm{up}}_v\in\pi^{\mathrm{up}}$ to be the characteristic function of ${}^t\!\Gamma_0(\varpi_v)$ in the above model of induced representation, and $\tilde{\varphi}^{\mathrm{low}}_v\in\pi^{\mathrm{low}}$ to be the characteristic function of $\Gamma_0(\varpi_v)$.
We also define $\varphi^{\mathrm{up}}_v\in \pi^{\mathrm{up}}$ and $\varphi^{\mathrm{low}}_v\in\pi^{\mathrm{low}}$ be the spherical vectors taking the constant function $1$ on $\mathrm{GL}_r(\mathcal{O}_{F,v})$ and $\mathrm{GL}_s(\mathcal{O}_{F,v})$ respectively.
\end{definition}
Let $X=(X_1, X_2, X_3)$ with respect to the partition ($n=b+a+b$). For $g\in\mathrm{GL}_n(F_v)$, let $Z_1=Xg$ and $Z_2=(X_3, X_2, X_1)$. Write $Z_1'=(X_1g, X_2g)$, $Z''_1=(X_3g)$, $Z''_1=X_3$ and $Z'_2=(X_2, X_1)$.
Let $\tilde{R}'=\begin{pmatrix}1_b&&&&&\\&1_a&&&&\\&&&&1_b&\\&&&1_a&&\\&&&&&1_b\\&&1_b&&&\end{pmatrix}$. We have
$$(0,0,0;X_1,X_2,X_3)S^{\prime,-1}\tilde{R}^{\prime,-1}=(X_3,X_2,X_1;X_1,X_2,X_3)\tilde{R}^{\prime,-1}=
(X_3,X_2,X_1;X_3,X_2,X_1).$$

Write $w''=\begin{pmatrix}&&&1_b&\\&&&&1\\&&1_a&&\\1_b&&&&\\&1&&&\end{pmatrix}$.
Then we have the zeta integral
\begin{align*}
&&&\langle F(f^{(n+1),\prime},w''\begin{pmatrix}w_{j_1}&\\&w_{j_2}\end{pmatrix},z), \tilde{\pi}(w'')\tilde{\varphi}_v\rangle&\\&=&&\mathrm{Vol}(\Gamma)^{-1}\int_{\mathrm{GL}_n(F_v)}\int_{\mathrm{GL}_n(F_v)}\chi_{2,v}(\det Z_1)\chi_{1,v}^{-1}(\det Z_2)&\\
&&&\times|\det(Z_1Z_2)|^{z+\frac{n}{2}}\Phi^{(n)}_1(Z'_1, Z''_2)\Phi^{(n)}_{2,w_{j_1},w_{j_2}}(Z''_1, Z'_2)\langle \pi(Z_1)\tilde{\varphi},\pi(Z_2)\varphi\rangle d^\times Z_1 d^\times Z_2.&
\end{align*}

We take integrals of $Z_1$ over the set of matrices
$$\begin{pmatrix}1&C_1\\&1\end{pmatrix}\begin{pmatrix}A_1&\\&D_1\end{pmatrix}\begin{pmatrix}1&B_1\\&1\end{pmatrix}$$
with measures given by $$|\det A^s_1\det D_1^{-r}| d C_1 d^\times A_1 d^\times D_1 d B_1.$$ We take integrals of $Z_2$ over matrices of the form
$$\begin{pmatrix}1&B_2\\&1\end{pmatrix}\begin{pmatrix}A_2&\\&D_2\end{pmatrix}\begin{pmatrix}1&\\C_2&1\end{pmatrix}$$
with measures given by $$|\det A^{-s}_2\det D_2^{r}|d C_2 d^\times A_2 d^\times D_2 dB_2.$$ Then we have
$$\Phi^{(n)}_1(Z'_1, Z''_2)=\Phi^{(n)}_1(\begin{pmatrix}A_1&B_2D_2\\C_1A_1&D_2\end{pmatrix});$$
$$\Phi^{(n)}_{2,w_{j_1},w_{j_2}}(Z''_1, Z'_2)=\Phi^{(n)}_{2,w_{j_1},w_{j_2}}(\begin{pmatrix}A_1B_1&A_2+B_2D_2C_2\\C_1A_1B_1+D_1&D_2 C_2\end{pmatrix}).$$
\begin{lemma}
The product $\Phi^{(n)}_1(Z'_1, Z''_2)\cdot\Phi^{(n)}_{2,w_{j_1},w_{j_2}}(Z''_1, Z'_2)$ can be nonzero only when the following conditions are met. The $A_1$ and $D_2$ are congruent to identity modulo $\varpi_v$. The entries of $C_1$ and $B_2$ are divisible by $\varpi_v$. The $C_2$ has entries in $\mathcal{O}_{F,v}$. The $B_1$, $A_2$ and $D_1$ have entries in $\frac{\mathcal{O}_{F,v}}{\varpi_v}$.
\end{lemma}
The proof is straightforward.

We write $\Phi^{(n),\mathrm{up}}_{2, w_{j_1}, w_{j_2}}$ for the restriction of $\Phi^{(n)}_{w_{j_1}, w_{j_2}}$ to the upper right $r\times r$ block, and $\Phi^{(n),\mathrm{low}}_{2, w_{j_1}, w_{j_2}}$ for its restriction to the lower left $s\times s$ block.
\begin{proposition}
We have the factorization of the zeta integral
$$\Phi^{(n)}_{1,v}(Z'_1, Z''_2)\Phi^{(2)}_{2,w_{j_1},w_{j_2}}(Z''_1, Z'_2)\langle \tilde\pi(Z_1)\tilde{\varphi}_v,\pi(Z_2w')\varphi_v\rangle=\mathrm{Vol}(\Gamma_{r,s})J_1J_2$$
where
$$J_1=\Phi^{(n),\mathrm{low}}_{2,w_{j_1},w_{j_2}}(D_1)|\det D_1^r|^{\frac{1}{2}}\langle\tilde\pi^{\mathrm{low}}(D_1)\tilde{\varphi}^{\mathrm{low}}_v, \varphi^{\mathrm{low}}_v\rangle d^\times D_1,$$
and
$$J_2=\Phi^{(n),\mathrm{up}}_{2,w_{j_1},w_{j_2}}(A_2)|\det A_2^s|^{\frac{1}{2}}\langle\tilde\pi^{\mathrm{up}}(A^{-1}_2)\tilde{\varphi}^{\mathrm{up}}_v, \varphi^{\mathrm{up}}_v\rangle d^\times A_2.$$
\end{proposition}

\begin{proof}
We observe that from the definition $\tilde{\varphi}_v$ is invariant under $\begin{pmatrix}1&B_1\\&1\end{pmatrix}$ for $B_1$ with entries in $\frac{\mathcal{O}_{F,v}}{\varpi_v}$, and $\varphi_v$ is invariant under $\begin{pmatrix}1&\\C_2&1\end{pmatrix}$ where $C_2$ has entries in $\mathcal{O}_{F,v}$. So
$$\langle \tilde\pi(Z_1)\tilde{\varphi}_v, \pi(Z_2)\varphi_v\rangle=\langle\tilde\pi(\begin{pmatrix}1&\\C_1&D_1\end{pmatrix})\tilde{\varphi}_v, \pi\begin{pmatrix}A_2&B_2\\&1\end{pmatrix}\varphi_v\rangle.$$
But $\begin{pmatrix}1&-B_2\\&1\end{pmatrix}\begin{pmatrix}1&\\ C_1&D_1\end{pmatrix}$ can be written as
$$\begin{pmatrix}1&\\C&1\end{pmatrix}\begin{pmatrix}A&\\&D\end{pmatrix}\begin{pmatrix}1&B\\&1\end{pmatrix}$$
with $A$ congruent to identity modulo $\varpi^2_v$; $C=C_1A^{-1}$, $D\in (1+\varpi^2_v M(\mathcal{O}_{F,v}))D_1$, $B\in M_{b\times b}(\mathcal{O}_{F,v})$. So the above expression is
\begin{align*}
&\langle\tilde\pi(\begin{pmatrix}1&\\C_1&1\end{pmatrix}\begin{pmatrix}1&\\&D_1\end{pmatrix})\tilde\varphi_v,
\pi(\begin{pmatrix}A_2&\\&1\end{pmatrix})\varphi_v\rangle
&&=\langle\tilde\pi(\begin{pmatrix}1&\\&D_1\end{pmatrix})\tilde\varphi_v,\pi(\begin{pmatrix}A_2&\\&1\end{pmatrix}
\begin{pmatrix}1&\\-CA_2&1\end{pmatrix})\varphi_v\rangle&\\
&=\langle\tilde\pi(\begin{pmatrix}1&\\&D_1\end{pmatrix})\tilde\varphi_v,\pi(\begin{pmatrix}A_2&\\&1\end{pmatrix})\varphi_v\rangle&
&=\langle\tilde\pi(\begin{pmatrix}A^{-1}_2&\\&D_1\end{pmatrix})\tilde\varphi_v, \varphi_v\rangle.&
\end{align*}
Then from the definition, the $\tilde\pi(\begin{pmatrix}A^{-1}_2&\\&D_1\end{pmatrix})\tilde\varphi_v$ is supported in $$B_{r,s}(F_v)(\begin{pmatrix}A^{-1}_2&\\&D_1\end{pmatrix}\begin{pmatrix}1_r&\\ \varpi^2_v M(\mathcal{O}_{F,v})&1_s\end{pmatrix}\begin{pmatrix}A^{-1}_2&\\&D_1\end{pmatrix}^{-1})$$
and is invariant under the action of $\begin{pmatrix}A^{-1}_2&\\&D_1\end{pmatrix}\begin{pmatrix}1_r&\\ \varpi^2_v M(\mathcal{O}_{F,v})&1_s\end{pmatrix}\begin{pmatrix}A^{-1}_2&\\&D_1\end{pmatrix}^{-1}$. The latter matrix is contained in $\begin{pmatrix}1_r&\\M(\mathcal{O}_{F,v})&1_s\end{pmatrix}$, under whose action the $\varphi_v$ is invariant. So $\langle\tilde\pi(\begin{pmatrix}A^{-1}_2&\\&D_1\end{pmatrix})\tilde\varphi_v, \varphi_v\rangle$ can be factorized as
$$\mathrm{Vol}(\Gamma_{r,s})|\det A_2^s\det D^r_1|^{\frac{1}{2}}\langle\tilde\pi^{\mathrm{up}}(A^{-1}_2)\tilde\varphi^{\mathrm{up}}_v,\varphi^{\mathrm{up}}_v\rangle
\cdot\langle\tilde\pi^{\mathrm{up}}(D_1)\tilde\varphi^{\mathrm{low}}_v,\varphi^{\mathrm{low}}_v\rangle.$$

Now we see the zeta integral can be factorized as in the proposition.
\end{proof}

To get a description of the pullback Klingen Eisenstein section, we just need to evaluate at the $(w_{j_1}, w_{j_2})$'s and pair with the $\tilde{\varphi}_v$ which run over all Iwahori invariant test vectors corresponding to the $n!$ stabilizations, which we denote as $\varphi_{st_i}$'s.

It is easy to see that we are reduced to computing the integrals
$$I_2=\int_{\mathrm{GL}_{a+b}(F_v)}\chi^{-1}_{1,v}(\det A_2)\Phi^{(n),\mathrm{up}}_{2,w_{j_1},w_{j_2}}(A_2)|\det A_2|^{z+\frac{r}{2}}\langle\tilde\pi^{\mathrm{up}}(A^{-1}_2)\tilde\varphi^{\mathrm{up}}_v, \varphi^{\mathrm{up}}_v\rangle d^\times A_2$$
and
$$I_1=\int_{\mathrm{GL}_b(F_v)}\chi_{2,v}(\det D_1)\Phi^{(n),\mathrm{low}}_{2,w_{j_1},w_{j_2}}(D_1)|\det D_1|^{z+\frac{s}{2}}\langle\tilde\pi^{\mathrm{low}}(D_1)\tilde\varphi^{\mathrm{low}}_v, \varphi^{\mathrm{low}}_v\rangle d^\times D_1.$$
In the following we consider a Hida family $\mathbf{f}$ with coefficient a normal domain $\mathbb{I}$ whose specialization at an arithmetic point $\phi_0$ is an ordinary form $f\in\pi$.

We record the following easy lemma.
\begin{lemma}\label{Lemma 4.35}
Suppose $\mathbb{I}$ contains all Satake parameters $\alpha_{\mathbf{f},1}, \cdots, \alpha_{\mathbf{f},n}$ of $\mathbf{f}$ (this can be ensured by taking a finite extension of $\mathbb{I}$). Take an ordering $\sigma$ of $\alpha_{\mathbf{f},1}, \cdots, \alpha_{\mathbf{f},n}$. Consider the induced representation $\pi_\phi=\mathrm{Ind}^{\mathrm{GL}_n(F_v)}_{B(F_v)}\chi_{\alpha_{\mathbf{f},1,\phi}}\otimes\cdots\chi_{\alpha_{\mathbf{f},n,\phi}}$
Then there is an $\mathrm{Frac}(\mathbb{I})$-valued (whose denominators are nonzero at $\phi_0$) function $u_\sigma$ on the Weyl group $W_n\subset \mathrm{GL}_n(F_v)$, such that for any $\phi$ outside a closed subspace of $\mathrm{Spec}\ \mathbb{I}$ of lower dimension, the specialization of $u_\sigma$ to $\phi$ is the stabilization in $\pi_\phi$ corresponding to $\sigma$ (i.e. the eigenvalues under $U_{v,i}$ are given by $\alpha_{\mathbf{f},\phi,\sigma(n)}\cdots\alpha_{\mathbf{f},\phi,\sigma(n+1-i)}$).
\end{lemma}
The lemma follows by applying appropriate polynomials of the $U_{v,i}$ operators to the spherical vector.

\noindent\underline{Quantitative Results}\\
We first study the zeta integral for $f^{(n+1)}_v$ (instead of $f^{(n+1),\prime}_v$) at $(w_0, w_0)$. This is relatively easier.
By the Godement-Jacquet functional equation as in \cite[Theorem 4.3.9]{EHLS}, the second integral is

$$I_2=\mathrm{Vol}(\Gamma_{b,0}(p))\frac{L(z+\frac{1}{2},\tilde{\pi}^{\mathrm{up}}\otimes\chi^{-1}_{1,v})q^z_v}{L(-z+\frac{1}{2},\pi^{\mathrm{up}}
\otimes\chi_{1,v})}\int_{\mathrm{GL}_r(F_v)}\hat{\Phi}^{(n),\mathrm{up}}_{w_0,w_0}(A_2)|A_2|^{-z+\frac{r}{2}}
\chi_{1,v}(\det A_2)\langle\tilde\pi^{\mathrm{up}}(A^{-1}_2)\tilde\varphi^{\mathrm{up}}_v,\varphi^{\mathrm{up}}_v\rangle,$$

which equals
\begin{equation}\label{**}
I_2=\mathrm{Vol}(\Gamma_{b,0}(p))\frac{L(z+\frac{1}{2},\tilde{\pi}^{\mathrm{up}}\otimes\chi^{-1}_{1,v})q^z_v}{L(-z+\frac{1}{2},\pi^{\mathrm{up}}
\otimes\chi_{1,v})}\langle\tilde\varphi^{\mathrm{up}}_{v},\varphi^{\mathrm{up}}_v\rangle.
\end{equation}

We similarly get the formula for $I_1$.

\begin{proposition}\label{notzero}
Let $z$ be an integer. Then $F_{\varphi^{\mathrm{sph}}_v}(f_{\mathrm{sieg},v},-,z)\not=0$.
\end{proposition}
\begin{proof}
Note that by (\ref{(15)}) and Corollary \ref{supportt}, it is enough to see $F_{\varphi_v^{\mathrm{sph}}}(f^{(n+1)}_v,-,z)\not=0$, which is clear from the above computation on $I_1$ and $I_2$.
\end{proof}

Now we turn to values at other $(w_{j_1}, w_{j_2})$'s. These are more complicated, and we content ourselves with showing the description of the pullback section in the following proposition, which is enough for proving part (iii) of Theorem \ref{Construction}.\\

\noindent\underline{Qualitative Results}\\
We prove the following proposition. For $1\leq i_1\leq n!$ we write $\varphi_{v,i_1}$ for the stabilization as in Definition \ref{Definition 4.32} corresponds to the $i_1$-th ordering the the Satake parameters of $\pi_v$.
\begin{proposition}
For any $i_1$, $w_{j_1}$ and $w_{j_2}$, there exists elements $G_{j_1,j_2,i_1}\in\mathrm{Frac}\ (\mathbb{I}[[\Gamma_\mathcal{K}]])$ which is non-vanishing at the arithmetic point $\phi_0$ which corresponds to the ordinary form $f\in\pi$, such that for a Zariski dense set of arithmetic points $\phi$, we have
$$F_{\varphi^{\mathrm{sph}}}(f^{(n+1),\prime}_{\mathcal{D}_\phi}, \begin{pmatrix}w_{j_1}&\\&w_{j_2}\end{pmatrix},z)=\sum_{i_1}\phi(G_{j_1,j_2,i_1})\varphi_{v,i_1}.$$
\end{proposition}
\begin{remark}
We can also get such descriptions for the pullback section of $f_{sieg,v}$ using (\ref{(15)}) and Lemma \ref{Lemma 4.29}.
\end{remark}
\begin{proof}
To save notations we compute the $I_2$ in the case when $r=4$ and $j_1=2$. The general case is similar. It is equivalent to computing the pullback integrals for the Siegel-Weil section associated to $\Phi^{(n)}_1$ and $\Phi^{(n)}_{2,w_{j_1}, w_{j_2}}$. As in the quantitative results we use the Godement-Jacquet functional equation to evaluate it. We consider the Weyl element $w''=\begin{pmatrix}1&&&\\&&1&\\&&&1\\&1&&\end{pmatrix}$. For notational convenience we define $\hat{\Phi}^{(n),\mathrm{up},\prime}_{2, w_{j_1},w_{j_2}}$ to be the $\hat{\Phi}^{(n),\mathrm{up}}_{2,w_{j_1}, w_{j_2}}$ composed with this conjugation $g\mapsto g^{w''}=(w'')^{-1}g(w'')$.  Thus it is the characteristic function of the set of matrices in $\begin{pmatrix}\mathcal{O}^\times_{F,v}&\varpi_v\mathcal{O}_{F,v}&\varpi_v\mathcal{O}_{F,v}&\varpi_v\mathcal{O}_{F,v}\\ \mathcal{O}_{F,v}
&\mathcal{O}^\times_{F,v}&\varpi_v\mathcal{O}_{F,v}&\mathcal{O}_{F,v}\\ \mathcal{O}_{F,v}&\mathcal{O}_{F,v}&\mathcal{O}^\times_{F,v}&\mathcal{O}_{F,v}\\ \mathcal{O}_{F,v}&\mathcal{O}_{F,v}&\mathcal{O}_{F,v}&\mathcal{O}_{F,v}\end{pmatrix}$.
Write an element of it (block matrices with respect to $(r-1)+1$)
$$\begin{pmatrix}A&B\\C&D\end{pmatrix}=\begin{pmatrix}A&0\\C&D-CA^{-1}B\end{pmatrix}\begin{pmatrix}1&A^{-1}B\\&1
\end{pmatrix}=\begin{pmatrix}1&\\CA^{-1}&1\end{pmatrix}\begin{pmatrix}A&\\&D-CA^{-1}B
\end{pmatrix}\begin{pmatrix}1&A^{-1}B\\&1\end{pmatrix}.$$
Let $B'=A^{-1}B$, $C'=CA^{-1}$, $D'=D-CA^{-1}B$, then
$$\begin{pmatrix}A&B\\C&D\end{pmatrix}=\begin{pmatrix}1&\\C'&1\end{pmatrix}\begin{pmatrix}A&\\&D'\end{pmatrix}
\begin{pmatrix}1&B'\\&1\end{pmatrix}.$$
Here $A$ runs over ${}^t\!\Gamma_0(\varpi_v)$, $B'$ runs over matrices of the form ${}^t\!(\varpi_v,\mathcal{O}_{F,v},\mathcal{O}_{F,v})$, $C'$ runs over matrices whose entries are in $\mathcal{O}_{F,v}$, $D'$ runs over $\mathcal{O}_{F,v}$. We decompose the integrals according to valuation $t$ of $D'$ at $v$. More precisely for a fixed $D'$ with $\mathrm{ord}_vD'=t$, we decompose the above set as
$$\cup_{C'\in \mathcal{O}_{F,v}/\varpi^t_v\mathcal{O}_{F,v}}\begin{pmatrix}1&\\C'&1\end{pmatrix}\begin{pmatrix}1&\\&D'\end{pmatrix}\begin{pmatrix}A&\\C''&1
\end{pmatrix}\begin{pmatrix}1&B'\\&1\end{pmatrix}
$$
where $B'$ is as above and $C''$ runs over $\mathcal{O}_{F,v}$.

We can easily see from Lemma \ref{Lemma 4.35} that we can write
$$\int_A\int_{C''\in M(\mathcal{O}_{F,v})}\int_{B'\in {}^t\!(\varpi_v,\mathcal{O}_{F,v},\mathcal{O}_{F,v})}\tilde\pi^{\mathrm{up}}(\begin{pmatrix}A&\\C''&1\end{pmatrix}^{w''}
\begin{pmatrix}1&B'\\&1\end{pmatrix}^{w''})^{-1}
\tilde\varphi^{\mathrm{up}}_{v,i_1}$$ as
$$\sum^{r!}_{i_2=1} F^{\mathrm{up}}_{j_1,\mathrm{st}_{i_1},\mathrm{st}_{i_2}}\tilde\varphi^{\mathrm{up}}_{\mathrm{st}_{i_2}}$$
where $\tilde\varphi_{\mathrm{st}_{i_2}}$ runs over Iwahori invariant stabilizations of $\tilde{\pi}^{\mathrm{up}}$ with respect to the Borel subgroup ${}^t\!B_r^{w''}$, and $F^{\mathrm{up}}_{j_1,\mathrm{st}_{i_1},\mathrm{st}_{i_2}}$'s are elements in $\mathrm{Frac}\ \mathbb{I}$ whose denominators are non-vanishing at $\phi_0$.

Now we compute the
$$\sum_{C'\in\mathcal{O}_{F,v}/\varpi^t_v\mathcal{O}_{F,v}}\pi(\begin{pmatrix}1&\\C'&1\end{pmatrix})^{w''}
\int_{\varpi^t_v||D'}\int_{C''}\int_{B'}\tilde\pi^{\mathrm{up}}
(\begin{pmatrix}1&\\&D'\end{pmatrix}^{w''}\begin{pmatrix}A&\\C''&1\end{pmatrix}^{w''}\begin{pmatrix}1&B'\\&1\end{pmatrix}^{w''}
\tilde\varphi^{\mathrm{up}}_{v,i_1},$$
and consider the summation over $t$ in the expression for $I_2$, we get

\begin{equation}\label{qualitative}\frac{1}{1-\chi_{1,v}(\varpi_v)\alpha^{-1}_{\mathrm{st}_{i_2}}q_v^{z-\frac{1}{2}}}\cdot\sum F^{\mathrm{up}}_{j_1,\mathrm{st}_{i_1},\mathrm{st}_{i_2}}\tilde\varphi^{\mathrm{up}}_{\mathrm{st}_{i_2}}.
\end{equation}
where $\alpha_{\mathrm{st}_{i_2}}$ is the $U_{v,1}$-eigenvalue of $\tilde\varphi_{\mathrm{st}_{i_2}}$. Pairing with the test vector $\varphi^{\mathrm{up}}_v$, we get the desired property.
\end{proof}

\section{Differential Operators}\label{Diff}
In this section we fix one Archimedean place $v$ and study differential operators at this place. Throughout this section we omit the subscript $v$ for simplicity.
\begin{lemma}
Let $m\geq n$ be two positive integers. Let $S_{r,m,n}$ be the natural algebraic representation of $\mathrm{GL}_m\times\mathrm{GL}_n$ on the space of homogeneous degree $r$ polynomials with variables being the entries of $M_{m\times n}$ ($m$ by $n$ matrices). Then this representation decomposes as direct sums of
$$V_{(a_1,\cdots,a_n,0,\cdots,0)}\boxtimes V_{(a_1,\cdots, a_n)}$$
running over all sequences $a_1\geq \cdots\geq a_n\geq 0$ satisfying $\sum_i a_i=r$. Moreover each terms appears with multiplicity one.
\end{lemma}
\begin{proof}
This is \cite[Theorem 12.7]{Shi00}.
\end{proof}
\begin{lemma}
Let $a_1\geq \cdots \geq a_n$ be a sequence of integers and $V_{(a_1,\cdots,a_n)}$ be the algebraic representation of $\mathrm{GL}_n$ with highest weight $(a_1,\cdots,a_n)$. Then the representation $V_{(a_1,\cdots,a_n)}\otimes V_{(k,0, \cdots,0)}$ can be decomposed as the direct sum of representations with highest weight $V_{a_1+c_1,\cdots, a_j+c_j,\cdots, a_n+c_n})$ where $c_j$ runs over $n$-tuples of non-negative integers whose sum is $k$, and such that for each $1\leq j \leq n$ such that $a_{j}+c_j\geq a_{j+1}+c_{j+1}$.
\end{lemma}
\begin{proof}
This is a restatement of \cite[Proposition 15.25 (i)]{Fulton}.
\end{proof}
The following corollary is immediate from the above lemma.
\begin{corollary}\label{corollary 5.3}
Suppose $a_{n-1}\geq k$. Then the representation $V_{(a_1,\cdots, a_{n-1}, k)}$ appears in $V_{(a_1,\cdots,a_{n-1},0)}\otimes V_{(k,\cdots, 0)}$ and $V_{(a_1,\cdots,a_{n-1},0)}\otimes V^{\otimes k}_{(1,\cdots,0)}$ both with multiplicity one. Moreover for any tuple $(b_1,\cdots,b_{n-1},0)$ with $b_1\geq \cdots \geq b_{n-1}\geq 0$, the $V_{(a_1,\cdots, a_{n-1}, k)}$ does not appear in $V_{(b_1,\cdots,b_{n-1},0)}\otimes V_{(k,\cdots, 0)}$ and $V_{(b_1,\cdots,b_{n-1},0)}\otimes V^{\otimes k}_{(1,\cdots,0)}$ if $(a_1,\cdots, a_{n-1})$ is not $(b_1,\cdots, b_{n-1})$.
\end{corollary}
\noindent\underline{Klingen Eisenstein series}\\
For a non-negative integer $j$ such that $\frac{r+s+2+j}{2}\leq a_r$ and $\frac{r+s+2+j}{2}\leq b_1$. We define
$$\kappa=r+s-j,\ \ \underline{\kappa}=(\frac{r+s-j}{2},\cdots,\frac{r+s-j}{2};\frac{r+s-j}{2},\cdots,\frac{r+s-j}{2}).$$

Write $a_i'=a_i-\frac{r+s+2+j}{2}$ and $b_j'=b_j-\frac{r+s+2+j}{2}$. These are integers as explained in the Introduction.
We define
$$\underline{k}^{(r+1,s+1)}=(a'_1,\cdots, a'_r, 0; 0, b'_1,\cdots, b'_s)$$
$$\underline{k}^{(r+1,s+1),\prime}=(a'_1+1+j,\cdots, a'_r+1+j, 0; 0, 1+j+b'_1,\cdots, 1+j+b'_s)$$
and
$$\underline{k}^{(r+1,s+1),''}=(a'_1+1+j,\cdots, a'_r+1+j, 1+j; 1+j, 1+j+b'_1,\cdots, 1+j+b'_s).$$
We also define
$$\underline{k}^{(s,r)}=(b'_s+1+j,\cdots, b'_1+1+j;a'_r+1+j,\cdots, a'_1+1+j).$$

We divide the matrix into regions as follows.
$$\begin{pmatrix}\begin{matrix}I_b&&&\\&1&&\\&&I_a&\\&&&I_b\end{matrix}&
\begin{matrix}*&*&*&*\\
*&*&*&*\\
*&*&*&*\\
*&*&*&*\end{matrix}
\\0&\begin{matrix}I_b&&&\\&1&&\\&&I_a&
\\&&&I_b\end{matrix}
\end{pmatrix},$$
and we write the upper right matrix as $(\underline{X}):=\begin{pmatrix}\underline{X}_1&\underline{X}_2\\ \underline{X}_3&\underline{X}_4\end{pmatrix}$ with respect to the partition $((r+1)+s)\times((s+1)+r)$.

We define $\mathrm{Sym}^\bullet((\underline{X}_1)_{1+j},\underline{X}_2, \underline{X}_3)$ to be the set of polynomials involving only terms in $\underline{X}_2$, $\underline{X}_3$ and degree $(1+j)$ terms in $\underline{X}_1$. We similarly define $\mathrm{Sym}^\bullet(\underline{X}_2, \underline{X}_3)$.
We write $$f_{123}(\underline{X}):=\mathrm{Proj}_{((\underline{X}_1)_{1+j}, \underline{X}_2£¬\underline{X}_3)}\det(\underline{X})^{1+j}$$ for taking the terms expressing in $\det(\underline{X})^{1+j}$ involving only terms in $\underline{X}_2$, $\underline{X}_3$ and degree $1+j$ terms in $\underline{X}_1$ (thus not involving terms in $\underline{X}_4$).
Then from Corollary \ref{corollary 5.3} applied with $k=1+j$, we see the $V_{\underline{k}^{(r+1,s+1),''}\boxtimes \underline{k}^{(s,r)}}$ component of $\mathrm{Sym}^\bullet ((\underline{X}_1)_{1+j}, \underline{X}_2£¬\underline{X}_3)$ consists of elements spanned by
$f_{123}(\underline{X})\cdot f_j(\underline{X})$ where $f_j(\underline{X})$ runs over a basis of $V_{\underline{k}^{(r+1,s+1)}\boxtimes \underline{k}^{(s,r)}}$ in $\mathrm{Sym}^\bullet(\underline{X}_2, \underline{X}_3)$. Moreover if $f^{\mathrm{hw}}_{\underline{k}^{(\mathrm{r+1,s+1})}\boxtimes \underline{k}^{(s,r)}}(\underline{X})$ is the highest weight vector there, then
$f_{123}(\underline{X})\cdot f^{\mathrm{hw}}_{\underline{k}^{(\mathrm{r+1,s+1})}}(\underline{X})$ is the highest vector for $V_{\underline{k}^{(r+1,s+1),''}}$.

We write $e_\kappa$ for the standard basis of the one-dimensional representation $V_{\underline{\kappa}}$.
We choose the $f^{\mathrm{hw}}_{\underline{k}^{(\mathrm{r+1,s+1})}\boxtimes \underline{k}^{(s,r)}}(\underline{X})$ to be the polynomial
$$\det(\underline{X}^{1}_2)^{a'_1-a'_2}\det(\underline{X}^{2}_2)^{a'_2-a'_3}\cdots\det(\underline{X}^{r}_2)^{a'_r}\cdot
\det(\underline{X}^{1}_3)^{-b'_1+b'_2}\det(\underline{X}^{2}_3)^{-b'_2+b'_3}\cdots\det(\underline{X}^{r}_3)^{b'_s}, $$
where $\underline{X}^i_j$ are the $i$-th upper-left minors of $\underline{X}_j$.
Denote the $$f^{\mathrm{hw}}_{\underline{k}^{(r+1,s+1),''}\boxtimes \underline{k}^{(s,r)}}=f_{123}(\underline{X})\cdot f^{\mathrm{hw}}_{\underline{k}^{(\mathrm{r+1,s+1})}\boxtimes \underline{k}^{(s,r)}}(\underline{X}).$$

\begin{definition}\label{Definition 5.4}
We define the differential operator $\delta_{r+1,s+1}$ on the space of weight $V^\vee_{\underline{k}^{(r+1,s+1),\prime}+\underline{\kappa}}$ forms by
$$\delta_{r+1,s+1}f= \langle D^{1+j}f,\ f^{\mathrm{hw}}_{\underline{k}^{(r+1,s+1),''}\boxtimes \underline{k}^{(s,r)}}\otimes e_\kappa\rangle$$
where we use the simple notation $D$ to denote the $C^\infty$ or $p$-adic differential operators $\partial$ in Section \ref{Diffe}.
\end{definition}
\begin{proposition}\label{proposition 5.5}
For any $p$-adic automorphic form $f_\kappa$ of scalar weight $\kappa$,
we define
$$\delta_1f_\kappa:=\langle\mathrm{Proj}_{\underline{k}^{(r+1,s+1),''}+\underline{\kappa}}\circ D^{1+j}\circ \mathrm{Proj}_{V^\vee_{\underline{k}^{(r+1,s+1),\prime}}}(\underline{X}_2,\underline{X}_3)\circ D^\mathsf{d} f_\kappa,\ f^{\mathrm{hw}}_{\underline{k}^{(r+1,s+1),''}}\cdot e_\kappa\rangle,$$
where $\mathsf{d}=a'_1+\cdots+a'_r+b'_1+\cdots+b'_s+(1+j)(r+s)$ and
$$\delta_2f_\kappa:= \langle \mathrm{Proj}_{V^\vee_{\underline{k}^{(r+1,s+1),''}+\underline{\kappa}}}\circ D^{\mathsf{d}+1+j}f_\kappa,\ \det(\underline{X})^{1+j}\cdot f^{\mathrm{hw}}_{\underline{k}^{(r+1,s+1)}}e_\kappa\rangle.$$
Then we have the restriction of $\delta_1 f_\kappa-\delta_2f_\kappa$ to $\mathrm{U}(r+1,s+1)\times\mathrm{U}(s,r)$ is killed by the $e^{\mathrm{ord}}$ on $\mathrm{U}(s,r)$.
\end{proposition}

\begin{proof}
We first observe that for each term in the expression for $\det \underline{X}$, if there is no factor of this term in region $\underline{X}_4$, then there is at most degree one factor in region $\underline{X}_1$.

The restriction of the difference
$$\det(\underline{X})^{1+j}\cdot\mathrm{Proj}_{V_{\underline{k}}^{(r+1,s+1)}}\mathrm{Sym}^\bullet(\underline{X}_2, \underline{X}_3)e_\kappa-(\mathrm{Proj}_{((\underline{X}_1)_1,\underline{X}_2, \underline{X}_3)}\det(\underline{X})^{1+j})\cdot\mathrm{Proj}_{V_{\underline{k}}^{(r+1,s+1)}}\mathrm{Sym}^\bullet(\underline{X}_2, \underline{X}_3)e_\kappa$$
is an entry of a $p$-adic automorphic form in the image of Maass-Shimura differential operator on $X_{\mathrm{U}(s,r)}$ (because they involve factors in $\underline{X}_4$) which is killed by the $e^{\mathrm{ord}}$ on $\mathrm{U}(s,r)$. (Note that if we have the Gauss-Manin connections
$$\nabla_1:\mathcal{E}_1\rightarrow\mathcal{E}_1\otimes\Omega^1_{X_{\mathrm{U}(r,s)}},$$
$$\nabla_2:\mathcal{E}_1\rightarrow\mathcal{E}_1\otimes\Omega^1_{X_{\mathrm{U}(s,r)}}.$$
Then the Gauss-Manin connection on the product on $X_{\mathrm{U}(r,s)}\times X_{\mathrm{U}(s,r)}$ is given by $\nabla(v_1\otimes v_2)=\nabla_1v_1\otimes v_2+v_1\otimes\nabla_2(v_2)$.
\end{proof}
The above proof also gives the following corollary.
\begin{corollary}
For any automorphic representation $\pi$ of $\mathrm{U}(s,r)$ whose Archimedean components are holomorphic discrete series of weight $\underline{k}^{(s,r)}$, the $\pi$ component of the restriction of $(\delta^{C^\infty}_1-\delta^{C^\infty}_2)f_\kappa$ on $\mathrm{U}(s,r)$ is zero. (Here the superscript $C^\infty$ means taking entries of $C^\infty$ differential operators.)
\end{corollary}
This follows from that the difference is in the image of Maass-Shimura differential operator as in the above proof. Note that the holomorphic vector is the lowest weight in the corresponding holomorphic discrete series representation.\\

\noindent\underline{$p$-adic $L$-functions}\\
The case for $p$-adic $L$-functions is completely similar and easier than the Klingen Eisenstein series case. We define
$$\kappa=r+s-j,\ \ \underline{\kappa}=(\frac{r+s-j}{2},\cdots,\frac{r+s-j}{2};\frac{r+s-j}{2},\cdots,\frac{r+s-j}{2}).$$

Write $a_i'=a_i-\frac{r+s+j}{2}$ and $b_j'=b_j-\frac{r+s+j}{2}$.
We define
$$\underline{k}^{(r,s)}=(a'_1,\cdots, a'_r, ; b'_1,\cdots, b'_s)$$
$$\underline{k}^{(r,s),\prime}=\underline{k}^{(r,s),''}=(a'_1+j,\cdots, a'_r+j ; j+b'_1,\cdots, j+b'_s)$$

We also define
$$\underline{k}^{(s,r)}=(b'_s+j,\cdots, b'_1+j;a'_r+j,\cdots, a'_1+j).$$
We write $f'_{123}(\underline{X}):=\mathrm{Proj}_{( \underline{X}_2£¬\underline{X}_3)}\det(\underline{X})^{1+j}$ for taking the terms expressing in $\det(\underline{X})^{1+j}$ involving only terms in $\underline{X}_2$, $\underline{X}_3$ (thus not involving terms in $\underline{X}_4$ or $\underline{X}_1$. Note the difference here from the case of Klingen Eisenstein series). As before we define $f^{\mathrm{hw}}_{\underline{k}^{(r,s),''}\boxtimes \underline{k}^{(s,r)}}$ and also the differential operator
$$\delta_{r,s}f= \langle Df,\ f^{\mathrm{hw}}_{\underline{k}^{(r,s),''}\boxtimes \underline{k}^{(s,r)}}\otimes e_\kappa\rangle.$$
The following proposition is proved in the same way as Proposition \ref{proposition 5.5}.
\begin{proposition}\label{proposition 5.6}
For any $p$-adic automorphic form $f_\kappa$ of scalar weight $\kappa$,
we define
$$\delta'_1f_\kappa:=\langle \mathrm{Proj}_{V^\vee_{\underline{k}^{(r,s),\prime}}}(\underline{X}_2,\underline{X}_3)\circ D^\mathsf{d} f_\kappa,\ f^{\mathrm{hw}}_{\underline{k}^{(r,s),'}}\cdot e_\kappa\rangle,$$
where $\mathsf{d}=a'_1+\cdots+a'_r+b'_1+\cdots+b'_s+j(r+s)$ and
$$\delta'_2f_\kappa:= \langle \mathrm{Proj}_{V^\vee_{\underline{k}^{(r,s),''}+\underline{\kappa}}}\circ D^{\mathsf{d}}f_\kappa,\ \det(\underline{X})^j\cdot f^{\mathrm{hw}}_{\underline{k}^{(r,s)}}\cdot e_\kappa\rangle.$$
Then we have the restriction of $\delta'_1 f_\kappa-\delta'_2f_\kappa$ to $\mathrm{U}(r,s)\times\mathrm{U}(s,r)$ is killed by the $e^{\mathrm{ord}}$ on $\mathrm{U}(s,r)$.
\end{proposition}

\section{Global Computations and $p$-adic Interpolation}
\subsection{Hecke Projector}\label{projector}

\begin{lemma}\label{MIT}
We write $\mathscr{H}_v(K)$ over $\mathbb{C}$ (for $K$ an open compact subgroups of $\mathrm{U}(r,s)(\mathcal{O}_{F,v})$) for the abstract Hecke algebra of $\mathrm{U}(r,s)$ at $v$ defined by actions of double cosets $K\backslash\mathrm{U}(r,s)(F_v)/K$. Let $M_1,\cdots,M_n$ be the irreducible $\mathscr{H}_v(K)$-modules which are pairwise non-isomorphic. Then the image of $\mathscr{H}_\ell(K)$ in $\oplus_i \mathrm{End}_\mathbb{C}M_i$ is surjective.
\end{lemma}
\begin{proof} This is a standard fact of representation theory of finite dimensional algebras. For example, this can be deduced easily from  \cite[Theorem 7.6]{MIT}, noting that the dimension of the image is less than or equal to $\sum_i(\mathrm{dim}_\mathbb{C}M_i)^2$.
\end{proof}
\begin{lemma}
Suppose $\pi$ is a cuspidal automorphic representation of $\mathrm{U}(r,s)$ whose base change to $\mathrm{GL}(r+s)_{/E}$ is cuspidal. Suppose moreover that the Archimedean components of $\pi$ are cohomological with respect to an algebraic representations $V$ of $\mathrm{U}(r,s)$. Then the Archimedean components of $\pi$ are in the discrete series.
\end{lemma}
\begin{proof}
Since the base change of $\pi$ is cuspidal, it is well known that this base change is essentially tempered. Therefore the $\pi$ itself is in essentially tempered Arthur packet. But a cohomological and essentially tempered representations must be discrete series. We thus obtain the result.
\end{proof}
\begin{definition}\label{stability}
We write $\mathsf{m}_{r,s}$ for the cardinality of the Weyl group quotient $W_{\mathrm{U}(r,s)(\mathbb{R})}/W_{\mathrm{U}(r)(\mathbb{R})\times\mathrm{U}(s)(\mathbb{R})}$.
\end{definition}
Let $M$ be the space of ordinary cuspidal families on $\mathrm{U}(r,s)$ with some tame level group $K$, localized at the maximal ideal $\mathfrak{m}$ corresponding to the mod-$p$ Galois representation of $\pi$ (which is residually irreducible by our running assumption). As we have seen from Hida theory, this is free of finite rank over the weight algebra. For any regular algebraic cuspidal automorphic representation $\pi$ of $\mathrm{U}(r,s)$ whose residual Galois representation is irreducible, we know its base change to $\mathrm{GL}(r+s)_\mathcal{K}$ must be cuspidal. Thus it corresponds to a tempered and cohomological Arthur packet. The Archimedean Arthur packet of it consists of the set of $\mathsf{m}_{r,s}$ discrete series with the same infinitesimal character.  By \cite[Theorem 1.7.1]{KMSW}, for any cusp automorphic representation $\pi=\pi_\infty\otimes\pi_f$ appearing in this space of global sections of automorphic sheaves, localized at $\mathfrak{m}$, and for each $\pi'_\infty$ in the same tempered Arthur packet as $\pi_\infty$, the multiplicity for $\pi'_\infty\otimes\pi_f$ is exactly one. (These representations are stable in the sense that the $S_\psi$ in \emph{loc.cit.} is trivial, since the base change is cuspidal.) They only contribute to the middle degree cohomology, each with dimension one.

\begin{proposition}\label{classicity}
Let $g$ be a cuspidal ordinary $p$-adic automorphic eigenform whose residual Galois representation is absolutely irreducible. Suppose $g$ has cohomological weight and trivial nebentypus at $p$. Then $g$ is classical (i.e. holomorphic).
\end{proposition}
\begin{proof}
It is well known that an ordinary $p$-adic automorphic form has to be overconvergent.
If the weight is slightly regular in the sense of the main theorem of \cite{Pilloni}, then the result is a consequence of that theorem.

If we only assume the weight is cohomological, we use an argument of comparing dimension of the ordinary $p$-adic automorphic forms from global sections of coherent automorphic sheaves and from the cohomology of arithmetic groups. Let $\mathrm{dim}^{\mathrm{ord}}_{\mathrm{coh}}$ be the rank of the space of $\Lambda$-adic ordinary cuspidal $p$-adic automorphic forms that we defined using global sections of coherent automorphic sheaves, localized at the maximal ideal $\mathfrak{m}$. We also consider the action of the identity element in the Hecke algebra (regarded as an element of the Hecke algebra at prime to $p$ bad places), and then write $\mathrm{dim}^{\mathrm{ord}}_{\mathrm{arith}}$ as a function of $\phi$ for the character of it acting on the ordinary part of the alternating overconvergent arithmetic cohomology as in \cite{Urban} (see Introduction there). It is by definition a rigid analytic function and only takes integer values, and is thus a locally constant function. We first look at a point $\phi$ satisfying the Pilloni's regular assumption (so that we have the classicality result), then
$$\mathrm{dim}^{\mathrm{ord}}_{\mathrm{arith}}(\phi)=\mathsf{m}_{r,s}\cdot\mathrm{dim}^{\mathrm{ord}}_{\mathrm{coh}}
=\mathsf{m}_{r,s}\cdot\mathrm{dim}^{\mathrm{ord}}_{\mathrm{coh},\mathrm{cl},\phi}$$
where the subscript $\mathrm{cl}$ standards for the subspace of classical forms.
Now we look at $\phi_0$ which is also of cohomological weight, thus we do have classicality result for arithmetic group cohomology side \cite[Corollary 4.3.12]{Urban}. From the paragraph right before this proposition, we also have
$$\mathrm{dim}^{\mathrm{ord}}_{\mathrm{arith}}(\phi_0)=\mathsf{m}_{r,s}\cdot
\mathrm{dim}^{\mathrm{ord}}_{\mathrm{coh},\mathrm{cl},\phi_0}.$$
Taking $\phi$ in a neighbourhood of $\phi_0$, we have
$$\mathrm{dim}^{\mathrm{ord}}_{\mathrm{arith}}(\phi)=\mathrm{dim}^{\mathrm{ord}}_{\mathrm{arith}}(\phi_0).$$
These altogether implies the classicality at the weight $\phi_0$.
\end{proof}

We consider the $\mathbb{C}_p$-coefficient tensor product Hecke algebra of $\prod_v\mathscr{H}_v(K)$'s for all $v$ in $\Sigma\backslash\{p\}$. Then from Lemma \ref{MIT} and Proposition \ref{classicity}, we can find an element $t\in \prod_v\mathscr{H}_v(K)$ so that its action on $\phi_0(M)$ has distinct eigenvalues $\alpha_1$, $\alpha_2$,... $\alpha_{\mathsf{n}}$. Let $\mathbb{I}$ be the coefficient ring of the ordinary Hida families on $\mathrm{U}(r,s)$ which we suppose to be a Noetherian normal domain. Now we consider the action of $t$ on $M\otimes\mathrm{Frac}\ \mathbb{I}$ has distinct eigenvalues $\alpha_{1,M}$, $\alpha_{2,M}$, ... $\alpha_{\mathsf{n},M}$ whose denominators are non-vanishing at $\phi_0$, and their specializations at $\phi_0$ are just the $\alpha_1$, $\alpha_2$,... $\alpha_{\mathsf{n}}$. Then we define the projector
\begin{equation}
\mathrm{Proj}_\mathbf{f}:=\frac{(t-\alpha_{2,M})(t-\alpha_{3,M})\cdots(t-\alpha_{\mathsf{n},M})}
{(\alpha_{1,M}-\alpha_{2,M})(\alpha_{1,M}-\alpha_{3,M})\cdots(\alpha_{1,M}-\alpha_{\mathsf{n},M})}.
\end{equation}
Note that for a Zariski dense set of arithmetic points $\phi$, the eigenvalues for $t$ acting on $\phi(M)$ are pairwise distinct. At these points, the vectors in $\phi(M)$ contained in each automorphic representation must be spanned by eigenvectors for $t$.
\subsection{Interpolation}\label{Inter}

\begin{definition}
We define a family of Eisenstein datum as a quadruple $\mathbf{D}=(\mathbf{f}, \mathbb{I}, \tau_0, \Sigma)$ where $\mathbf{f}$ is a Hida family with normal coefficient ring $\mathbb{I}$ with $\mathbf{f}_{\phi_0}$ an ordinary vector in $\pi$; the $\tau_0$ is a Hecke character of $\mathcal{K}^\times\backslash\mathbb{A}^\times_\mathcal{K}$ of finite order, and $\Sigma$ is a set of primes of $F$ containing all bad primes.

We define the parameter space as $\mathbb{I}[[\Gamma_\mathcal{K}^+]]$. Note that we only include the cyclotomic direction for the Hecke character since the anti-cyclotomic twist direction is essentially absorbed into the weight space for $\mathbb{I}$. Let $\boldsymbol{\tau}=\tau_0\Psi_\mathcal{K}$.
We write $\boldsymbol{\tau}_\phi$ for the composition of $\boldsymbol{\tau}$ with $\phi$. Let $\phi\in\mathrm{Spec}\ \mathbb{I}[[\Gamma_\mathcal{K}]]$ be such that its restriction to $\mathbb{I}$ is arithmetic with weight $(a_{1,\phi}, \cdots, a_{r,\phi}; b_{1,\phi},\cdots, b_{s,\phi})$. Suppose $\boldsymbol{\tau}_\phi$ corresponds to the Hecke character with Archimedean type $(\frac{\kappa_\phi}{2},\frac{\kappa_\phi}{2})$. We say $\phi$ is arithmetic for constructing $p$-adic $L$-function if $\kappa_\phi\geq r+s$, $a_{r,\phi}\geq \frac{\kappa_\phi}{2}$ and $b_{1,\phi}\geq \frac{\kappa_\phi}{2}$; say it is arithmetic for constructing $p$-adic Eisenstein family if $\kappa_\phi\geq r+s+2$, $a_{r,\phi}\geq \frac{\kappa_\phi}{2}$ and $b_{1,\phi}\geq \frac{\kappa_\phi}{2}$. We define the Eisenstein datum $\mathbf{D}_\phi$ at $\phi$ in the sense of Definition \ref{Datum} to be $(\mathbf{f}_\phi, \boldsymbol{\tau}_\phi|\cdot|^{-\frac{\kappa_\phi}{2}}, \kappa_\phi, \Sigma)$

To study functional equations, we also define the ``dual Eisenstein datum'' as follows. Let $\phi\in\mathrm{Spec}\ \mathbb{I}[[\Gamma^+_\mathcal{K}]]$ be a non-arithmetic point such that $\tau_\phi$ is of infinity type $(r+s-j_\phi,r+s-j_\phi)$. (It is not an interpolation point for the Klingen Eisenstein family since they do not correspond to classical weights.) Then we define $\tilde{\mathbf{D}}^{(1)}_\phi$ by $(\pi_{\mathbf{f}_\phi}, \tilde{\tau}^c_\phi|\cdot|^{r+s-j_\phi},r+s+j_\phi,\Sigma)$, and $\tilde{\mathbf{D}}^{(2)}_\phi=(\pi_{\mathbf{f}_\phi}, \tilde{\tau}^c_\phi|\cdot|^{r+s-j_\phi},r+s+2+j_\phi,\Sigma)$. These are arithmetic points and are used for $p$-adic functional equations for $p$-adic $L$-functions and $p$-adic Klingen Eisenstein series respectively. Note that $L(\pi,\tau^c,z)=L(\tilde{\pi},\tau, z)$.

We define a \emph{distinguished} non-arithmetic point $\phi_0$ in which $\mathbf{f}_\phi$ specialize to an ordinary vector in $\pi$ and $\boldsymbol{\tau}_\phi$ is $\tau_0|\cdot|^{\frac{r+s}{2}}$.
\end{definition}

\begin{definition}
We write $\boldsymbol{\xi}_i$'s for the $\mathbb{I}[[\Gamma^+_\mathcal{K}]]$-valued characters interpolating the $\xi_i$'s in Definition \ref{generic} at points $\phi\in\mathrm{Spec}\mathbb{I}[[\Gamma^+_\mathcal{K}]]$ where the $\boldsymbol{\tau}_\phi|_{\mathcal{O}_{\mathcal{K},p}}$ and $\chi_i|_{\mathbb{Z}^\times_p}$'s there are in the generic case, and the $a_{1,\phi}=\cdots=a_{r,\phi}=b_{1,\phi}=\cdots=b_{s,\phi}=\kappa_\phi=0$. We omit the precise formula since it requires introducing unnecessary notations. Their specializations to general weight $(a_1,\cdots,a_r;b_1,\cdots,b_s)$ are related to local Fourier coefficient as in Lemma \ref{ArchiFourier} through the function defined below.
We define a function
\begin{equation}
\Phi_\xi(x)=\left\{\begin{array}{ll} 0&x\not \in \mathfrak{X}, \\ \boldsymbol{\xi}_1/\boldsymbol{\xi}_2(\det C_1)\cdots\boldsymbol{\xi}_{r-1}/\boldsymbol{\xi}_{r}(\det C_{r-1})\boldsymbol{\xi}_{r}(C_{r})&\\ \times\boldsymbol{\xi}_{a+b+2}/\boldsymbol{\xi}_{r+3}(\det B_1)\cdots\boldsymbol{\xi}_{r+s}/\boldsymbol{\xi}_{r+s+1}(\det B_{s-1})\boldsymbol{\xi}_{r+s+1}(\det B_s).& x\in \mathfrak{X}.\end{array}\right.
\end{equation}
where $B$ is a $(r+1)\times r$ or $r\times r$ matrix, $C$ is a $s\times (s+1)$ or $s \times s$ matrix. The $B_i$ and $C_i$ are the upper left $i\times i$ minors of $B$ and $C$, respectively.
\end{definition}

We first give the formula for the $\beta$-th Fourier coefficient of the Siegel Eisenstein series below. These are the Siegel Eisenstein series constructed in previous sections, normalized by the factors $B_\mathcal{D}$ and $B'_\mathcal{D}$ in \cite[Section 5.3.1]{WAN} respectively. We write
$$\beta=\begin{pmatrix}A_\beta& B_\beta\\ C_\beta & D_\beta\end{pmatrix}$$
with respect to the partition $((r+1)+s)\times((s+1)+r)$ or $(r+s)\times(s+r)$ depending on the size.
We write $\mathcal{A}_{\det}(\beta)$ for the element in $\mathbb{I}[[\Gamma^+_\mathcal{K}]]$ interpolating the $(\det\beta)^{\kappa_\phi}|\det\beta|^{\kappa_\phi}_p$ at $\phi$ with $\boldsymbol{\tau}_\phi$ having infinity type $(\kappa_\phi,\kappa_\phi)$.

Let
\begin{align}\label{(66)}
&\boldsymbol{f}_{sieg,\beta}&&=\mathcal{A}_{\det}(\beta)\prod_{v\not\in\Sigma\cup\{v_{\mathrm{aux}}\}}
h_{v,\beta}(\bar{\tau}'(\varpi_v)q_v^{-\kappa_\phi})&\notag\\
&&&\times\prod_{v\in\Sigma,v\nmid p}\mathrm{Vol}(S_{n+1}(\mathcal{O}_{F,v}))e_v(\mathrm{Tr}_{\mathcal{K}_v/F_v}(\frac{\beta_{a+b+2,1}+...+\beta_{a+2b+1,b}}{x_v}))
+\frac{\beta_{b+2,b+2}+...+\beta_{b+1+a,b+1+a}}{y_v\bar{y}_v})&\notag\\
&&&\times\prod_{v|p}\Phi_{\boldsymbol{\xi},v}(\beta)\times f_{sieg,v_{\mathrm{aux}},\beta}&
\end{align}
and

\begin{align}\label{6666}
&\boldsymbol{f}'_{sieg,\beta}&&=\mathcal{A}_{\det}(\beta)\prod_{v\not\in\Sigma\cup\{v_{\mathrm{aux}}\}}
h_{v,\beta}(\bar{\tau}'(\varpi_v)q_v^{-\kappa_\phi})&\notag\\
&&&\times\prod_{v\in\Sigma,v\nmid p}\mathrm{Vol}(S_{n+1}(\mathcal{O}_{F,v}))e_v(\mathrm{Tr}_{\mathcal{K}_v/F_v}(\frac{\beta_{a+b+1,1}+...+\beta_{a+2b,b}}{x_v}))
+\frac{\beta_{b+1,b+1}+...+\beta_{b+a,b+a}}{y_v\bar{y}_v})&\notag\\
&&&\times\prod_{v|p}\Phi_{\boldsymbol{\xi},v}(\beta)\times \hat{\Phi}^{(n)}_{2,w_0,w_0}(\beta)&
\end{align}
where $\hat{\Phi}^{(n)}_{2,w_0,w_0}(\beta)$ is defined in Definition \ref{defSch}.

\begin{proposition}\label{Siegel Measure}
There are $\Lambda_{r,s}[[\Gamma_\mathcal{K}]]$-adic formal Fourier expansions $\mathbf{E}_{\mathbf{D},sieg}$ and $\mathbf{E}_{\mathbf{D},sieg}'$ such that
$$\mathbf{E}_{\mathbf{D},sieg,\phi}=E_{sieg,\mathbf{D}_\phi}(\prod_v f_{sieg,v}, z_\kappa,-),$$
$$\mathbf{E}'_{\mathbf{D},sieg,\phi}=E'_{sieg,\mathbf{D}_\phi}(\prod_v f'_{sieg,v}, z'_\kappa,-)$$
in terms of formal Fourier expansions.
The formal $q$-expansion is given by (\ref{(66)}) and (\ref{6666}) above.
\end{proposition}
\begin{proof}
This is a formal application of Kummer congruences using our interpolation of the Fourier-expansion (\ref{(66)}), as detailed in \cite[Lemma 3.15]{Hsieh13}.
\end{proof}
We also define the Siegel section used for functional equations
$$f^{\mathrm{fteq}}=\prod_{v|\infty}f_{sieg,v}\prod_{v<\infty}f^{\mathrm{fteq}}_v,$$
$$f^{\mathrm{fteq},\prime}=\prod_{v|\infty}f'_{sieg,v}\prod_{v<\infty}f^{\mathrm{fteq},\prime}_v.$$

\begin{theorem}\label{Construction}
Let $\mathbf{f}$ be an $\mathbb{I}$-coefficient nearly ordinary cuspidal eigenform on $\mathrm{U}(r,s)$ such that the specialization $\mathbf{f}_\phi$ at a Zariski dense set of ``generic'' arithmetic points $\phi$ is classical and generates an irreducible automorphic representation of $\mathrm{U}(r,s)$. Let $\Sigma$ be a finite set of primes containing all primes dividing the any entry of $\zeta$ or the conductor of $\mathbf{f}$ or $\mathcal{K}$. Then
\begin{itemize}
\item[(i)] There is an element $\mathcal{L}_{\mathbf{f},\tau_0}^\Sigma\in\mathbb{I}[[\Gamma_\mathcal{K}]]\otimes_{\mathbb{I}}F_{\mathbb{I}}$ (the $F_{\mathbb{I}}$ is the fraction field of $\mathbb{I}$) whose denominators nonzero at $\phi_0$, such that for any generic arithmetic points $\phi\in\mathrm{Spec}\mathbb{I}[[\Gamma_\mathcal{K}]]$, we have
\begin{align*}
&\phi(\mathcal{L}_{\mathbf{f},\tau_0}^\Sigma)=&&c_\kappa'(z_{\kappa_\phi}')(\frac{(-2)^{-d(a+2b)}
(2\pi i)^{d(a+2b)\kappa_\phi}(2/\pi)^{d(a+2b)(a+2b-1)/2}}{\prod_{j=0}^{a+2b-1}(\kappa_\phi-j-1)\!^d})^{-1}\cdot C_{\mathbf{f}_\phi}^p
&\\
&&&\times \prod_{v|p}(\mathrm{Vol}_{\phi,v}\cdotp^{-ss_2(\frac{1+a+2b}{2})}p^{-\sum_{j=1}^rt_j(\frac{a+2b+1}{2})}
\times |p^{t_1+...+t_r+s\cdot s_2}|^{-\frac{\kappa_\phi}{2}}&\\
&&&\times \prod_{i=r+1}^{r+s}\mathfrak{g}(\chi_i^{-1}\tau_2)\chi_i\tau_2^{-1}(p^{s_2})\prod_{j=1}^r\mathfrak{g}
(\chi_j\tau_1^{-1})
\chi_j^{-1}\tau_1(p^{t_j}))\cdot \frac{L^\Sigma(\tilde{\pi}_{\mathbf{f}_\phi},\bar{\tau}_\phi^c,\frac{\kappa_\phi-r-s+1}{2})}
{\langle\tilde{\varphi}_\phi^{ord},\varphi_\phi\rangle}&
\end{align*}
where
$$\mathrm{\mathrm{Vol}_{\phi,v}}=(\frac{p^{(r+s)(r+s-1)/2}\cdot(p-1)^{r+s}}{(\prod_{i=1}^r
p^{t_i\cdot(r+s-i)})\cdot
(\prod_{i=1}^{s}p^{t_{r+i}(s-i)})\cdot
\prod_{j=1}^{r+s}(p^j-1)}$$
is nothing but the volume of the level group for $\varphi_\phi$ at $v$, the
$\chi_i$'s are defined in Definition \cite[Definition 4.42]{WAN}, $\tau_{\phi,v}=(\tau_1, \tau_2^{-1})$ such that $\tau_i$ has conductor $p^{s_i}$ with $s_2>s_1$. The
$$C_{\mathbf{f}_\phi}^p=\prod_{v\nmid p,v\in\Sigma}\tau(y_v\bar{y}_vx_v)|(y_v\bar y_v)^2x_v\bar{x}_v|_v^{-z_{\kappa_{\phi}'-\frac{a+2b}{2}}}\mathrm{Vol}(\mathfrak{Y}_v)$$
(the $x_v$ and $y_v$ are the $x$ and $y$ in Subsubsection \ref{Ramified Pullback} and $\mathfrak{Y}_v$ is defined in Definition \ref{rfd}.)
The $c'_{\underline{k},\kappa}$ is the nonzero constant defined in Lemma \ref{Archimedean constant} and $\kappa_\phi$ is the weight associated to the arithmetic point $\phi$. The $\varphi_\phi$ and $\tilde{\varphi}_\phi^{ord}$ are the specialization of $\mathbf{f}$ and the $\mathbf{f}^\vee$ provided by the assumption ``DUAL'' and $\mathrm{Proj}_{\mathbf{f}^\vee}$ in \cite[Section 5.2.3]{WAN} (we explain its validity in the proof). Note that when we identify $\mathrm{U}(r,s)$ with $\mathrm{U}(s,r)$ in the obvious way, the Borel subgroups with respect to which the $\mathbf{f}$ and $\mathbf{f}^\vee$ are ordinary are opposite to each other. The $\tau_\phi$ are specializations of the family of CM characters $\boldsymbol{\tau}$. The $p^{t_i}$'s are conductors of some characters defined in Definition \ref{generic}. Note that we have re-written the formulas in \emph{loc.cit} using that $a+b=r$ and $b=s$ there. We also correct some errors in the expression in \emph{loc.cit}.
\item[(ii)] There is a set of formal $q$-expansions $\mathbf{E}_{\mathbf{f},\tau_0}:=\{\sum_\beta a_{[g]}^h(\beta)q^\beta\}_{([g],h)}$ for $$\sum_\beta a_{[g]}^h(\beta)q^\beta\in(\mathbb{I}^{ur}[[\Gamma_\mathcal{K}]]\hat{\otimes}_{\mathbb{Z}_p}
    \mathcal{R}_{[g],\infty})\otimes_{\mathbb{I}^{ur}}F_{\mathbb{I}^{ur}},$$ whose denominators are nonzero at $\phi_0$, where $\mathcal{R}_{[g],\infty}$ is some ring to be defined later in equation (\ref{(5)}), $([g],h)$ are $p$-adic cusp labels, such that for a Zariski dense set of arithmetic points $\phi\in\mathrm{Spec} \mathbb{I}[[\Gamma_\mathcal{K}]]$, $\phi(\mathbf{E}_{\mathbf{f},\tau_0})$ is the Fourier-Jacobi expansion of the holomorphic nearly ordinary Klingen Eisenstein series $E(f_{Kling,\phi},z_{\kappa_\phi},-)$ we construct before. Here $f_{Kling}$ is the tensor product ``Klingen section'' of the local pullback sections $F_\varphi(f_{sieg,v},;z,-)$ in our local computations in Section \ref{Section Eisenstein Series}.
\item[(iii)] The terms $a_{[g]}^t(0)$ are divisible by $\mathcal{L}_{\mathbf{f},\tau_0}^\Sigma\cdot\mathcal{L}_{\bar{\tau}'_0\chi^a_\mathcal{K}}^\Sigma$ where $\mathcal{L}_{\bar{\tau}_0'\chi^a_\mathcal{K}}^\Sigma$ is the $p$-adic $L$-function of the character $\bar{\tau}'_0$ (note that we missed the character $\chi^a_\mathcal{K}$ in \emph{loc.cit.}).
\end{itemize}
\end{theorem}
\begin{proof}
This is essentially the main theorem of \cite{WAN} and the proof is given in \cite[Section 5.3]{WAN}. Note that in \cite[Section 5.2.3]{WAN} we made an assumption ``DUAL'', which says that for the given Hida family $\mathbf{f}$, we can find a dual Hida family $\tilde{\mathbf{f}}$ in the sense of \emph{loc.cit}. However here, this assumption can be deduced from the operator $\mathrm{Proj}_{\mathbf{f}}$ defined in Section \ref{projector}: we pullback the $\mathbf{E}'_{\mathbf{D},sieg}$ under
$$\mathrm{U}(r,s)\times\mathrm{U}(s,r)\hookrightarrow \mathrm{U}(r+s,r+s),$$
and apply the Hecke projector $\mathrm{Proj}_{\mathbf{f}}\circ e^{\mathrm{ord}}$ to the $\mathrm{U}(s,r)$-part. By our assumption on $\pi$ and $\tau$ (namely the Hodge-Tate weight of $M$ does not contain $0$ and $1$), the $L(\tilde{\pi},\bar{\tau}^c\cdot(\omega\circ\mathrm{Nm}), \frac{3}{2})$ is a critical value and is thus nonzero. The resulting family on $\mathrm{U}(r,s)$ (restricting to an appropriate subfamily parameterized by $\mathrm{Spec}\ \mathbb{I}$) is the desired $\tilde{\mathbf{f}}$.

Then one constructs the Klingen Eisenstein family by first pullback the $\mathbf{E}_{\mathbf{D},sieg}$ under
$$\mathrm{U}(r+1,s+1)\times\mathrm{U}(s,r)\hookrightarrow \mathrm{U}(r+s+1,r+s+1),$$
and apply the Hecke projector $\mathrm{Proj}_{\tilde{\mathbf{f}}}\circ e^{\mathrm{ord}}$ to the $\mathrm{U}(s,r)$-part.
\end{proof}

\subsection{Functional Equation and Non-vanishing}
\subsection{Functional Equation}\label{Feq}
Before continuing we need the following lemma.
\begin{lemma}
Let $q\in Q_n(\mathbb{A}_F)$ or $Q_{n+1}(\mathbb{A}_F)$ and $\det\beta=0$. Then the $\beta$-th Fourier coefficient for $E_{\mathrm{sieg},\beta}(f^{\mathrm{fteq}},z,q)$ is identically zero as a function of $z$.
\end{lemma}
\begin{proof}
Applying global functional equation for Siegel Eisenstein series, then the lemma follows from our computations of local Fourier coefficient at $v_{\mathrm{aux}}$.
\end{proof}

The following proposition is due to Kudla-Sweet \cite{KS}.
\begin{proposition}\label{KudlaS}
Let  $f_v\in I_n(\chi_v)$.
We have following equation
\begin{align*}
&M(f_v,z)_{-z,\beta}=&&f_{v,z,\beta}\cdot\chi_v(\det\beta)^{-1}|\det\beta|^{-z}_v\gamma(E_v/F_v,\psi_v)^{\frac{n(n-1)}{2}}
\chi_{E/F,v}(\det\beta)^{-1}&\\&&&\times\prod^n_{r=1}\epsilon(z-n+r,\chi_v\chi^r_{E/F,v},\psi_v)^{-1}\frac{L(1-z,(\chi_v\chi^r_{E/F,v})^{-1})}
{L(z,\chi_v\chi_{E/F,v})}.&
\end{align*}
\end{proposition}
This is just \cite[Proposition 3.1]{KS1}. The $\gamma$ is the Weil index as in \emph{loc.cit.}
\begin{corollary}
Let $\phi$ be a non-arithmetic point we defined before with the associated integer $j_\phi\geq 0$. For any finite prime $v$ and any $\beta$ with $\det\beta\not=0$, we have
$$\prod_{v|\infty}(\det\beta)^{1+j_\phi}_v\phi(\boldsymbol{f}_{sieg,\beta})=
f^{\mathrm{fteq}}_{\tilde{\mathbf{D}}^{(2)}_\phi,\beta}.$$
\end{corollary}
\begin{proof}
We prove it by combining the previous proposition with our computations of local Fourier coefficients for Siegel sections. Note that $(\prod_{v|\infty}|\det\beta|_v\prod_{v|p}|\det\beta|_v)$ is a $p$-adic unit, and the specialization of the factor $(\prod_{v|\infty}|\det\beta|_v\prod_{v|p}|\det\beta|_v)^{\kappa-n-1}$ appearing in the Archimedean and $p$-adic Fourier coefficient to $\phi$ is given by
$(\prod_{v|\infty}|\det\beta|_v\prod_{v|p}|\det\beta|_v)^{-1-j_\phi}$. The good primes and $\Sigma\backslash\{p\}$ contributions are computed similarly. Note also that the product of the local Weil indices is equal to $1$.
\end{proof}
\noindent\underline{Proof of Theorem \ref{Theorem 1.4} and \ref{Theorem 1.5}}
\begin{proof}
By the effects of differential operators on $q$-expansions (as in the proof of \cite[Proposition 5.3]{EW}, which uses \cite[Theorem 9.2 (4)]{Eischen}) we see that the multiplying by $(\det\beta)^{1+j_\phi}_v$ in the above Corollary is equivalent to applying differential operator $D^{(n+1)(1+j_\phi)}$ and pairing with the $\det(\underline{X})^{1+j_\phi}$ as in Section \ref{Diff}. Applying Definitions \ref{Definition 5.4} and Propositions \ref{proposition 5.5} and \ref{proposition 5.6} to the specialization of the family of scalar valued Siegel Eisenstein series (no differential operators applied) to the weight $\underline{\kappa}$ there (note that this is $p$-adic limit instead of a classical form) , the proof of Theorem \ref{Construction} also gives Theorem \ref{Theorem 1.5}. Theorem \ref{Theorem 1.4} follows similarly (we omit the details). Note that the $\delta_1$ and $\delta_2$ correspond to the left and right side of the equality of Theorem \ref{Theorem 1.5} respectively.
\end{proof}

\subsection{Non-vanishing}\label{Nv}
We first record a proposition which is a key ingredient to study the non-vanishing of the Klingen Eisenstein family at $\phi_0$.
\begin{proposition}\label{formula intertwining}
Suppose our data $(\pi_v,\tau_v)$ comes from the local component at $v$ of a global data. Then there are meromorphic functions $\gamma^{(1)}(\rho_v,z)$ and $\gamma^{(2)}(\rho_v,z)$ such that
$$F'_{\varphi^\vee}(-z,M(z,f_v),g)=\gamma^{(1)}(\rho_v,z)
F'_\varphi(f_v;z,-))_{-z}(g)$$
and
$$F_{\varphi^\vee}(-z,M(z,f_v),g)=\gamma^{(2)}(\rho_v,z)A(\rho_v,z,F_\varphi(f_v;z,-)_{-z}(g).$$
Moreover if $v$ is a good prime then
$$\gamma^{(1)}(\rho_v,z)=\frac{\prod^{n-1}_{i=0}L(2z+i-n+1,\bar{\tau}'_v\chi^i_{\mathcal{K},v})}
{\prod^{n-1}_{i=0}L(-2z+n-i,\tau'_v\chi^i_{\mathcal{K},v})}\frac{L(\pi_v,\tau^c_v,\frac{1}{2}-z)}
{L(\tilde{\pi}_v,\bar{\tau}^c_v,z+\frac{1}{2})},$$
and
\begin{equation}\label{exp}
\gamma^{(1)}(\rho_v, z)=\gamma^{(2)}(\rho_v, z+\frac{1}{2}).
\end{equation}
\end{proposition}
\begin{proof}
The first part is just \cite[Proposition 4.40]{WAN}, which is a formal generalization of \cite[Proposition 11.13]{SU}. The proof of the formula at good primes is just applying \cite[Lemma 11.7]{SU} and the pullback formula at these primes. Note that we apply \emph{loc.cit.} for $\beta=0$ to obtain the formula for $M(-,f^{sph})$ (the Siegel series $h_{v,\beta}$ there for $\beta=0$ is the constant function $1$ by \cite[Proposition 19.2]{Shimura97}.
\end{proof}
In a moment we use this proposition to compute the pullback formula for $f^{\mathrm{fteq}}_v$'s. We expect (\ref{exp}) to be true for bad primes as well, but are unable to prove this (this is not needed).

To prove the next proposition we need some preparations on Casselman-Shahidi's theory of intertwining operators.
\begin{lemma}\label{Irreducibility}
Suppose $\mathrm{U}(r,s)(F_v)$ is quasi-split and $\pi_v$ is tempered and generic. Then $I(\tilde{\rho}_v,\frac{1}{2})$ is reducible if and only if $L(\bar{\tau}'_v\chi^{a+2b}_{\mathcal{K},v},z)$ has a pole at $z=0$. Recall the notation $I(\tilde{\rho}_v,\frac{1}{2})$ is the induced representation in Section \ref{K E S} with the action at $z=\frac{1}{2}$.
\end{lemma}
\begin{proof}
This follows from \cite[Proposition 5.3]{CS}.
\end{proof}
\begin{lemma}
We can replace $\pi$ by a cuspidal automorphic representation (which we still denote as $\pi$) with the same Galois representation as that of $\pi$, which is the holomorphic discrete series at all Archimedean places, and is generic when $\mathrm{U}(r,s)(F_v)$ is quasi-split.
\end{lemma}
\begin{proof}
This follows from \cite[Theorem 1.7.1]{KMSW} and the tempered packet conjecture proved in \cite[Corollary 9.2.4]{Mok} (which says that any tempered Arthur packet for the quasi-split unitary group contains a generic element). Note that by our assumption that the base change of $\pi$ is cuspidal (from (Irred)), the $\pi$ has stable parameter in the sense that the $S_\psi$ in \emph{loc.cit} is trivial.
\end{proof}

\begin{lemma}\label{poles}
Suppose $\tilde{\pi}_v$ is generic.
Suppose $L(\tilde{\pi}_v,\bar{\tau}^c_v,z)$ has poles at $z=-\frac{1}{2}$. Then $A(\rho_v,z)$ has a pole at $z=-\frac{1}{2}$ with at least the same multiplicity.
\end{lemma}
\begin{proof}
Suppose first that $L(\bar{\tau}'_v\chi^{a+2b}_{\mathcal{K},v},z)$ does not have a pole at $z=0$. Note that the normalization factor (\ref{normalizationfactor}) at $z=-\frac{1}{2}$ has a zero, and is holomorphic at $z=\frac{1}{2}$. Moreover we have $A(\tilde{\rho}_v, z)$ is holomorphic at $z=\frac{1}{2}$. Since the normalized intertwining operator $\mathcal{N}$ (see \cite{Zhang}) with the normalization factor
\begin{equation}\label{normalizationfactor}
\frac{L(\tilde{\pi}_v,\bar{\tau}^c_v,z+1)L(\bar{\tau}'_v\chi^a_{\mathcal{K},v}, 2z+1)}
{L(\tilde{\pi}_v,\bar{\tau}^c_v, z)L(\bar{\tau}'_v\chi^a_{\mathcal{K},v},2z)}
\end{equation}

satisfies (\cite[Proposition 3.3.1]{Mok})
$$\mathcal{N}(\tilde{\rho}_v,\frac{1}{2})\circ\mathcal{N}(\rho_v, -\frac{1}{2})=\mathrm{id}.$$
We see that $A(\rho_v, z)$ must have a pole at $z=-\frac{1}{2}$ with at least the same multiplicity with that of $L(\tilde{\pi}_v,\bar{\tau}^c_v,z)$.

Then suppose that $L(\bar{\tau}'_v\chi^{a+2b}_{\mathcal{K},v},z)$ has a pole. Then by Lemma \ref{Irreducibility}, $I(\tilde{\rho}_v, \frac{1}{2})$ is reducible, and thus $A(\tilde{\rho}_v, \frac{1}{2})$ kills a non-trivial subrepresentation of it. Noting that the normalization factor is nonzero and holomorphic at $z=\frac{1}{2}$. We see again $A(\rho_v, -\frac{1}{2})$ must have a pole with at least the same multiplicity of $L(\tilde{\pi}_v,\bar{\tau}^c_v,z)$.
\end{proof}

\begin{lemma}\label{generates}
The $F(f_{x,y,v},-\frac{1}{2})$'s in Definition \ref{xyv} for different choices of $x$, $y$'s with $\mathrm{ord}_v(x)\gg 0$ and $\mathrm{ord}_v(y)\gg 0$ generate the $I(\rho_v,-\frac{1}{2})$ as $\mathrm{U}(r,s)(F_v)$-representation.
\end{lemma}
\begin{proof}
The proof is straightforward by noting that any Klingen section supported in the big open cell is generated by these sections $F(f_{x,y,v},-\frac{1}{2})$'s. Then we apply translations of these sections by Weyl elements and they do generated $I(\rho_v,-\frac{1}{2})$.
\end{proof}
To prove the non-vanishing result, we need to compute the pullback of the Siegel section $f^{\mathrm{fteq}}_{sieg}$.
\begin{proposition}\label{proposition 6.17}
We have the $F(f^{\mathrm{fteq}}_{\tilde{\mathbf{D}}^{(2)}_{\phi_0}},-,z)|_{z=\frac{1}{2}}$ is not the zero section.
\end{proposition}
\begin{proof}
For simplicity we omit the subscripts $\phi_0$ through this proof. We first treat the pullback formula at the prime $v_{\mathrm{aux}}$. By our computations the local $\beta$-th Fourier coefficient for $f_{sieg,v_{\mathrm{aux}}}$ is nonzero only when $\beta$ is an element in $\mathrm{GL}_{r+s+1}(\mathcal{O}_{F,v_{\mathrm{aux}}})$. From Proposition \ref{KudlaS} (note that the characters there are all unramified at $v_{\mathrm{aux}}$), we see that if we replace $f^{\mathrm{fteq}}_{v_{\mathrm{aux}}}$ by $f_{sieg,v_{\mathrm{aux}}}$, then the resulting Siegel Eisenstein series has the same Fourier coefficient for all $\beta$ as $F(f^{\mathrm{fteq}},-,z)$ (as forms parameterized by $z$). So we can use the replaced section to compute $F(f^{\mathrm{fteq}},-,z)$. Note also we have proved in Proposition \ref{notzero} that the pullback section of $f^{\mathrm{fteq}}_{v_{\mathrm{aux}}}$ is nonzero.

Next note that by comparing the global functional equation for Siegel and Klingen Eisenstein series, we see
$$\prod_{v}\gamma^{(1)}(\rho_v,z-\frac{1}{2})=\prod_{v}\gamma^{(2)}(\rho_v,z)=1.$$

We first claim that
$$\prod_{v|\infty}\gamma^{(1)}(\rho_v,z-\frac{1}{2})=C_\infty\prod_{v|\infty}\gamma^{(2)}(\rho_v,z)$$
where $C_\infty$ is a nonzero constant independent of $z$. (It seems likely one can prove this $C_\infty=1$ by more refined computations, but we do not need this.)
We prove it by a simple trick. We first take another character $\tau''$ with the same Archimedean type as $\tau$ and is ramified at all primes in $\Sigma$. Then we replace the $\tau$ in our Eisenstein datum by $\tau''$ and compute the pullback sections. As in \cite[Proposition 11.17]{SU}, we see that
$$\gamma^{(1)}(\rho_v,z-\frac{1}{2})=\gamma^{(2)}(\rho_v,z)$$
up to a multiplying by a nonzero constant independent of $z$, for all non-Archimedean primes $v$, which implies the claim, by noting that $\tau$ and $\tau''$ have the same Archimedean types. (The proof of \cite[Proposition 11.17]{SU} uses \cite[Lemma 11.10]{SU}, which needs our assumption that $\tau''$ is ramified at all primes in $\Sigma$. In our situation, although we allow the $v$ to be ramified in $\mathcal{K}$, the proof there still works. The double coset in the proof of \emph{loc.cit.} is valid with the $\ell$ in $K_{\mathbb{Q}_n}(\ell)$ there replaced by the uniformizer of $\mathcal{O}_{\mathcal{K},v}$. Note also the small error in \emph{loc.cit.} that the $(\ell^u)$ there should be the conductor of $\chi\chi^c$ instead of that of $\bar{\chi}^c$.)

We have from the formulas for unramified pullback sections,
$$\prod_{v\not\in\Sigma}\gamma^{(1)}(\rho_v,z-\frac{1}{2})=\prod_{v\not\in\Sigma}\gamma^{(2)}(\rho_v,z),$$
and also
$$\prod_{v|\infty}\gamma^{(1)}(\rho_v,z-\frac{1}{2})=\prod_{v|\infty}\gamma^{(2)}(\rho_v,z).$$
Then it follows that
$$\prod_{v\in\Sigma}\gamma^{(1)}(\rho_v,z-\frac{1}{2})=\prod_{v\in\Sigma}\gamma^{(2)}(\rho_v,z).$$
We find from Proposition \ref{formula intertwining} that the value of $F(f^{\mathrm{fteq}}, g, z)$ at $g=\prod_{v\nmid p}1_v\prod_{v|p}(ww_{Borel})_v$
is given by
$$\prod_{v\in\Sigma}\epsilon(\tilde{\pi}_v,\bar{\tau}^c_v,-z)\frac{L(\tilde{\pi},\bar{\tau}^c,z+1)}
{\prod_{v\in\Sigma}
L_v(\pi,\tau^c,-z)}\prod_{v\in\Sigma}\frac{L(\bar{\tau}'_v\chi^a_{\mathcal{K},v}, 2z+1)}{L(\tau'_v\chi^a_{\mathcal{K},v},-2z)} L^\Sigma(\xi,z+\frac{3}{2})\otimes_{v\in\Sigma\backslash\{v|p\}}
C_v\varphi_v\otimes_{v|\infty}C_v(z)\varphi_v|_{z=\frac{1}{2}}.$$
The $C_v$'s and $C_v(\frac{1}{2})$'s are nonzero from our previous local computations of the local pullback sections and their image under the intertwining operators $A(\rho,z,-)$ when applying Proposition \ref{formula intertwining}. The above expression is clearly nonzero if $\prod_{v\in\Sigma}
L_v(\pi,\tau^c,-z)$ does not have poles. If they do have poles, then we apply Lemma \ref{poles} and \ref{generates}, we can still conclude that $F(f^{\mathrm{fteq}}, g, \frac{1}{2})$ is not zero (these poles are cancelled by poles provided by Lemma \ref{poles}).

\end{proof}

We need only the following lemma to conclude that the specialization of our Klingen Eisenstein family to $\phi_0$ is nonzero.
\begin{lemma}
The $E^{p-\mathrm{adic}}_{\mathrm{Kling},\tilde{\mathcal{D}}_\phi}$ is the $p$-adic avatar of $E^{C^\infty}_{\mathrm{Kling},\tilde{\mathcal{D}}_\phi}$.
\end{lemma}
\begin{proof}
Look at the construction of $E^{C^\infty}_{\mathrm{Kling},\tilde{\mathcal{D}}_\phi}$ and $E^{p-\mathrm{adic}}_{\mathrm{Kling},\tilde{\mathcal{D}}_\phi}$. We restrict the Siegel Eisenstein series to $\mathrm{U}(r+1,s+1)\times\mathrm{U}(s,r)$, and decompose with respect to the restriction on $\mathrm{U}(s,r)$. Recall the classicality result that any ordinary $p$-adic automorphic form of weight $\underline{k}$ must be classical (holomorphic). So if we write out the restriction on $\mathrm{U}(s,r)$ with respect to irreducible automorphic representations, then for any term whose restriction to $\mathrm{U}(s,r)$ is not in the holomorphic discrete series of weight $\underline{k}$, its $p$-adic avatar must be killed by the ordinary projector on $\mathrm{U}(s,r)$. Now the lemma follows easily by applying appropriate Hecke operators.
\end{proof}
From the computation of Fourier-coefficient for the Siegel Eisenstein series and the pullback formula, we see $E^{C^\infty}_{\mathrm{Kling},\tilde{\mathbf{D}}^{(2)}_{\phi_0}}$ does not have a pole. Moreover its constant term has two terms: the pullback Klingen section and its image under the intertwining operator (see \cite[Lemma 9.2]{SU}). We have computed that Klingen section term is nonzero. By looking at the Archimedean component, we see that the constant term of $E^{C^\infty}_{\mathrm{Kling},\tilde{\mathbf{D}}^{(2)}_{\phi_0}}$, and thus $E^{C^\infty}_{\mathrm{Kling},\tilde{\mathbf{D}}^{(2)}_{\phi_0}}$ it self must be nonzero. We thus have the following Proposition from the above lemma.
\begin{proposition}\label{P-property}
The $E^{p-\mathrm{adic}}_{\mathrm{Kling},\mathbf{D}_{\phi_0}}$ is nonzero.
\end{proposition}

\section{Proof of Bloch-Kato Conjecture}
Now we prove the main theorem on Selmer groups. This is similar to previous work (e.g. \cite{SU}) on the ``lattice construction''. One difference is in the following Proposition \ref{noCAP}. \cite{SU} used a modularity lifting result to deduce that there is no CAP (i.e. cusp forms with the same Galois representation as Klingen Eisenstein series) at sufficiently regular weight, while modularity results for general unitary group seems require lots of assumptions. Here instead we use result of Shin \cite{Shin1} on the description of base change lift of cusp forms on unitary groups.

Let $K_\mathbf{D}$ be an open compact subgroup of $\mathrm{GU}(r+1,s+1)(\mathbb{A}_F)$ maximal at $p$ and all primes outside $\Sigma$ such that the Klingen-Eisenstein series we construct is invariant under $K^{(p)}_\mathbf{D}$. We consider the ring $\mathbb{T}_\mathbf{D}$ of reduced Hecke algebras acting on the space of ${\Lambda}''_\mathbf{D}$-adic nearly ordinary cuspidal forms with level group $K_\mathbf{D}$. It is generated by the Hecke operators at primes outside $\Sigma$, together with the $\mathrm{U}_p$-operator and then taking the maximal reduced quotient.

Suppose the Fourier-Jacobi coefficient $\mathrm{FJ}_{\beta,\theta,g}$ of $\phi_0(\mathrm{E}_{\mathrm{Kling}})$ is nonzero. This is possible by Proposition \ref{P-property} and the injectivity of the Fourier-Jacobi expansion map. We consider the $\mathbb{I}[[\Gamma^+_\mathcal{K}]]$-valued functional on the space of $\mathbb{I}[[\Gamma^+_\mathcal{K}]]$-valued forms on $\mathrm{U}(r+1,s+1)$ given by $\mathrm{FJ}_{\beta,\theta,g}$.
\begin{definition}
We define the ideal $I_\mathbf{D}$ of $\mathbb{T}_\mathbf{D}$ to be generated by $\{t-\lambda(t)\}_t$ for $t$'s in the abstract Hecke algebra and $\lambda(t)$ is the Hecke eigenvalue of $t$ on $\mathbf{E}_{\mathbf{D},Kling}$. Then it is easy to see that the structure map ${\Lambda}_\mathbf{D}\rightarrow \mathbb{T}_\mathbf{D}/I_\mathbf{D}$ is surjective. Suppose the inverse image of $I_\mathbf{D}$ in ${\Lambda}_\mathbf{D}$ is $\mathcal{E}_\mathbf{D}$. We call it the Eisenstein ideal. It measures the congruences between the Hecke eigenvalues of cusp forms and Klingen-Eisenstein series.
\end{definition}
Now we specify the $\chi$. Note that by (\ref{Klingen Galois}), the Selmer group constructed at the point $\phi_0$ is for the Galois representation $\rho_{\pi_{\mathbf{f}_{\phi_0}}}\otimes\chi^{-1}$ where $\chi$ is the $\bar{\boldsymbol{\tau}}^c_{\phi_0}\cdot\epsilon^{\frac{\kappa_{\phi_0}-r-s}{2}+\frac{\kappa_{\phi_0}}{2}}$.
\begin{lemma}\label{8.1}
Let $P$ be a height $1$ prime contained in the prime of $\mathbb{I}[[\Gamma^+_\mathcal{K}]]$ corresponding to $\phi_0$. Then
$$\mathrm{ord}_P(\mathcal{L}^\Sigma_{\chi'\chi^a_\mathcal{K},F}\mathcal{L}_{\mathbf{f},\xi,\mathcal{K}}^\Sigma)\leq \mathrm{ord}_P(\mathcal{E}_\mathbf{D}).$$
\end{lemma}
\begin{proof}
Suppose $t:=\mathrm{ord}_P(\mathcal{L}^\Sigma_{\chi'\chi^a_\mathcal{K},F}\mathcal{L}^{\Sigma}_{\mathbf{f},\mathcal{K},\xi})>0$. By the fundamental exact sequence Theorem \ref{fes} there is an $\mathbf{H}=\mathbf{E}_{\mathbf{D},Kling}-\mathcal{L}^\Sigma_{\chi'\chi^a_\mathcal{K},F}
\mathcal{L}^{\Sigma}_{\mathbf{f},\xi,\mathcal{K}}F$ for some $\Lambda_\mathbf{D}$-adic form $F$ such that $\mathbf{H}$ is a cuspidal family. Recall we have constructed a $\mathbb{I}[[\Gamma^+_\mathcal{K}]]$-valued functional $\ell:=\mathrm{FJ}_{\beta, \theta,g}$ on the space of $\Lambda_\mathbf{D}$-adic forms, which maps $\mathbf{H}$ to an element outside $P$, thanks to our assumption that $P$ is contained in the prime corresponding to $\phi_0$. By our assumption on $P$ we have proved that $\ell(\mathbf{H})\not\equiv 0(\mathrm{mod}P)$. Consider the $\Lambda_\mathbf{D}$-linear map:
$$\mu: \mathbb{T}_\mathbf{D}\rightarrow \Lambda''_{\mathbf{D},P}/P^r\Lambda_{\mathbf{D},P}$$
given by:
$\mu(t)=\ell(t.\mathbf{H})/\ell(\mathbf{H})$ for $t$ in the Hecke algebra. Then:
$$\ell(t.\mathbf{H})\equiv \ell(t\mathbf{E}_\mathbf{D})\equiv \lambda(t)\ell(\mathbf{E}_\mathbf{D})\equiv \lambda(t)\ell(\mathbf{H})(\mathrm{mod}P^t),$$
so $I_\mathbf{D}$ is contained in the kernel of $\mu$. Thus it induces:
$\Lambda_{\mathbf{D},P}/\mathcal{E}_\mathbf{D}\Lambda_{\mathbf{D},P}\twoheadrightarrow
\Lambda_{\mathbf{D},P}/P^t\Lambda_{\mathbf{D},P}$
which proves the lemma.
\end{proof}

We then state a result on lattice construction proved in \cite[Proposition 4.17]{SU}.
\begin{proposition}
Let $X^\Sigma_{\chi'\chi^a_\mathcal{K},F}$ be the dual Iwasawa Selmer group of the Hecke character $\chi'\chi^a_\mathcal{K}$ of $F^\times\backslash \mathbb{A}^\times_F$. This is a finitely generated module over $\mathcal{O}_L[[\Gamma^+_\mathcal{K}]]$, which we naturally regard as a module over $\Lambda_\mathbf{D}$. Suppose $P$ is a height one prime of $\mathbb{I}[[\Gamma^+_\mathcal{K}]]$ such that
$$\mathrm{ord}_P\mathrm{char}(X^\Sigma_{\chi'\chi^a_\mathcal{K},F})<\mathrm{ord}_P \mathcal{E}_\mathbf{D}.$$
Then we have
$$\mathrm{ord}_P\mathrm{char}(X^\Sigma_{\mathbf{f},\chi,\mathcal{K}})\geq 1.$$
\end{proposition}
In this case we do not exclude the possibility that there are common divisors between $\mathrm{char}(X^\Sigma_{\mathbf{f},\chi,\mathcal{K}})$ and $\mathrm{char}(X^\Sigma_{\chi'\chi^a_\mathcal{K},F})$, which causes complication in constructing elements of the Selmer groups. Nevertheless one can still prove weaker result that the order is positive, in the case when $P$ is a divisor of $\mathcal{L}^\Sigma_{\mathbf{f},\chi,\mathcal{K}}$. This is enough for our purpose.

We also need the following proposition, which is the analogue of \cite[Theorem 7.5]{SU}.
\begin{proposition}\label{noCAP}
Let $\mathfrak{m}_{\mathrm{Kling}}$ be the maximal ideal of the Hecke algebra corresponding to the Klingen Eisenstein family we construct. Let $\mathbf{J}$ be an irreducible component of $T_{\mathbf{D},\mathfrak{m}_{\mathrm{Kling}}}$. Let $R_\mathbf{J}$ be the corresponding semi-simple Galois representation defined over the total ring of fractions of $\mathbf{J}$. Then either (1) $\mathbf{R}_\mathbf{J}$ is irreducible, or (2) $\mathbf{R}_\mathbf{J}=R_1+R_2$ for $R_1$ has the same residual character as $\mathbf{f}$, and $R$ is a two dimensional irreducible representation.
\end{proposition}
\begin{proof}
The proof of \cite[Theorem 7.5]{SU} made use of the result of Harris on the non-existence of CAP forms on in the absolute convergent range of Klingen Eisenstein series, and modularity lifting results for $\mathrm{GL}_2$. However we argue differently since we do not have a satisfying modularity lifting results for general unitary groups. We first prove $\mathbf{R}_\mathbf{J}$ is not a sum of three irreducible representations. Otherwise suppose it is $R_1\oplus R_2\oplus R_3$ where $R_2$ and $R_3$ are one dimensional. We specialize to an arithmetic point $\phi$ which corresponds to regular discrete series at $\infty$, and apply \cite[Theorem A.1 (v)]{Shin1}. Suppose $R_{2,\phi}$ corresponds to one of the isobaric summands $\Pi_i$ in \emph{loc.cit.}. Let $k_{R_2,1}$ and $k_{R_2,2}$ be the Archimedean type of $R_{2,\phi}$. Then by the conjugate self-duality we have $k_{R_2,1}+k_{R_2,2}=0$. This gives a contradiction by considering the residual representation. So it has to be the case that the base change at $\phi$ is an isobaric sum $\Pi_1\oplus\Pi_2$, in which one of them (say $\Pi_2$) is two dimensional, or is just the $\Pi_1$ (only one summand). By by our assumption of regularity of weight, the $\Pi_2$ is cuspidal and tempered. Then the Galois representation of $\Pi_2$ cannot be a sum of two crystalline characters, as our previous consideration of residual representation and Archimedean weights, a contradiction. If the base change is just $\Pi_1$, then it is cuspidal and tempered by the regularity of weight. But as before the Galois representation cannot have some crystalline character as a summand.

The case when $\mathbf{R}_\mathbf{J}$ is $R_1\oplus R_2$ where $R_i$ are irreducible and $R_2$ is one dimensional can be excluded similarly.
\end{proof}

\begin{theorem}
Suppose $L_\mathcal{K}(\tilde{\pi}_f,\chi, \frac{1}{2})=0$, then the corank of the Selmer group for $\tilde{\rho}_\pi\otimes\chi^{-1}$ is positive.
\end{theorem}
\begin{proof}
It is easy to see that it is enough to prove it for $\Sigma$-primitive Selmer groups. By the Iwasawa main conjecture for Hecke characters proved by Wiles \cite{Wiles}, we see the characteristic ideal for $\mathrm{char}(X^\Sigma_{\chi'\chi^a_\mathcal{K},F})$ is bounded by the $p$-adic $L$-function $\mathcal{L}^\Sigma_{\chi'\chi^a_\mathcal{K}, F}$. By our assumption, there is a height one prime $P$ of $\mathbb{I}[[\Gamma^+_\mathcal{K}]]$ contained in the primes corresponding to $\phi_0$, such that the order of $\mathcal{L}^\Sigma_{\mathbf{f},\chi,\mathcal{K}}$ at $P$ is positive. By our discussion above, we see that
$$\mathrm{ord}_P\mathrm{char}(X^\Sigma_{\mathbf{f},\chi,\mathcal{K}})\geq 1.$$
Specializing to $\phi_0$ and applying the control theorem of Selmer groups \cite[Proposition 3.7, 3.10]{SU}, this implies the corank of the Selmer group at $\phi_0$ is positive, which proves the theorem.
\end{proof}

\end{document}